\documentclass[reqno]{amsart}

\usepackage{mathabx} 
\usepackage{graphicx}
\usepackage{amssymb,amsmath,amsthm,amscd}
\usepackage{mathrsfs}
\usepackage{enumerate}
\usepackage{enumitem}
\usepackage{verbatim}
\usepackage[usenames,dvipsnames]{color}
\usepackage[colorlinks=true, pdfstartview=FitV,linkcolor=blue,citecolor=blue,urlcolor=blue]{hyperref}
\usepackage[all]{xy}
\usepackage{tikz}
\usepackage{tikz-cd}
\usetikzlibrary{matrix, arrows}
\usepackage{stmaryrd} 
\usepackage{dsfont}
\usepackage{bm} 
\usepackage{array}
\usepackage{adjustbox}

\allowdisplaybreaks[4]

\allowdisplaybreaks[1]

\setlist[itemize,1]{leftmargin=.4in}
\setlist[enumerate,1]{leftmargin=.4in,label=(\alph*)} 
\setlist[description,1]{leftmargin=.4in,font=\normalfont\itshape}

\DeclareSymbolFont{sansops}{OT1}{\sfdefault}{m}{n}
\SetSymbolFont{sansops}{bold}{OT1}{\sfdefault}{b}{n}

\makeatletter
\renewcommand\operator@font{\mathgroup\symsansops}
\makeatother

\newcommand{\nc}{\newcommand}
\newcommand{\rnc}{\renewcommand}

\numberwithin{equation}{section}

\nc{\eqrefs}[2]{\text{(\ref{#1}-\ref{#2})}}

\usepackage[colorinlistoftodos]{todonotes}
\usepackage{mathtools}
\mathtoolsset{showmanualtags}
\mathtoolsset{showonlyrefs,showmanualtags}
\usepackage{empheq}

\theoremstyle{plain}
\newtheorem{lemma}{Lemma}[subsection]
\newtheorem{prop}[lemma]{Proposition}
\newtheorem{theorem}[lemma]{Theorem}

\newcommand{\Prop}{\begin{prop}}
	\newcommand{\enprop}{\end{prop}}
\newcommand{\Lemma}{\begin{lemma}}
	\newcommand{\enlemma}{\end{lemma}}
\newcommand{\Th}{\begin{theorem}}
	\newcommand{\enth}{\end{theorem}}
\newtheorem{corollary}[lemma]{Corollary}
\newcommand{\Cor}{\begin{corollary}}
	\newcommand{\encor}{\end{corollary}}
\newtheorem{definition}[lemma]{Definition}

\newcommand{\Def}{\begin{definition}}
	\newcommand{\edf}{\end{definition}}
\newtheorem{sublemma}[lemma]{Sublemma}
\newcommand{\Sublemma}{\begin{sublemma}}
	\newcommand{\ensub}{\end{sublemma}}

\theoremstyle{definition}
\newtheorem{remark}[lemma]{Remark}
\newtheorem{remarks}[lemma]{Remarks}
\newtheorem{example}[lemma]{Example}
\newtheorem{Convention}[lemma]{Convention}
\newcommand{\Conv}{\begin{Convention}}
	\newcommand{\enconv}{\end{Convention}}
\nc{\Rem}{\begin{remark}}
	\nc{\enrem}{\end{remark}}

\nc{\rmkend}{\hfill$\triangledown$}
\nc{\defend}{\hfill$\triangle$}

\nc{\be}{\begin{enumerate}}
	\nc{\ee}{\end{enumerate}}
\newcommand{\eq}{\begin{eqnarray}}
	\newcommand{\eneq}{\end{eqnarray}}
\nc{\bc}{\begin{cases}}
	\nc{\ec}{\end{cases}}
\newcommand{\eqn}{\begin{eqnarray*}}
	\newcommand{\eneqn}{\end{eqnarray*}}
\newcommand{\ba}{\begin{array}}
	\newcommand{\ea}{\end{array}}

\newcommand{\arxiv}[1]{\href{http://arxiv.org/abs/#1}{\tt arXiv:\nolinkurl{#1}}}

\nc{\cA}{{\mathcal A}}
\nc{\cB}{{\mathcal B}}
\nc{\cC}{{\mathcal C}}
\nc{\cD}{{\mathcal D}}
\nc{\cE}{{\mathcal E}}
\nc{\cF}{{\mathcal F}}
\nc{\cG}{{\mathcal G}}
\nc{\cH}{{\mathcal H}}
\nc{\cI}{{\mathcal I}}
\nc{\cJ}{{\mathcal J}}
\nc{\cK}{{\mathcal K}}
\nc{\cL}{{\mathcal L}}
\nc{\cM}{\mathcal{M}}
\nc{\cN}{{\mathcal N}}
\nc{\cO}{{\mathcal O}}
\nc{\cP}{{\mathcal P}}
\nc{\calQ}{{\mathcal Q}}
\nc{\cR}{{\mathcal R}}
\nc{\cS}{\mathcal{S}}
\nc{\cT}{{\mathcal T}}
\nc{\cU}{\mathcal{U}}
\nc{\cV}{{\mathcal V}}
\nc{\cX}{{\mathcal X}}
\nc{\cY}{\mathcal{Y}}
\nc{\cW}{\mathcal{W}}
\nc{\cZ}{{\mathcal Z}}

\nc{\bbA}{{\mathbb{A}}}
\nc{\bbB}{{\mathbb{B}}}
\nc{\bbC}{{\mathbb{C}}}
\nc{\bbD}{{\mathbb{D}}}
\nc{\bbE}{{\mathbb{E}}}
\nc{\bbF}{{\mathbb{F}}}
\nc{\bbG}{{\mathbb{G}}}
\nc{\bbH}{{\mathbb{H}}}
\nc{\bbI}{{\mathbb{I}}}
\nc{\bbJ}{{\mathbb{J}}}
\nc{\bbK}{{\mathbb{K}}}
\nc{\bbL}{{\mathbb{L}}}
\nc{\bbM}{{\mathbb{M}}}
\nc{\bbN}{{\mathbb{N}}}
\nc{\bbO}{{\mathbb{O}}}
\nc{\bbP}{{\mathbb{P}}}
\nc{\bbQ}{{\mathbb{Q}}}
\nc{\bbR}{{\mathbb{R}}}
\nc{\bbS}{{\mathbb{S}}}
\nc{\bbT}{{\mathbb{T}}}
\nc{\bbU}{{\mathbb{U}}}
\nc{\bbV}{{\mathbb{V}}}
\nc{\bbX}{{\mathbb{X}}}
\nc{\bbY}{{\mathbb{Y}}}
\nc{\bbW}{{\mathbb{W}}}
\nc{\bbZ}{{\mathbb{Z}}}

\usepackage{mathrsfs}

\nc{\scrA}{{\mathscr A}}
\nc{\scrB}{{\mathscr B}}
\nc{\scrC}{{\mathscr C}}
\nc{\scrD}{{\mathscr D}}
\nc{\scrE}{{\mathscr E}}
\nc{\scrF}{{\mathscr F}}
\nc{\scrG}{{\mathscr G}}
\nc{\scrH}{{\mathscr H}}
\nc{\scrI}{{\mathscr I}}
\nc{\scrJ}{{\mathscr J}}
\nc{\scrK}{{\mathscr K}}
\nc{\scrL}{{\mathscr L}}
\nc{\scrM}{{\mathscr M}}
\nc{\scrN}{{\mathscr N}}
\nc{\scrO}{{\mathscr O}}
\nc{\scrP}{{\mathscr P}}
\nc{\scrQ}{{\mathscr Q}}
\nc{\scrR}{{\mathscr R}}
\nc{\scrS}{{\mathscr S}}
\nc{\scrT}{{\mathscr T}}
\nc{\scrU}{{\mathscr U}}
\nc{\scrV}{{\mathscr V}}
\nc{\scrX}{{\mathscr X}}
\nc{\scrY}{{\mathscr Y}}
\nc{\scrW}{{\mathscr W}}
\nc{\scrZ}{{\mathscr Z}}

\nc{\sfA}{{\mathsf A}}
\nc{\sfB}{{\mathsf B}}
\nc{\sfC}{{\mathsf C}}
\nc{\sfD}{{\mathsf D}}
\nc{\sfE}{{\mathsf E}}
\nc{\sfF}{{\mathsf F}}
\nc{\sfG}{{\mathsf G}}
\nc{\sfH}{{\mathsf H}}
\nc{\sfI}{{\mathsf I}}
\nc{\sfJ}{{\mathsf J}}
\nc{\sfK}{{\mathsf K}}
\nc{\sfL}{{\mathsf L}}
\nc{\sfM}{{\mathsf M}}
\nc{\sfN}{{\mathsf N}}
\nc{\sfO}{{\mathsf O}}
\nc{\sfP}{{\mathsf P}}
\nc{\sfQ}{{\mathsf Q}}
\nc{\sfR}{{\mathsf R}}
\nc{\sfS}{{\mathsf S}}
\nc{\sfT}{{\mathsf T}}
\nc{\sfU}{{\mathsf U}}
\nc{\sfV}{{\mathsf V}}
\nc{\sfX}{{\mathsf X}}
\nc{\sfY}{{\mathsf Y}}
\nc{\sfW}{{\mathsf W}}
\nc{\sfZ}{{\mathsf Z}}

\nc{\sfa}{{\mathsf a}}
\nc{\sfb}{{\mathsf b}}
\nc{\sfc}{{\mathsf c}}
\nc{\sfd}{{\mathsf d}}
\nc{\sfe}{{\mathsf e}}
\nc{\sff}{{\mathsf f}}
\nc{\sfg}{{\mathsf g}}
\nc{\sfh}{{\mathsf h}}
\nc{\sfi}{{\mathsf i}}
\nc{\sfj}{{\mathsf j}}
\nc{\sfk}{{\mathsf k}}
\nc{\sfl}{{\mathsf l}}
\nc{\sfm}{{\mathsf m}}
\nc{\sfn}{{\mathsf n}}
\nc{\sfo}{{\mathsf o}}
\nc{\sfp}{{\mathsf p}}
\nc{\sfq}{{\mathsf q}}
\nc{\sfr}{{\mathsf r}}
\nc{\sfs}{{\mathsf s}}
\nc{\sft}{{\mathsf t}}
\nc{\sfu}{{\mathsf u}}
\nc{\sfv}{{\mathsf v}}
\nc{\sfx}{{\mathsf x}}
\nc{\sfy}{{\mathsf y}}
\nc{\sfw}{{\mathsf w}}
\nc{\sfz}{{\mathsf z}}

\nc {\bfA}{{\mathbf A}}
\nc {\bfB}{{\mathbf B}}
\nc {\bfC}{{\mathbf C}}
\nc {\bfD}{{\mathbf D}}
\nc {\bfE}{{\mathbf E}}
\nc {\bfF}{{\mathbf F}}
\nc {\bfG}{{\mathbf G}}
\nc {\bfH}{{\mathbf H}}
\nc {\bfI}{{\mathbf I}}
\nc {\bfJ}{{\mathbf J}}
\nc {\bfK}{{\mathbf K}}
\nc {\bfL}{{\mathbf L}}
\nc {\bfM}{{\mathbf M}}
\nc {\bfN}{{\mathbf N}}
\nc{\bfO}{{\mathbf O}}
\nc {\bfP}{{\mathbf P}}
\nc {\bfQ}{{\mathbf Q}}
\nc {\bfR}{{\mathbf R}}
\nc {\bfS}{{\mathbf S}}
\nc {\bfT}{{\mathbf T}}
\nc {\bfU}{{\mathbf U}}
\nc {\bfV}{{\mathbf V}}
\nc {\bfX}{{\mathbf X}}
\nc {\bfY}{{\mathbf Y}}
\nc {\bfW}{{\mathbf W}}
\nc {\bfZ}{{\mathbf Z}}

\nc {\fka}{{\mathfrak a}}
\nc {\fkb}{{\mathfrak b}}
\nc {\fkc}{{\mathfrak c}}
\nc {\fkd}{{\mathfrak d}}
\nc {\fke}{{\mathfrak e}}
\nc {\fkf}{{\mathfrak f}}
\nc {\fkg}{{\mathfrak g}}
\nc {\fkh}{{\mathfrak h}}
\nc {\fki}{{\mathfrak i}}
\nc {\fkj}{{\mathfrak j}}
\nc {\fkk}{{\mathfrak k}}
\nc {\fkl}{{\mathfrak l}}
\nc {\fkm}{{\mathfrak m}}
\nc {\fkn}{{\mathfrak n}}
\nc {\fko}{{\mathfrak o}}
\nc {\fkp}{{\mathfrak p}}
\nc {\fkq}{{\mathfrak q}}
\nc {\fkr}{{\mathfrak r}}
\nc {\fks}{{\mathfrak s}}
\nc {\fkt}{{\mathfrak t}}
\nc {\fku}{{\mathfrak u}}
\nc {\fkv}{{\mathfrak v}}
\nc {\fkx}{{\mathfrak x}}
\nc {\fky}{{\mathfrak y}}
\nc {\fkw}{{\mathfrak w}}
\nc {\fkz}{{\mathfrak z}}

\nc {\fkA}{{\mathfrak A}}
\nc {\fkB}{{\mathfrak B}}
\nc {\fkC}{{\mathfrak C}}
\nc {\fkD}{{\mathfrak D}}
\nc {\fkE}{{\mathfrak E}}
\nc {\fkF}{{\mathfrak F}}
\nc {\fkG}{{\mathfrak G}}
\nc {\fkH}{{\mathfrak H}}
\nc {\fkI}{{\mathfrak I}}
\nc {\fkJ}{{\mathfrak J}}
\nc {\fkK}{{\mathfrak K}}
\nc {\fkL}{{\mathfrak L}}
\nc {\fkM}{{\mathfrak M}}
\nc {\fkN}{{\mathfrak N}}
\nc {\fkO}{{\mathfrak O}}
\nc {\fkP}{{\mathfrak P}}
\nc {\fkQ}{{\mathfrak Q}}
\nc {\fkR}{{\mathfrak R}}
\nc {\fkS}{{\mathfrak S}}
\nc {\fkT}{{\mathfrak T}}
\nc {\fkU}{{\mathfrak U}}
\nc {\fkV}{{\mathfrak V}}
\nc {\fkX}{{\mathfrak X}}
\nc {\fkY}{{\mathfrak Y}}
\nc {\fkW}{{\mathfrak W}}
\nc {\fkZ}{{\mathfrak Z}}

\rnc{\a}{\fka}
\rnc{\b}{\fkb}
\rnc{\c}{\fkc}
\rnc{\d}{\fkd}
\nc{\e}{\fke}
\nc{\f}{\fkf}
\nc{\h}{\fkh}
\rnc{\i}{\fki}
\rnc{\j}{\fkj}
\rnc{\k}{\fkk}
\rnc{\l}{\fkl}
\nc{\m}{\fkm}
\nc{\n}{\fkn}
\rnc{\o}{\fko}
\nc{\p}{\fkp}
\nc{\q}{\fkq}
\rnc{\r}{\fkr}
\nc{\s}{\fks}
\rnc{\t}{\fkt}
\rnc{\u}{\fku}
\rnc{\v}{\fkv}
\nc{\x}{\fkx}
\nc{\y}{\fky}
\nc{\w}{\fkw}
\nc{\z}{\fkz}

\newcommand{\C}{{\mathbb C}}
\newcommand{\Q}{\mathbb {Q}}
\newcommand{\Z}{{\mathbb Z}}

\newcommand{\g}{{\mathfrak{g}}}
\nc{\Ad}{\operatorname{Ad}}
\nc{\ad}{\operatorname{ad}}
\nc{\sym}{\mathfrak{S}} 
\nc{\weyl}{\mathfrak{W}}
\nc{\Serre}{\operatorname{Serre}}

\newcommand{\Hom}{\operatorname{Hom}}
\newcommand{\End}{\operatorname{End}}
\nc{\Aut}{\operatorname{Aut}}

\nc{\coker}{{\operatorname{coker}}}
\newcommand{\Ker}{{\operatorname{Ker}}}
\newcommand{\id}{{\operatorname{id}}}
\nc{\Id}{\operatorname{Id}}
\nc{\aff}{{\sf aff}}
\nc{\Sp}{\operatorname{Sp}}

\nc{\rank}{{\operatorname{rank}}}
\nc{\Rep}{{\operatorname{Rep}}}
\nc{\Repfd}{{\operatorname{Rep}_{\sf fd}}}

\nc{\Mod}{{\operatorname{Mod}}}
\nc{\Modfd}{{\operatorname{Mod}_{\sf fd}}}

\nc {\ul}{\underline}
\nc {\ol}{\overline}
\nc {\wtil}{\widetilde}
\nc {\wh}{\widehat}
\nc {\wb}{\widebar}
\nc{\scs}{\scriptscriptstyle}
\nc{\scsop}{\scriptscriptstyle\operatorname} 

\nc {\ie}{\emph{i.e.}, } 
\nc {\eg}{\emph{e.g.}, } 
\nc {\cf}{{\emph{cf.}} } 
\nc {\loccit}{{\emph{loc. cit. }} } 

\nc {\aand}{\qquad\mbox{and}\qquad}

\nc{\al}{\alpha}
\nc{\ga}{\gamma}
\nc{\del}{\delta}
\nc{\eps}{\epsilon}
\nc{\ze}{\zeta}
\nc{\ka}{\kappa}
\nc{\la}{\lambda}
\nc{\si}{\sigma}
\nc{\om}{\omega}
\nc{\ups}{\upsilon}

\nc{\Del}{\Delta}

\nc{\ext}{\mathsf{ext}}

\nc{\fksl}{\mathfrak{sl}}
\nc{\fkso}{\mathfrak{so}}
\nc{\fksp}{\mathfrak{sp}}

\nc{\qu}{\quad}
\nc{\qq}{\qquad}

\rnc{\eq}[1]{\begin{equation} #1 \end{equation}}

\nc{\drc}[1]{\delta_{#1}}
\nc{\der}{\partial}
\nc{\sfad}{\mathsf{ad}}
\nc{\ten}{\otimes}

\nc{\ourcomment}[1]{}
\nc{\andreacomment}[1]{}
\nc{\bartcomment}[1]{}
\nc{\Omit}[1]{}
\nc{\summary}[1]{}
\nc{\tm}[1]{{\color{magenta}{#1}}}

\nc{\bsF}{\bbF} 

\nc{\Ons}{\mathbf{O}_q}
\nc{\aOns}{\mathbf{O}_q^{\mathsf{a}}}
\nc{\bOns}{\mathbf{O}_q^{\mathsf{b}}}

\nc{\km}{\mathbf{k}}
\nc{\qkm}{\mathfrak{k}}

\nc{\texp}[1]{\wt{s}_{#1}} 
\nc{\Lus}[1]{T_{#1}} 
\nc{\mLus}[1]{\mathbf{T}_{#1}} 

\nc{\fin}{{0}}

\nc{\Ifin}{{I_\fin}}
\nc{\Afin}{{A_\fin}}

\rnc{\aff}[1]{\wh{#1}}

\nc{\de}[1]{\epsilon_{#1}} 

\nc{\fIS}{I} 
\nc{\aIS}{\aff{I}} 
\nc{\IS}{I}

\nc{\veps}{\varepsilon}
\nc{\vtheta}{\vartheta}

\nc{\rt}[1]{\alpha_{#1}} 
\nc{\cort}[1]{h_{#1}}
\nc{\fwt}[1]{\varpi_{#1}}
\nc{\fcwt}[1]{\Lambda_{#1}^\vee}

\nc{\rootsys}{{\Phi}} 

\nc{\hrt}{\vartheta} 

\nc{\cp}[2]{#1(#2)}
\nc{\iip}[2]{\left( #1 , #2\right)} 

\nc{\drv}[1]{\delta_{#1}} 
\nc{\codrv}[1]{d_{#1}} 

\nc{\bsfld}{\mathbb{K}} 

\nc{\central}[1]{c_{#1}}

\nc{\Qlat}{\mathsf{Q}}
\nc{\Qpm}{\Qlat^\pm}
\nc{\Qp}{\Qlat^+}
\nc{\Qm}{\Qlat^-}
\nc{\Qlatv}{\Qlat^\vee}
\nc{\Qvpm}{\Qlat^{\vee,\pm}}
\nc{\Qvp}{\Qlat^{\vee,+}}
\nc{\Qvm}{\Qlat^{\vee,-}}
\nc{\Qextp}{\Qlat^+_{\sf ext}}
\nc{\Qvext}{\Qlat^{\vee}_{\sf ext}}
\nc{\Qvextt}{\Qlat^{\vee}_{{\sf ext},\tau}}
\nc{\Qvextp}{\Qlat^{\vee,+}_{\sf ext}}
\nc{\aQ}{\aff{\Qlat}}
\nc{\aQv}{\aff{\Qlat}^{\vee}}
\nc{\aQp}{\aff{\Qlat}^+}
\nc{\aQextp}{\aff{\Qlat}^+_{\sf ext}}
\nc{\aQvp}{\aff{\Qlat}^{\vee,+}}
\nc{\aQvext}{\wt{\Qlat}^{\vee}} 
\nc{\aQvextt}{\aff{\Qlat}^{\vee}_{{\sf ext},\tau}}
\nc{\aQvextp}{\aff{\Qlat}^{\vee,+}_{\sf ext}}
\nc{\Plat}{\mathsf{P}}
\nc{\Platv}{\mathsf{P}^\vee}
\nc{\Ppm}{\mathsf{P}^\pm}
\nc{\Pp}{\mathsf{P}^+}
\nc{\Pm}{\mathsf{P}^-}
\nc{\Pvpm}{\mathsf{P}^{\vee,\pm}}
\nc{\Pvp}{\mathsf{P}^{\vee,+}}
\nc{\Pvm}{\mathsf{P}^{\vee,-}}
\nc{\aP}{\wh{\Plat}}
\nc{\Pext}{\Plat}
\nc{\Pextpm}{\Plat_{\sf ext}^{\pm}}
\nc{\aPext}{\wh{\Plat}_{{\sf ext}}}
\nc{\aPpm}{\aP^{\pm}}
\nc{\aPextpm}{\aP^{\pm}_{\sf ext}}
\nc{\aPd}{{\aP/\bbZ\drv{}}} 

\nc{\PZ}{{\Plat_{\Z}}}
\nc{\PvZ}{{\Platv_{\Z}}}

\nc{\cl}[1]{\operatorname{cl}(#1)}

\nc{\Pcl}{\Plat_{\cl}}
\nc{\Pvcl}{\Plat^{\vee}_{\cl}}

\nc{\Pclz}{\Plat_{\cl,0}}

\nc{\fr}{_\mathsf{fr}}
\nc{\Qfr}{\Qlat\fr}
\nc{\Pfr}{\Plat\fr}

\nc{\Qvfr}{\Qlatv\fr}
\nc{\Pvfr}{\Platv\fr}

\nc{\wgt}[1]{\operatorname{wt}(#1)}

\nc{\hp}{\h'}
\nc{\gp}{\g'}
\nc{\kp}{\k'}

\nc{\Wfin}{{W_{\fin}}}
\nc{\Waff}{W}
\nc{\Wext}{\widetilde{W}}

\nc{\rfl}[1]{{s_{#1}}}

\nc{\UqgKM}{U_q(\g_{\scsop{KM}})}
\nc{\gKM}{\g_{\scsop{KM}}}

\nc{\Uqg}{U_q(\g)} 
\nc{\Uqgp}{U_q(\g')} 

\nc{\Uqh}{U_q(\h)} 
\nc{\Uqhp}{U_q(\h')} 

\nc{\Uqb}{U_q{(\b)}} 
\nc{\Uqbp}{U_q{(\b^+)}} 
\nc{\Uqbm}{U_q{(\b^-)}} 
\nc{\Uqbpm}{U_q{({\b}^{\pm})}} 
\nc{\Uqn}{U_q{(\n)}} 
\nc{\Uqnm}{U_q{({\n}^-)}} 
\nc{\Uqnp}{U_q{({\n}^+)}} 
\nc{\Uqnpm}{U_q{({\n}^{\pm})}} 

\nc{\Uqag}{U_q(\wt{\g})} 
\nc{\Uqagp}{U_q(\wh{\g})} 

\nc{\Uqan}{U_q(\wh{\n})} 
\nc{\Uqanp}{U_q(\wh{\n}^+)} 
\nc{\Uqanm}{U_q(\wh{\n}^-)} 
\nc{\Uqanpm}{U_q(\wh{\n}^{\pm})} 

\nc{\Uqah}{U_q(\wt{\h})} 
\nc{\Uqahp}{U_q(\wh{\h})} 

\nc{\Uqab}{U_q(\wt{\b})} 
\nc{\Uqabp}{U_q(\wt{\b}^+)} 
\nc{\Uqabm}{U_q(\wt{\b}^-)} 
\nc{\Uqabpm}{U_q(\wt{\b}^{\pm})} 
\nc{\Uqabpp}{U_q(\wh{\b}^+)} 

\nc{\UqgX}{U_q(\g_X)} 
\nc{\UqgZ}{U_q(\g_Z)}
\nc{\UqbpX}{U_q(\b_X^+)}
\nc{\UqnmX}{U_q(\n_X^-)}
\nc{\Uqht}{U_q(\h^{\theta})}

\nc{\CUqag}[1]{\Uqag^{#1}^{\cO^{\sf int}}}

\nc{\CUqg}[1]{\Uqg^{#1}^{\cO^{\sf int}}}

\nc{\Lg}{L\g} 
\nc{\ag}{\wt{\g}} 
\nc{\agp}{\wh{\g}} 
\nc{\ah}{\wt{\h}} 
\nc{\ahp}{\wh{\h}} 

\nc{\UqLg}{U_q(\Lg)} 

\nc{\CUqLg}[1]{(\UqLg^{#1})^{\sf fd}} 

\nc{\Kg}[1]{K_{#1}}
\nc{\Kgpm}[1]{K_{#1}^{\pm}}
\nc{\Kgp}[1]{K_{#1}^+}
\nc{\Kgm}[1]{K_{#1}^-}
\nc{\Eg}[1]{E_{#1}}
\nc{\Fg}[1]{F_{#1}}
\nc{\egp}[1]{e^{+}_{#1}}
\nc{\egm}[1]{e^{-}_{#1}}
\nc{\egpm}[1]{e^{\pm}_{#1}}
\nc{\egmp}[1]{e^{\mp}_{#1}}

\nc{\Ce}{\cC}

\nc{\xpm}[1]{x^{\pm}_{#1}}
\nc{\xmp}[1]{x^{\mp}_{#1}}
\nc{\xp}[1]{x^{+}_{#1}}
\nc{\xm}[1]{x^{-}_{#1}}
\nc{\xz}[1]{\xi_{#1}}

\nc{\phipm}[1]{\phi^{\pm}_{#1}}
\nc{\phip}[1]{\phi^{+}_{#1}}
\nc{\phim}[1]{\phi^{-}_{#1}}

\nc{\chev}{\omega}

\nc{\weightspace}[2]{{{#1}_{#2}}}
\nc{\wsp}[2]{\weightspace{#1}{#2}}
\nc{\wts}[1]{\mathsf{wt}(#1)}
\nc{\hwL}[1]{L(#1)}
\nc{\catO}[3]{\O_{#1}^{#2}{#3}}

\nc{\qstr}[2]{\Sigma_{#1,#2}}

\nc{\evrep}[2]{V_{#1}(#2)}
\nc{\shrep}[2]{{#1}(#2)}
\nc{\Lshrep}[2]{\Lfml{#1}{#2}}
\nc{\Pshrep}[2]{\Pfml{#1}{#2}}

\nc{\qstrep}[2]{V(\qstr{#1}{#2})}
\nc{\qsrep}[1]{V(#1)}

\nc{\HL}[1]{\mathcal{C}_{#1}}

\nc{\Oint}{{\cO^{\sf int}}}
\nc{\Ointp}{\cO_{\sf int}^{+}}
\nc{\Ointm}{\cO_{\sf int}^{-}}
\nc{\Ointpm}{\cO_{\sf int}^{\pm}}

\nc{\shift}[1]{\Sigma_{#1}}
\nc{\tshift}[1]{\Sigma^{\tau}_{#1}}
\nc{\shiftm}[2]{{#1}_{#2}}

\nc{\shifta}[1]{{\chi}_{#1}}

\nc{\Deltaop}{\Delta^{\sf op}}

\nc{\fml}[2]{{#1}[\negthinspace[#2]\negthinspace]} 
\nc{\Lfml}[2]{{#1}(\negthinspace(#2)\negthinspace)} 
\nc{\Pfml}[2]{{#1}\{#2\}} 

\nc{\qsl}[1]{U_q({\mathfrak{sl}}_{#1})}
\nc{\qasl}[1]{U_q(\wh{\mathfrak{sl}}_{#1})}
\nc{\qlsl}[1]{U_q(L\mathfrak{sl}_{#1})}

\nc{\brS}[1]{S_{#1}} 
\nc{\brSg}[1]{\widetilde{s}_{#1}} 

\nc{\Br}[1]{\mathscr{B}_{#1}} 
\nc{\RBr}[1]{\mathscr{RB}_{#1}} 

\nc{\qWS}[1]{S_{#1}}
\nc{\LT}[1]{T_{#1}}

\nc{\tcorr}[1]{\ka_{#1}} 
\nc{\bt}[1]{\mathcal{S}_{#1}} 

\nc{\adt}[1]{\mathcal{T}_{#1}} 
\nc{\adbt}[1]{\mathcal{T}_{#1}} 

\nc{\Rcorr}[1]{\eta_{#1}}

\nc{\intg}{\mathbf{W}^{\mathsf{int}}}
\nc{\Vect}[1]{\operatorname{Vect}_{#1}}

\nc{\corank}{\operatorname{cork}}

\nc{\gsat}[1]{\mathsf{GSat}(#1)} 
\nc{\auxgsat}[2]{\mathsf{Sat}(#1;#2)} 
\nc{\sat}[1]{{\mathbf{S}}} 
\nc{\tsat}{\theta} 
\nc{\zsat}{\zeta} 
\nc{\tsatq}{\tsat_{q}} 
\nc{\zsatq}{\zsat_{q}} 

\nc{\tinv}[1]{\theta_{#1}}

\nc{\oi}{\mathsf{oi}} 

\nc{\Parsetc}{\bm{\Gamma}} 
\nc{\Parsets}{\bm{\Sigma}}

\nc{\Parc}{\bm{\gamma}} 
\nc{\Pars}{\bm{\sigma}}

\nc{\parc}[1]{\bm{\gamma}_{#1}} 
\nc{\pars}[1]{\bm{\sigma}_{#1}}

\nc{\ctheta}{\theta} 

\nc{\qtheta}{\theta_q} 
\nc{\qthetat}{\widetilde{\theta}_q} 

\nc{\Uqk}{U_q(\mathfrak{k})}
\nc{\Uqkp}{U_q({\mathfrak{k}'})}

\nc{\wt}{\widetilde}

\nc{\Ieq}{I_{\sf eq}} 
\nc{\Idiff}{I_{\sf diff}} 
\nc{\Ins}{I_{\sf ns}} 
\nc{\aIeq}{\aIS_{\sf eq}} 
\nc{\aIdiff}{\aIS_{\sf diff}} 
\nc{\aIns}{\aIS_{\sf ns}} 

\nc{\bg}[1]{b_{#1}} 
\nc{\Bg}[1]{B_{#1}} 

\nc{\QK}[1]{\Upsilon_{#1}} 
\nc{\RM}[1]{{R}_{#1}} 
\nc{\QR}[1]{{\Xi}_{#1}} 
\nc{\sRM}[2]{{{R}}_{#1}(#2)} 
\nc{\sRMv}[2]{{\widecheck{R}}_{#1}(#2)} 
\nc{\rRM}[2]{R^{\operatorname{trig}}_{#1}(#2)} 
\nc{\rRMC}[3]{{\mathbf{R}}^{#1}_{#2}(#3)} 
\nc{\rRMv}[2]{\widecheck{\mathbf{R}}_{#1}(#2)} 
\nc{\rRMt}[2]{\widetilde{\mathbf{R}}_{#1}(#2)} 
\nc{\KM}[1]{K_{#1}} 
\nc{\sKM}[2]{K_{#1}(#2)} 
\nc{\rKM}[2]{K_{#1}^{\operatorname{trig}}(#2)} 
\nc{\trKM}[2]{\wt{\mathbf{K}}_{#1}(#2)} 
\nc{\TKM}[1]{{\bbK}_{#1}} 

\nc{\zeroKM}[2]{K^0_{#1}(#2)} 
\nc{\inftyKM}[2]{K^{\infty}_{#1}(#2)} 

\nc{\tsRM}[2]{{{R}}^{\tau}_{#1}(#2)} 
\nc{\tsRMv}[2]{{\widecheck{R}}^{\tau}_{#1}(#2)} 

\nc{\auxsat}[1]{\bfT} 

\nc{\hext}[2]{{#1}[\negthinspace[#2]\negthinspace]} 

\nc{\binomb}[2]{\genfrac{[}{]}{0pt}{}{#1}{#2}}

\nc{\GQSP}{\cG_{\tsat,\parc{}}}

\nc{\Supp}{\mathsf{Supp}}
\nc{\twistrep}[2]{{#1}^{#2}}  

\nc{\wtbeta}{\wt{\beta}}

\nc{\Gg}{\cG} 
\nc{\gau}{{\bf g}} 

\nc{\op}{\operatorname{op}}
\nc{\cop}{\operatorname{cop}}

\nc{\sclQK}{\mathbf{Y}} 

\nc{\tproj}[1]{[#1]_\tsat}

\nc{\hgt}{{\mathsf{ht}}}

\nc{\Vpsi}{V^{\psi}}
\nc{\Wpsi}{W^{\psi}}
\nc{\VWpsi}{(V\ten W)^{\psi}}

\nc{\sint}{{\operatorname{int}}}
\nc{\pint}[1]{{{#1}\negthinspace\operatorname{-int}}}
\nc{\Xint}{{{{X}\negthinspace\operatorname{-int}}}}
\nc{\WUqg}{\cW} 
\nc{\OUqg}{\cO_\infty} 
\nc{\OintUqg}{\cO_\infty^{\sint}} 
\nc{\POUqg}[1]{\cO_\infty^{\pint{#1}}} 
\nc{\vect}{\operatorname{Vect}}

\nc{\FF}[2]{\sff_{#1}^{#2}}

\nc{\CO}{\mathscr{E}} 
\nc{\COUqg}[2]{\End(\FF{#1}{#2})}
\nc{\COintUqg}[1]{\End(\FF{\sint}{#1})}
\nc{\CPOUqg}[2]{\End(\FF{\pint{#1}}{#2})} 
\nc{\CWUqg}[1]{\End(\FF{\cW}{#1})}
\nc{\WO}[1]{\operatorname{C}^{#1}}
\nc{\WUqk}{\cW_{\theta}}
\nc{\CWUqk}[1]{\End(\sff_\theta^{#1})}
\nc{\Dr}{\cD}
\nc{\Drt}{\cE_{\theta}}
\nc{\DrX}{\cE^{\op}_X}
\nc{\TQK}[1]{\Xi_{#1}}

\nc{\PROPB}{\mathsf{PROP/B}}
\nc{\PROB}{\mathsf{PROB}}
\nc{\PROP}{\mathsf{PROP}}
\nc{\bfb}{\mathbf{b}}
\nc{\G}{\cG}
\nc{\Fun}{\operatorname{Fun}}

\nc{\monactp}{\vartriangleleft}
\nc{\monactm}{\blacktriangleleft}

\rnc{\UqLg}{U_q(\mathfrak{L})} 
\nc{\gfin}{\mathring{\g}}
\nc{\hfin}{\mathring{\h}}

\nc{\fdUqL}{\cC}
\nc{\fdUqk}{\cC_{\theta}}

\nc{\rnRM}[2]{R^{\operatorname{norm}}_{#1}(#2)}
\nc{\rTKM}[2]{\mathbb{K}^{\operatorname{trig}}_{#1}(#2)}
\nc{\rnTKM}[2]{\mathbb{K}^{\operatorname{norm}}_{#1}(#2)}
\nc{\rnKM}[2]{K^{\operatorname{norm}}_{#1}(#2)}

\nc{\bmb}{{\bm \beta}}

\nc{\Ct}{\mathbf{t}}
\nc{\Cb}{\mathbf{b}}
\nc{\FSeq}{\operatorname{Seq}}
\nc{\BSeq}{\operatorname{BSeq}}

\AtBeginDocument{%
   \def\MR#1{}
}

\title[Tensor K-matrices for quantum symmetric pairs]{Tensor K-matrices for quantum symmetric pairs}

\author[A. Appel]{Andrea Appel} 
\address{Dipartimento di Scienze Matematiche, 
	Fisiche e Informatiche, Universit\`a di Parma, 
	INdAM - GNSAGA,
	and INFN - Gruppo Collegato di Parma, 
	Parco Area delle Scienze 53/A, 
	43124 Parma (PR), Italy}
\email{\href{mailto:andrea.appel@unipr.it}{andrea.appel@unipr.it}}

\author[B. Vlaar]{Bart Vlaar}
\address{
Beijing Institute of Mathematical Sciences and Applications, No.\ 544, Hefangkou Village, Huaibei Town, Huairou, Beijing, China}
\email{\href{b.vlaar@bimsa.cn}{b.vlaar@bimsa.cn}}

\thanks{The first author is partially supported by the Programme {FIL 2022} of the University of Parma and co-sponsored by Fondazione Cariparma, and by an INdAM Project Grant 2024.
Both authors are partially supported by the International Scientists Program of the Beijing Natural Sciences Foundation (grant number {IS24003}).}



\begin{document}
	

\begin{abstract}
Let $\g$ be a symmetrizable Kac-Moody algebra, $\Uqg$ its quantum group, and $\Uqk\subset\Uqg$ a quantum symmetric pair subalgebra determined by a Lie algebra automorphism $\theta$. 
We introduce a category $\WUqk$ of \emph{weight} $\Uqk$-modules, which is acted on by the category of weight $\Uqg$-modules via tensor products. 
We construct a universal tensor K-matrix $\TKM{}$ (that is, a solution of a reflection equation) in a completion of $\Uqk \ten \Uqg$.
This yields a natural operator on any tensor product $M\ten V$, where $M\in\WUqk$ and $V\in\cO_\theta$, \ie $V$ is a $\Uqg$-module in category $\cO$ satisfying an integrability property determined by $\theta$. 
Canonically, $\WUqk$ is equipped with a structure of a bimodule category over $\cO_\theta$ and the action of $\TKM{}$ is encoded by a new categorical structure, which we call a \emph{boundary} structure on $\WUqk$.
This generalizes a result of Kolb which describes a braided module structure on finite-dimensional $\Uqk$-modules when $\g$ is finite-dimensional.


We also consider our construction in the case of the category $\fdUqL$ of finite-dimensional modules over a quantum affine algebra, providing the most comprehensive universal framework to date for large families of solutions of parameter-dependent reflection equations. 
In this case the tensor K-matrix gives rise to a formal Laurent series with a well-defined action on tensor products of any module in $\WUqk$ and any module in $\fdUqL$. 
This series can be normalized to an operator-valued rational function, which we call trigonometric tensor K-matrix, if both factors in the tensor product are in $\fdUqL$.
\end{abstract}


\maketitle

\setcounter{tocdepth}{1}
\tableofcontents


\section{Introduction}

\subsection{} \label{ss:overview} 
The reflection equation was introduced by Cherednik \cite{Che84} and Sklyanin \cite{Skl88} in the context of quantum integrable systems with compatible boundary conditions. 
The underlying algebraic theory was first outlined by Sklyanin and Kulish in \cite{KS92}. 

In \cite{AV22a} we proved that the Drinfeld--Jimbo quantum group $\Uqg$, where $\g$ is a symmetrizable Kac--Moody algebra, is a natural source of 
solutions, which are referred to as K-matrices, of a generalized version of the reflection equation, introduced by Cherednik \cite{Che92}.
Our result is a generalization of previous constructions by Bao and Wang \cite{BW18b} and Balagovi\'{c} and Kolb \cite{BK19}.
In particular, it relies on a more refined initial datum, which is essential for extending the theory to arbitrary quantum Kac-Moody algebras, and subsequently to finite-dimensional modules over quantum affine algebras in \cite{AV22b}.\\

In \cite{Skl88}, Sklyanin observed that a K-matrix on a vector space $V$ can be promoted to a solution of the reflection equation acting on a larger tensor product. 
This is achieved by left- and right-multiplying the initial K-matrix with twisted R-matrices, \ie solutions of twisted Yang-Baxter equations.
In this paper, we apply this principle to the universal K-matrices obtained in \cite{AV22a}, leading to a natural construction of \emph{tensor} K-matrices.
This is a generalization of the result of Kolb \cite{Kol20} for quantum groups of finite type. 
Furthermore, in the case of quantum affine algebras, we obtain (a family of) \emph{trigonometric} tensor K-matrices, \ie operators depending rationally on a formal parameter $z$, acting on any tensor product of irreducible finite-dimensional representations. 
Thus we obtain a ``tensor analogue'' of \cite{AV22b}.

\subsection{}
We recall the construction of universal K-matrices from \cite{AV22a} for the quantum Kac-Moody algebra $\Uqg$ (for more detail also see Sections \ref{ss:aux-datum-k}-\ref{ss:gauged-kmx} of the current work).
It depends on two data.
The first one is a quantum symmetric pair subalgebra, \ie a right coideal subalgebra $\Uqk\subset\Uqg$, which quantizes a \mbox{(pseudo-)}fixed-point Lie subalgebra with respect to a (pseudo-)involution $\theta$ of the second kind, see \cite{Let02,Kol14, RV22, AV22a}. 
We refer to $\Uqk$ as a \emph{QSP subalgebra}. 
The second one is a distinguished algebra automorphism $\psi$ of $\Uqg$, called \emph{twist automorphism}, which naturally appears in the axioms for $K$ and causes a twist of the $\Uqg$-action on modules.
In the simplest case $\psi$ arises essentially as a lift of the automorphism $\theta$ to $\Uqg$.

For any such pair $(\Uqk,\psi)$ there exists a solution $K$ of a reflection equation acting on $\Uqg$-modules in the category $\OUqg$, which is (a quantum analogue of) the category $\cO$ considered in \cite{Kac90} but without finite-dimensionality of weight spaces. 
The reflection equation takes the following form:
\eq{ \label{intro:generalRE}
	R^{\psi \psi}_{21} \cdot (1 \ten K) \cdot R^\psi \cdot (K \ten 1) = (K \ten 1) \cdot R^\psi_{21} \cdot (1 \ten K) \cdot R,
}
where $R$ is the universal R-matrix of $\Uqg$, $R^\psi = (\psi \ten \id)(R)$ and $R^{\psi\psi} = (\psi \ten \psi)(R)$. 
Moreover, $K$ is a $\psi$-twisted centralizer of the subalgebra $\Uqk$:  
\eq{ \label{intro:linear}
	K \cdot b = \psi(b) \cdot K \qq \text{for all } b\in\Uqk.
}
In fact, $K$ is obtained as the unique solution of \eqref{intro:linear} among operators with a prescribed form.
This is consistent with the examples of explicit K-matrices, depending on a spectral parameter, provided in \cite{MN98,DG02,DM03}.  

More precisely, every universal K-matrix obtained in \cite{AV22a} is eventually built out of the \emph{quasi}-K-matrix $\Upsilon$, an operator originally introduced in \cite{BW18b} as a canonical solution in a completion of the positive part of $\Uqg$ of a certain intertwining equation for a QSP subalgebra of $U_q(\fksl_N)$ and generalized in \cite{BK19,AV22a} to almost arbitrary\footnote{There is a constraint on $\theta$ (automatically satisfied for $\fkg$ of finite and affine type) guaranteeing the natural requirement that $\theta$ extend to a map on the extended weight lattice, see \cite{Kol14,AV22a}.} QSP subalgebras of $\Uqg$. 
Although it may appear different at first, this equation reduces to the form \eqref{intro:linear} after a simple algebraic manipulation. 
Relying on this observation, we proved in \cite{AV22a} that $\Upsilon$ satisfies a generalized reflection equation of the form \eqref{intro:generalRE}. 

The equations \eqrefs{intro:generalRE}{intro:linear} are preserved under the following natural group action, called \emph{gauge transformation}:
\eq{
\psi \mapsto \Ad(g) \circ \psi, \qq \qq K \mapsto g \cdot K, 
}
where $g$ is an invertible element of some completion of $\Uqg$ and $\Ad$ denotes conjugation.
The general idea behind gauging is that one can, by modifying $\Upsilon$ and its twist automorphism, obtain a more convenient K-matrix or a more convenient twist automorphism.
Depending on the choice of $g$, this may lead to convergence issues, which requires the use of a smaller category of $\Uqg$-modules. 
For instance, when $\g$ is of finite type, we obtain a family of universal K-matrices interpolating between $\QK{}$, which acts on category $\OUqg$ modules, and the universal K-matrix $\KM{\sf BK}$ emphasized by Balagovi\'{c} and Kolb in \cite{BK19}, which acts on integrable modules in this category (in particular, on finite-dimensional modules).

\subsection{}
In the present paper, we describe a 2-tensor version of the universal K-matrix, which we denote by $\TKM{}$ and refer to as a \emph{tensor K-matrix}. 
This generalizes the results from \cite{Kol20} in the same way that the construction of the universal K-matrix given in \cite{AV22a} generalizes the results from \cite{BK19}.

The idea behind tensor K-matrices is the following.
Note that, owing to the coideal property, it is natural to consider a subcategory $\cM \subseteq \Mod(\Uqk)$ which is a module category over monoidal categories such as $\cO_\infty$, whose objects are $\Uqg$-modules.
In many cases, the latter has a natural braided structure, determined by the universal R-matrix $R$. 
For suitable $\cM$, a tensor K-matrix allows one to extend this braiding to a natural compatible braiding on $\cM$.

The general bialgebraic framework for tensor K-matrices is that of \emph{reflection bialgebras}, see Appendix \ref{ss:reflection-bialgebras} and references therein, where the axioms for $\TKM{}$ are natural conditions independent of the basic K-matrix $\KM{}$. 
In this setting, see Lemma \ref{lem:refl-to-cyl}, one naturally obtains $\KM{}$ from $\TKM{}$ and observes the following factorization
\eq{ \label{intro:factorization}
	\TKM{} = R^\psi_{21} \cdot (1 \ten \KM{}) \cdot R.
} 
Such factorizations appeared first in \cite{Skl88,KS92} and more recently in \cite{DKM03,Enr07,Bro13,BZBJ18,Kol20,LBG23}.

It is tempting to \emph{define} $\TKM{}$ using this formula, starting from a basic K-matrix $\KM{}$. 
However, \emph{a priori}, the expression on the right-hand side of \eqref{intro:factorization} acts only on certain tensor products of $\Uqg$-modules (if $\KM{} = \QK{}$, tensor products in category $\OUqg$). 
Hence, it is necessary to show that $\TKM{}$ is \emph{supported on $\Uqk$}, \ie 
\eq{ \label{intro:support}
	R^\psi_{21} \cdot (1 \ten \KM{}) \cdot R \in \Uqk\wh{\ten}\Uqg
}
where $\Uqk\wh{\ten}\Uqg$ is a completion of $\Uqk\ten\Uqg$. 

In particular, Kolb showed in \cite{Kol20} for quantum groups of finite type with $\KM{} = \KM{\sf BK}$ and $\psi$ a diagram automorphism that \eqref{intro:support} holds.
Furthermore, Kolb's tensor K-matrix $\TKM{}$ acts on tensor products of the form $M \ten V$ with $M$ and $V$ finite-dimensional modules over $\Uqk$ and $\Uqg$, respectively.\footnote{In contrast with our construction, Kolb's tensor K-matrix does not act on a tensor product of non-integrable objects in $\OUqg$, for instance Verma modules.} 
In this case the support property implies a concise categorical statement. 
Restricting to the respective categories of finite-dimensional modules, $\Mod_{\sf fd}(\Uqk)$ is a module category over $\Mod_{\sf fd}(\Uqg)$, by virtue of the coideal property. 
Then, the action of $\TKM{}$ is easily seen to define a braided module structure on $\Mod_{\sf fd}(\Uqk)$ over $\Mod_{\sf fd}(\Uqg)$.\footnote{
More precisely, in order to have a genuine braided module structure, for certain QSP subalgebras $\Mod_{\sf fd}(\Uqg)$ needs to be replaced by its $\bbZ/2\bbZ$-{\em equivariantization}, see \cite[Sec.~3.2]{Kol20}.
}
For arbitrary $\g$, however, this fails, since twisting by $\psi$ does not preserve $\OUqg$.

\subsection{}

We develop instead an alternative approach to the support condition \eqref{intro:support}.
In the rest of the introduction, let $\KM{}$ be the quasi K-matrix $\QK{}$ and $\psi$ the corresponding twist automorphism.
Let $\CO$ be the completion of $\Uqg$ with respect to $\OUqg$, \ie the algebra of endomorphisms of the forgetful functor from category $\OUqg$ to vector spaces, see Section~\ref{ss:completions}.
By \cite{ATL24b}, $\Uqg$ embeds in $\CO$.
Since $\KM{}\in\CO$, $\Ad(\KM{})$ is an endomorphism of $\CO$ which does not preserve $\Uqg$ but which acts on $\Uqk$ as the restriction of $\psi \in \Aut_{\sf alg}(\Uqg)$.\footnote{In the same way, $\Ad(R)$ is not an automorphism of $\Uqg^{\ten 2}$, although it preserves the subalgebra $\Delta(\Uqg)$, acting as the simple transposition of tensor factors.} 
Then we consider the automorphism
\eq{
\xi \coloneqq \Ad(\KM{})^{-1}\circ\psi: \CO \to \CO
}
and the subalgebra\footnote{In the main text we rather define a subalgebra $B'_\xi$ of the quantized derived subalgebra $\Uqgp$. 
In the introduction we give the gist of the argument, working in $\Uqg$ for simplicity.}
\eq{
	B_\xi = \{ x\in \Uqg \,\vert\, (\xi \ten \id)(\Delta(x)) = \Delta(x)\}.
}
Here, the identity $(\xi \ten \id)(\Delta(x)) = \Delta(x)$ is to be understood in the completion of $\Uqg^{\ten 2}$ with respect to category $\OUqg$.
Note that an automorphism similar to $\xi$ was considered in \cite{KY20} and shown to exist for Nichols algebras of diagonal type.

By construction, $B_\xi$ is the maximal coideal subalgebra contained in the fixed-point subalgebra $\Uqg^\xi$, see Lemma \ref{lem:Bxi:maximal}.
In fact, by our first main result, Theorem \ref{thm:QSP-maximal-xi}, we have
\eq{ \label{intro:keyidentity}
	\Uqk=B_\xi
}
(note that $\Uqk\subseteq B_\xi$ follows immediately from \eqref{intro:linear}).
The proof of \eqref{intro:keyidentity} is given in Section \ref{ss:proof-QSP-maximal-xi} and relies on three ingredients: a refined study of the classical limit of $\KM{}$, see Section \ref{ss:classlimitK}, the maximality property of $\Uqk$ proved by Kolb in \cite{Kol14}, and the \emph{Iwasawa decomposition} of $\fkg$ with respect to any pseudo-fixed-point subalgebra $\fkk$, established by the second author and V.~Regelskis in \cite{RV22}.\\

We obtain the desired support property \eqref{intro:support} as follows. 
In Section \ref{ss:weightQSPmodules} we introduce a full subcategory $\cW_\theta \subset \Mod(\Uqk)$. 
We prove in Theorem \ref{thm:reflectionalgebra:QSP} the key properties of $\TKM{}$ in the QSP setting.
Namely, it satisfies the linear equation and coproduct formulas appropriate for a tensor K-matrix. 
Moreover, it naturally acts on any tensor product $M\ten V$, where $M\in\WUqk$ and $V$ belongs to a full subcategory of $\OUqg$ satisfying an integrability property dictated by $\Uqk$.\footnote{
We denote this category by $\POUqg{X}$. 
Its objects are modules in category $\OUqg$ whose restriction to $\UqgX$ is integrable. Here, $X$ is the maximal Dynkin subdiagram such that $\theta$ acts trivially on $\fkg_X$.
} 
The proof of the last property relies on \eqref{intro:keyidentity}.

The objects of $\cW_\theta$ are called \emph{QSP weight modules}; see \cite{BW18a,Wat24} for similar categories.
Note that, up to twisting by an automorphism of $\Uqk$, every irreducible finite-dimensional $\Uqk$-module lies in $\cW_\theta$, see Appendix \ref{ss:findimQSPmodules}.
However, $\cW_\theta$ is large, see \eg Remark \ref{rmk:weightspace:dimension}.
In particular, even in finite type we obtain a much more general result than \cite{Kol20}.
	
\subsection{}
Let $\g$ be a Kac-Moody algebra of untwisted affine type. 
Let $\fdUqL$ be the category of finite-dimensional type-${\bf 1}$ $\Uqgp$-modules.
In \cite{AV22b} we proved that, for suitable twist automorphisms $\psi$, the universal K-matrix $\KM{}$ gives rise to a formal Laurent series with a well-defined action on finite-dimensional modules of $\Uqg$. 
Furthermore, on irreducible modules this operator-valued series simplifies to a scalar multiple of an operator-valued rational function, called \emph{trigonometric K-matrix}.

In Section \ref{ss:trigtensorK} we extend these results to tensor K-matrices. 
Specifically, in Theorem \ref{thm:spectraltensorK} we observe that the action of $\TKM{}$ descends to tensor products of the form $M \ten V$ with $M \in \cW_\theta$ and $V$ a finite-dimensional $\Uqg$-module, yielding linear operators with a formal dependence on $z$. 

In order to study when this dependence simplifies to a rational dependence on $z$, we consider the full subcategory $\cC_\theta \subset \cW_\theta$ whose objects are finite-dimensional.
By standard arguments involving the one-dimensionality of the relevant space of $\Uqk$-intertwiners, see \eg \cite[Sec. 5]{AV22b}, we obtain trigonometric tensor K-matrices whenever the $\Uqk$-module $M\ten V \in \cC_\theta$ is generically irreducible. 
Furthermore, relying on the factorization \eqref{intro:factorization}, we can prove directly, if $V$ is irreducible and $M \in \cC_\theta$ is the restriction of an irreducible $\Uqg$-module, that the action of $\TKM{}$ is proportional to a rational operator called \emph{trigonometric tensor K-matrix}, see Theorem \ref{thm:trigtensorK:1}.

For general quantum symmetric pairs of affine type, the description of (generically) irreducible tensor products in $\cC_\theta$ is open and hence it is unclear whether there exist trigonometric tensor K-matrices beyond those coming from Theorem \ref{thm:trigtensorK:1}.
In the special case where $\Uqgp = U_q(\wh\fksl_2)$ and $\Uqk$ is the q-Onsager algebra, \ie the QSP subalgebra corresponding to the Chevalley involution, Ito and Terwilliger proved that every irreducible object in $\cC_\theta$ is obtained by restriction of an irreducible $\Uqgp$-module \cite{IT09}. 
Thus, \emph{a posteriori}, in this case we obtain a trigonometric K-matrix acting on any tensor product $M\ten V$ with $M$ irreducible in $\cC_\theta$ and $V$ irreducible in $\fdUqL$. 

\subsection{}
The tensor K-matrix provides a platform for the resolution of several open problems in the representation theory of affine quantum symmetric pairs with applications in quantum integrable systems with compatible boundary conditions. 
Below, we provide a brief overview, which will be the focus of forthcoming works.

\subsubsection{}
%
An important application of the universal formalism of transfer matrices built out of R-matrices is the study of q-characters for finite-dimensional modules of quantum affine algebras $\Uqgp$ \cite{FR99,FM01}.
In the original setting one considers the grading-shifted universal R-matrix of a quantum affine algebra $\Uqgp$, and evaluates only one factor on a finite-dimensional representation $V$ of $\Uqgp$. 
By taking the trace, one obtains a commuting family of formal series $\sft_V(z)$ valued in the quantum Borel subalgebra $U_q(\fkb^+)\subset\Uqg$, and hence a morphism of rings:
\begin{equation}
	\begin{tikzcd}
		\left[ \fdUqL \right] \arrow[r, "\sft"] & U_q(\fkb^+)
	\end{tikzcd}
\end{equation}	
where $\left[ \fdUqL \right]$ is the Grothendieck ring of $\fdUqL$.

For the tensor K-matrix, an analogous result follows from the generalization of Sklyanin's formalism of two-boundary transfer matrices \cite{Skl88}, which we describe in \cite{AV24}.
This involves a gauge transformation of the universal tensor K-matrix of an affine quantum symmetric pair $\Uqk\subset\Uqgp$ by a solution of a ``dual'' reflection equation. 
Then, evaluation of the second factor on $V$ and taking the trace yields a commuting family of series $\sfb_V(z)$ valued in $\Uqk$, and a morphisms of rings:
\begin{equation}
	\begin{tikzcd}
		\left[ \fdUqL \right] \arrow[r, "\sfb"] & \Uqk.
	\end{tikzcd}
\end{equation}	

It is natural to expect that $\sfb$ extends to a \emph{boundary} analogue of the q-character map of \cite{FR99,FM01}, providing refined tools in the finite-dimensional representation theory of $\Uqgp$. 
This requires the composition with an Harish-Chandra map associated to a Drinfeld (loop) presentation of $\Uqk$.
The latter was recently found by Lu, Wang and Zhang for various
quasi-split types in \cite{LW21, Zha22, LWZ23, LWZ24}.

For split types, Li and Prze\'{z}dziecki consider instead an \emph{opposite} q-character map
\begin{equation}
	\begin{tikzcd}
		\left[\cC_\theta\right]\arrow[r] & U_q(L\hfin)  
	\end{tikzcd}
\end{equation}
where $L\hfin\subset\gp$ is the loop Cartan subalgebra \cite{Prz24, LP24}. This is a \emph{module} map, directly defined in terms of the spectra of the abelian loop generators of $\Uqk$.
We expect that even this version can be obtained as above, by evaluating
the \emph{first} factor of the gauged tensor K-matrix on finite-dimensional $\Uqk$-modules.

\subsubsection{}
A representation-theoretic approach to Baxter Q-operators for quantum integrable systems and their functional relations was first set out in the 1990s \cite{BLZ, AF97}, by evaluating one leg of the universal R-matrix on a certain infinite-dimensional module of $U_q(\fkb^+)$ called \emph{asymptotic representation}, and taking a trace.
In \cite{HJ12,FH15} the corresponding representation theory of $U_q(\fkb^+)$ was developed, an approach to Baxter's functional relations was established (for all untwisted quantum affine algebras) and a conjecture from \cite{FR99} about expressions of the eigenvalues of transfer matrices in terms of q-characters was proved.

In quasi-split type, the tensor K-matrices $\TKM{}$ considered here lie in a completion of $\Uqk\ten U_q(\fkb^+)$. 
Hence it is natural to attempt a similar approach to functional relations for QSP subalgebras, with applications to systems with boundaries described by trigonometric K-matrices.
This should recover various results known in special cases, such as \cite{YNZ06,BT18,VW20,Tsu21,CVW24}.

\subsubsection{}
Transfer matrices built out of R-matrices are closely related to Yang's scattering matrices and hence to the operators appearing in quantum Knizhnik-Zamolodchikov (KZ) equations. 
Two-boundary analogues of scattering matrices and quantum KZ operators built from solutions of Yang-Baxter and reflection equations were first considered in \cite{Che91,Che92}. 
In the untwisted case, the general connection between two-boundary transfer and scattering matrices was given in \cite{Vla15}.
Solutions of various types for quantum KZ equations with K-matrices coming from QSP subalgebras of $U_q(\wh\fksl_2)$ were given in \cite{JKKMW95,RSV,SV15,HL21}.

As mentioned above, the quantum affine algebra $\Uqgp$ has two important categories of representations: $\OintUqg$ and finite-dimensional modules. 
In \cite{FR92} quantum KZ equations built out of R-matrices were shown to be satisfied by matrix entries of an intertwiner from a suitable irreducible object in $\OintUqg$ to a grading-shifted tensor product of another irreducible in $\OintUqg$ and a finite-dimensional module.
We expect that the universal tensor K-matrix of a QSP algebra $\Uqk \subset \Uqgp$ will allow for 2-boundary quantum KZ equations to be related to analogous intertwiners for certain infinite-dimensional $\Uqk$-modules and finite-dimensional $\Uqgp$-modules, with the construction from \cite{JKKMW95} appearing as a special case.
It will also be interesting to connect such constructions to recent works on differential operators (asymptotic Knizhnik-Zamolodchikov-Bernard operators) defined in terms of classical dynamical r- and k-matrices for symmetric spaces \cite{RS}.

\subsubsection{}
The last application is a universal K-matrix formalism for QSP subalgebras.
In \cite{FRT90}, Faddeev, Reshetikhin, and Takhtajan provided an alternative approach to Drinfeld-Jimbo quantum groups, which is commonly referred to as the R-matrix (or FRT) formalism. 
This yields a new presentation given in terms of a fixed matrix solution of the Yang-Baxter equation. 
The R-matrix formalism was extended to quantum affine algebras of classical Lie type in \cite{RSTS90}.
With the same approach, affine QSP subalgebras in type $\sf A$ have been described by 
Molev, Ragoucy, and Sorba \cite{MRS03} and by Chen, Guay, and Ma \cite{CGM14} in terms of an analogue K-matrix formalism.

More recently, Gautam, Rupert, and Wendlandt provided in \cite{GRW23, RW23} a unified approach to the R-matrix formalism for $\Uqg$ when $\g$ is finite-dimensional,
similar in spirit to the main result of \cite{RSTS90}.
They study the R-matrix algebra $U_{R}$ arising from the evaluation of the universal R-matrix of $\Uqg$ on the tensor square of an arbitrary finite-dimensional irreducible representation. 
Then, they prove that $\Uqg$ can be realized as a Hopf subquotient of $U_R$.
The most interesting feature of the construction is its generality, \ie it works uniformly in every type and for every irreducible representation.

Relying on the construction of the universal tensor K-matrix $\TKM{}$, we expect such a unified approach to extend to quantum symmetric pairs. 
First, for every finite-dimensional irreducible representation $V$, we naturally obtain a coideal subalgebra $U_K\subset U_R$ associated to the basic K-matrix $K_V$ corresponding to a quantum symmetric pair subalgebra $\Uqk\subset\Uqg$. 
Then, through the evaluation of the universal tensor K-matrix $\TKM{}$, we identify $\Uqk$ with a subquotient of $U_K$.
This is expected to extend to the affine case, recovering \cite{MRS03, CGM14}.

\subsection{}
Finally, we conclude with some remarks about the algebraic and the categorical frameworks underlying our construction, which we develop in Appendices~\ref{ss:cylindrical-algebraic} and \ref{ss:bdrybimodcat}, respectively. 

The algebraic framework revolves around the notion of a \emph{reflection bialgebra}, and it is a natural extension of the notion of a cylindrical bialgebra introduced in \cite{AV22a}. 
As the latter can be regarded as a generalization of the original definition given by tom Dieck and H\"aring-Oldenburg \cite{tD98, tDHO98}, and later by Balagovi\'{c} and Kolb \cite{BK19}, the former is a straightforward adaptation of the original definitions given by Enriquez \cite{Enr07}, and later by Kolb \cite{Kol20}. 
It also appears as a special case of a weakly quasitriangular comodule algebra as defined in \cite{KY20}.

Recently, Baseilhac, Gainutdinov and Lemarthe independently proposed this axiomatic scheme for comodule algebras, and 
applied it to the case of the alternating central extension $\cA_q$ of the q-Onsager algebra \cite{LBG23,Lem23}. 
Specifically, they study fusion relations of a certain explicit element of $\fml{\cA_q \ten \End(V)}{z}$, where $V$ is an evaluation representation of $U_q(\wh\fksl_2)$, and conjecture that this element is given by the action of a universal tensor K-matrix for $\cA_q$.
It is natural to expect a connection\footnote{
Note that in \cite{LBG23,Lem23} a twist automorphism is used which is a priori inaccessible for the q-Onsager algebra in our framework: $\psi = \omega \circ \tau$ with $\tau$ the nontrivial diagram automorphism.
This obstacle can be mitigated by noting that in \emph{ibid.} only modules $V$ are considered where $\tau$ acts as conjugation by an antidiagonal matrix in any basis of $V$ consisting of $U_q(\fksl_2)$-weight vectors.
} with our results specialized to the q-Onsager algebra.\\

Our main results are concisely summarized in terms of the categorical framework developed in Appendix~\ref{ss:bdrybimodcat}.
Specifically, the action of $\TKM{}$ yields a {\em boundary} structure on $\WUqk$ with respect to its {\em bimodule} structure over $\POUqg{X}$ determined by the coideal property and $\psi$.
Even in finite type, this yields a generalization of the braided module structure described by Kolb \cite{Kol20}.

\subsection{Outline}

In Section \ref{ss:qKM} we recall the definition and fundamental properties of a quantized universal enveloping algebra $\Uqg$ of a Kac-Moody algebra $\fkg$ and its universal R-matrix.
Next, in Section \ref{ss:qsp} we recall the definition of the subalgebra $\Uqk$ and the corresponding universal K-matrix in the general setting determined by pseudo-involutions of $\fkg$.
The quantum group analogue of the key identity \eqref{intro:keyidentity} is established in Section \ref{ss:class}, see Theorem \ref{thm:QSP-maximal-xi}, via a maximality property of $\Uqk$ and by taking a classical limit.
To promote this to the existence of the tensor K-matrix $\TKM{}$, in Section \ref{ss:tensorK} we describe the appropriate categories of $\Uqk$-modules, notably the category $\cW_\theta$, and $\Uqg$-modules and operators acting on them; the main result is Theorem \ref{eq:tensor-K-def}.
In Section \ref{ss:trigtensorK} we give the tensor K-matrix version of our previous work \cite{AV22b} on formal and trigonometric K-matrices for finite-dimensional modules of quantum loop algebras. 

In Appendix \ref{ss:findimQSPmodules} we show that finite-dimensional irreducible $\Uqk$-modules are objects in $\cW_\theta$, in support of work in Section \ref{ss:tensorK}.
We describe the algebraic tensor K-matrix formalism for general quasitriangular bialgebras in Appendix \ref{ss:cylindrical-algebraic}.
In Appendix \ref{ss:bdrybimodcat} the corresponding categorical-topological framework is developed, resulting in the notion of a boundary bimodule category, which gives rise to actions of cylindrical braid groupoids on tensor products. 
 
\subsection{Acknowledgments}

The authors would like to thank Pascal Baseilhac, Azat Gainutdinov, and Stefan Kolb for useful discussions and their interest in this work.


\section{Quantum Kac-Moody algebras} \label{ss:qKM}

Throughout this paper we will make use of the following notation, for any $k$-linear space $V$:
\[
X^\xi=(\xi\ten\id)(X)\,, \qq X^\xi_{21}=(X^\xi)_{21}\,,\aand X^{\xi\xi}=(\xi\ten\xi)(X)
\]
for any $X\in V\ten V$ and $k$-linear map $\xi: V\to V$. Here the field $k$ should be clear from the context.
 
Also, given a lattice $\mathsf{\Lambda}$ whose $\Z$-basis is clear from the context, we shall denote its non-negative part by $\mathsf{\Lambda}^+$. 


\subsection{Kac-Moody algebras}\label{ss:KM}

Let $\g$ be the Kac-Moody algebra defined over\footnote{
Or over any algebraically closed field of characteristic zero.
} $\bbC$ in terms of a symmetrizable generalized Cartan matrix $A=(a_{ij})_{i,j\in\IS}$ \cite{Kac90}. 
Let $\h\subset\g$ be the Cartan subalgebra.
Denoting the simple roots by $\al_i$ and the simple coroots by $h_i$ ($i \in \IS$), we consider the root and coroot lattices:
\eq{
\Qlat = \Sp_\Z \{ \al_i \, | \, i \in \IS \} \subset\h^*, \qq \Qlat^{\vee} = \Sp_\Z \{ h_i \, | \, i \in \IS \} \subset\h.
}
Let $\rootsys \subset \Q$ be the root system and $\rootsys^+\subset\Qp$ the positive subsystem. 
Finally, set $\g' = [\g,\g]$ and $\h' = \h \cap \g' = \Qlatv\ten_\bbZ\bbC$. 

Let $\iip{\cdot}{\cdot}$ be the invariant bilinear form on $\g$ depending on the choice of a linear complement $\h''$ of $\h'$ in $\h$.
We shall choose a basis $\{d_{1},\dots, d_{\corank(A)}\}$ of $\h''$ such that $\al_i(d_r) \in \Z$ for all $i \in \IS$, $1 \le r \le \corank(A)$. 
Now consider the \emph{extended} coroot lattice and the weight lattice, respectively given by 
\[
\Qvext=\Qlatv\oplus\bigoplus_{k=1}^{\corank(A)}\bbZ d_k\subset{\h}
\aand
\Pext=\{\lambda\in{\h}^*\;|\;\lambda(\Qvext)\subset\bbZ\}.
\]

The bilinear form on $\fkh^*$ dual to $\iip{\cdot}{\cdot}$ will be denoted by the same symbol.
Note that the restriction of this to $\Pext \times \Pext$ takes values in $\frac{1}{m}\Z$ for some $m \in \Z_{>0}$.

\subsection{Drinfeld-Jimbo quantum groups} \label{ss:quantumaffine}

Denote by $\bsF$ the algebraic closure of $\bbC(q)$ where $q$ is an indeterminate.\footnote{
The results of this paper remain valid if $\bsF$ is any algebraically closed field of characteristic 0 and $q\in\bsF^\times$ is not a root of unity.
}
Fix non-negative integers $\{\de{i}\;|\; i\in\IS\}$ such that the matrix $(\de{i}a_{ij})_{i,j\in\IS}$ is symmetric.
The quantum Kac-Moody algebra \cite{Dri86,Jim85} associated to ${\g}$ is the unital associative algebra $\Uqg$ defined over $\bsF$ with generators $\Eg{i}$ and $\Fg{i}$ ($i\in\IS$), and $\Kg{h}$ ($h\in\Qvext$) subject to the following defining relations:
\begin{gather}
	\Kg{h}\Kg{h'}=\Kg{h+h'}, \qq \Kg{0}=1, \\	
	\Kg{h}\Eg{i}=q^{\cp{\rt{i}}{h}}\Eg{i}\Kg{h}, \qq \Kg{h}\Fg{i}=q^{-\cp{\rt{i}}{h}}\Fg{i}\Kg{h},\\
 [\Eg{i},\Fg{j}]=\drc{ij}\frac{\Kg{i}-\Kg{i}^{-1}}{q_i-q_i^{-1}}, \label{Uqag:relns3} 
\end{gather}
for any $i,j\in\IS$ and $h,h'\in\Qvext$, where 
\eq{
q_i = q^{\de{i}}, \qq \qq \Kg{i}^{\pm1}=\Kg{\pm \de{i}\cort{i}}, 
}
together with the quantum Serre relations (see \eg \cite[Cor.~33.1.5]{Lus94})
\eq{ 
\Serre_{ij}(E_i,E_j) = \Serre_{ij}(F_i,F_j)=0 \qq \text{for all } i \ne j
}
where $\Serre_{ij}$ denotes the following polynomial in two noncommuting variables:
\eq{ \label{Serredef}
\Serre_{ij}(x,y) = \sum_{r=0}^{1-a_{ij}} (-1)^n \binom{1-a_{ij}}{r}_{q_i} x^{1-a_{ij}-r} y x^r
}
with the q-deformed binomial coefficient defined in \eg \cite[1.3.3]{Lus94}.\\

We consider the Hopf algebra structure on $\Uqg$ uniquely determined by the following coproduct formulae
\eq{
\Delta(\Eg{i}) = \Eg{i}\ten1+\Kg{i}\ten\Eg{i},
\qq \Delta(\Fg{i}) = \Fg{i}\ten\Kg{i}^{-1}+1\ten\Fg{i},	
\qq \Delta(\Kg{h}) = \Kg{h}\ten\Kg{h}, 
}
for any $i\in\IS$ and $h\in\Qvext$.
The \emph{Chevalley involution} is the algebra automorphism $\chev$ of $\Uqg$ defined by
\begin{equation}\label{eq:chevalley}
\chev(\Eg{i})=-\Fg{i}, \qq \chev(\Fg{i})=-\Eg{i}, \qq \chev(\Kg{h})=\Kg{-h}.
\end{equation}
for any $i\in\IS$ and $h\in\Qvext$.
It is a Hopf algebra isomorphism from $\Uqg$ to its \emph{co-opposite} $\Uqg^{\sf cop}$.\\

We denote by $\Uqnp$, $\Uqnm$, $\Uqh$, and $\Uqhp$ the subalgebras generated by the elements $\{\Eg{i}\}_{i\in\IS}$, $\{\Fg{i}\}_{i\in\IS}$, $\{ \Kg{h} \}_{h\in\Qvext}$, and $\{\Kg{\pm h_i} \}_{i\in\IS}$, respectively. 
We set 
\eq{
\Uqgp=\Uqnp \cdot \Uqhp \cdot \Uqnm, \qq \Uqbpm=\Uqnpm \cdot \Uqh.
} 
The adjoint action of the $K_h$ endows the subalgebras $\Uqnpm$ with a grading by the monoids $\pm \Qlat^+$, respectively; we denote by $\Uqnpm_{\pm \beta}$ the component of degree $\pm \beta$, for $\beta \in \Qlat^+$.


\subsection{Categories of $\Uqg$-modules}\label{ss:cat-O}

If $V$ is an $\Uqh$--module and $\mu\in\Pext$, we denote the corresponding weight space of $V$ by
\[
V_\mu = \{v\in V\,\vert\,\forall h\in\Qvext,\,\,\Kg{h} \cdot v=q^{\mu(h)}v\}.
\]
A $\Uqg$--module $V$ is

\vspace{1mm}
\begin{enumerate}[label=(C\arabic*)]\itemsep2mm
\item \label{cond:weight-dec} 
a (type $\bf 1$) {\it weight module} if $V=\bigoplus_{\mu\in\Pext} V_\mu$;

\item \label{cond:int} 
{\it integrable} if it is a weight module and all $E_i$, $F_i$ ($i\in\IS$) act locally nilpotently;

\item \label{cond:O}
in {\it category $\OUqg$} if it is a weight module and the action of $\Uqnp{}$ is locally finite.\end{enumerate}

We denote by $\WUqg$, $\OUqg$, $\OintUqg$ the full subcategories of weight, category $\OUqg$, integrable category $\OUqg$ modules in $\Mod(\Uqg)$, respectively. 
These three categories are monoidal; furthermore, $\OUqg$ is braided \cite{Dri86, Lus94}, see Section \ref{ss:R-matrix}, and $\OintUqg$ is a semisimple\footnote{
	Note that the integrability condition implies that $V$ is completely reducible as a (possibly infinite) direct sum of simple finite-dimensional modules
	over the subalgebra of $\Uqg$ generated by $E_i$, $F_i$, $K_{\pm h_i}$, for each $i \in \IS$.
} 
braided subcategory of $\OUqg$, whose simple modules are classified by dominant integral weights \cite[Thm~6.2.2, Cor.~6.2.3]{Lus94}.

Finally, let $Z\subseteq \IS$ be any subset and $\UqgZ\subseteq\Uqg$ the subalgebra generated by $\{E_i, F_i, \Kg{h_i}^{\pm1}\}_{i\in Z}$. 
We denote by $\POUqg{Z}\subseteq\OUqg$ the full subcategory of modules whose restriction to $\UqgZ$ is integrable. 
In particular, $\POUqg{\emptyset}=\OUqg$ and $\POUqg{\IS}=\OintUqg$.

\begin{remark}
Note that for $V \in \OUqg$ and any $v\in V$, we have $\Uqnp_{\beta} v=0$ for all but finitely many $\beta\in\sfQ^+$. 
Therefore, $\OUqg$ coincides with the category $\cC^{\operatorname{hi}}$ used in \cite[Sec.~3.4.7]{Lus94}.

We emphasize that we do not impose on objects in $\OUqg$ the condition that they are finitely generated (the resulting category would not be monoidal). 
It is also inconvenient to impose $\dim(V_\mu) < \infty$ as in \cite{Kac90}: the resulting categories would not be preserved under restrictions to subdiagrams $Z \subset \IS$.
Indeed in Section \ref{ss:2-Drinfeld-algebra} we will consider the case $\dim(\fkg_Z)<\infty$.
If in addition $\dim(\fkg)=\infty$, a nontrivial irreducible object in $\OintUqg \subset \OUqg$ is infinite-dimensional and as a $\UqgZ$-module must decompose as an infinite direct sum of irreducible objects in $\OintUqg(\UqgZ)$, which would fail the condition of finite-dimensional weight spaces, see \cite[Rmk.~15.1]{ATL24a}. \hfill \rmkend
\end{remark}

\subsection{Completions}\label{ss:completions}
Let $A$ be an algebra, $\cC\subset\Mod(A)$ a (full) subcategory, and $\End(\FF{\cC}{})$ the
algebra of endomorphisms of the forgetful functor $\FF{\cC}{}\colon\cC\to\vect$. By definition, an
element of $\End(\FF{\cC}{})$ is a collection
\[\varphi=\{\varphi_V\}_{V\in\cC}\in\prod_{V\in\cC}\End(V)\]
which is natural, \ie such that $f\circ\varphi_V=\varphi_W\circ f$ for any $f:V\to W$ in $\cC$. 
The action of $A$ on any $V\in\cC$ yields a morphism of algebras $A\to\End(\FF{\cC}{})$, and factors through the action of $\End(\FF{\cC}{})$ on $V$. 
We refer to $\End(\FF{\cC}{})$ as the completion of $A$ with respect to the category $\cC$. 
If $A$ is a bialgebra and $\cC$ is closed under tensor products, the completion $\End(\FF{\cC}{})$ is naturally endowed with a cosimplicial structure, see, \eg \cite{Dav07, ATL19}.

For any subcategory $\cC'\subset\cC$, there is a canonical morphism
\begin{equation}\label{eq:end-restriction}
\eta\colon\End(\FF{\cC}{})\to\End(\FF{\cC'}{})
\end{equation}
given by the restriction to $\cC'$. A \emph{lift} of $x\in\End(\FF{\cC'}{})$ is any element in $\eta^{-1}(x)\subset\End(\FF{\cC}{})$. 
Since $\eta$ is in general neither surjective nor injective, a lift may not exist or be unique.
However, in the following, we shall often consider operators whose action is given in terms of elements of $A$. 
Hence, if it exists, the lift is canonical (see, \eg~Section \ref{ss:R-matrix}). 
In these cases, we shall make no distinction between an element and its canonical lift.

\subsection{Completions of $\Uqg^{\ten n}$}
Here, we consider various representation-theoretic completions of $\Uqg^{\ten n}$.
Namely, for any $Z\subseteq \IS$ and $n>0$, let $\FF{\pint{Z}}{n}\colon(\POUqg{Z})^{ \boxtimes n}\to\vect$
be the $n$-fold forgetful functor given by 
\[
\FF{\pint{Z}}{n}(V_1,\dots, V_n)=V_1 \ten \cdots \ten V_n.
\]
We set $\FF{}{n}=\FF{\emptyset}{n}$,  $\FF{\pint{\IS}}{n}=\FF{\sint}{n}$, and $\FF{\pint{Z}}{}=\FF{\pint{Z}}{1}$. 
Similarly, let $\FF{\cW}{n}:\WUqg^{\boxtimes n}\to\vect$ be the $n$-fold forgetful functor with respect to the category of weight $\Uqg$-modules. 
Note that the chain of inclusions $\POUqg{Z}\subset\WUqg\subset\Mod(\Uqg)$ induces a chain of morphisms of algebras $\Uqg^{\ten n}\to\CWUqg{n}\to\CPOUqg{Z}{n}$.
By \cite[Thm.~3.1]{ATL24b}, $\Uqg^{\ten n}$ embeds in $\CPOUqg{Z}{n}$ and therefore in $\CWUqg{n}$.

For any monoidal category $\cC \subset \Mod(\Uqg)$, the morphisms
\[
\Delta, \Delta^{\op}\colon\End(\FF{\cC}{})\to\End(\FF{\cC}{}\boxtimes\FF{\cC}{}),
\]
are defined as follows. 
For any $a = (a_V)_{V \in \cC} \in \End(\FF{\cC}{})$, we set
\[
\Delta(a)_{VW}=a_{V\ten W}
\quad\mbox{and}\quad
\Delta^{\op}(a)_{VW}=(1\,2)\circ a_{W\ten V}\circ(1\,2)
\]
where $V,W \in \cC$. 
The morphisms $\Del \ten \id$ and $\id\ten\Delta$ are similarly defined.

\subsection{Tensor weight operators} \label{ss:2-weight}

Let $f\colon\Pext\to\bsF$ be a $\bbZ$-linear functional which is self-adjoint with respect to the pairing on $\Pext$, \ie 
$(f(\lambda),\mu)=(\lambda,f(\mu))$ for any $\lambda,\mu\in\Pext$.
Recall from \cite[Sec.~4.8]{AV22a} the symmetric functional $\WO{f}: \Pext \times \Pext \to \bsF$ given by
\begin{equation} \label{weightop:def}
\WO{f}(\la,\mu) = q^{(f(\la),\mu)}.
\end{equation}
For any $V,W\in\WUqg$, we denote by the same symbol the diagonal operator on $V\ten W$, acting on the subspace $V_\lambda\ten W_\mu$ as multiplication by $\WO{f}(\lambda,\mu)$. 
Then, one readily checks that $\WO{f}\in\End(\sff_\cW^2)$ and in $\End(\sff_\cW^3)$ the following identities hold:
\eq{ \label{C:coproduct}
(\Del \ten \id)(\WO{f}) = \WO{f}_{13} \WO{f}_{23}, \qq (\id \ten \Del)(\WO{f})= \WO{f}_{13} \WO{f}_{12}.
}
In the case $f = \id$, this element appears in the factorized formula for the universal R-matrix, which we review now. 

\subsection{The R-matrix}\label{ss:R-matrix}

By \cite{Dri86, Lus94}, the universal R-matrix of $\Uqg$ is a distinguished invertible element $\RM{}\in\End(\FF{}{}\boxtimes\FF{}{})$, whose action induces a braiding on the category $\OUqg$. 
More precisely, $\RM{}$ satisfies the intertwining identity in $\End(\FF{}{}\boxtimes\FF{}{})$
\begin{gather}
	\label{eq:R-intw-QG} R \cdot \Delta(x) = \Delta^{\op}(x) \cdot R,
\end{gather}
where $x\in \Uqg$,
and the coproduct identities in $\End(\FF{}{}\boxtimes\FF{}{}\boxtimes\FF{}{})$
\eq{ \label{eq:R-coprod-QG} 
(\Del \ten \id)(R) = R_{13}\cdot R_{23} \qq\mbox{and} \qq (\id \ten \Del)(R) = R_{13}\cdot R_{12}, 
}
where $\Delta$ denotes the coproduct and $\Delta^{\op} = (12) \circ \Delta$. 
From \eqref{eq:R-intw-QG} and \eqref{eq:R-coprod-QG} it readily follows that $R$ is a solution of the {Yang-Baxter equation}
\eq{ \label{eq:YBE-QG}
R_{12}\cdot R_{13}\cdot R_{23}=R_{23}\cdot R_{13}\cdot R_{12}.
}
in $\End(\FF{}{}\boxtimes\FF{}{}\boxtimes\FF{}{})$.

The universal R-matrix arises from the Drinfeld double construction of $\Uqg$ as the canonical tensor of a Hopf pairing between $\Uqbm$ and $\Uqbp$. 
More explicitly, we have
\eq{ \label{eq:R-matrix}
R=\WO{\id}\cdot\Theta
}
where the element $\Theta$, commonly referred to as the {quasi} R-matrix \cite[Ch.~4]{Lus94}, is of the form
\eq{
\Theta=\sum_{\beta\in\Qlat^+}\Theta_\beta\in\Uqnm\widehat{\ten}\Uqnp, \qq \Uqnm\widehat{\ten}\Uqnp=\prod_{\beta}\Uqnm_{-\beta}\ten\Uqnp_\beta.
}

Relying on the condition \ref{cond:O}, every element in $\Uqnm\widehat{\ten}\Uqnp$ has a well-defined and natural action on any tensor product of the form $V\ten W$ with $V\in\Mod(\Uqg)$ and $W\in\OUqg$.
More precisely, let $\wt{\FF{}{}}\colon\Mod(\Uqg)\to\vect$ be the forgetful functor. Following \cite[Sec.~3]{ATL24b}, one proves that the canonical map
\begin{equation}\label{eq:R-completion-embedding}
\begin{tikzcd}
\Uqnm\widehat{\ten}\Uqnp\arrow[r] & \End(\wt{\FF{}{}}\boxtimes\FF{}{})
\end{tikzcd}
\end{equation}
is an embedding. 
In particular, $\Theta\in\End(\wt{\FF{}{}}\boxtimes\FF{}{})$.  
Since the action of $\WO{\id}$ is defined only on weight modules, it follows that $R\in\End(\FF{\cW}{}\boxtimes\FF{}{})$. Then, it defines a braiding on $\OUqg$ through the morphism 
\begin{equation}\label{eq:weight-to-O}
\begin{tikzcd}
\End({\FF{\cW}{}}\boxtimes\FF{}{})\arrow[r]&\End(\FF{}{2})
\end{tikzcd}
\end{equation}
induced by the inclusion $\OUqg\subset\WUqg$.


\section{Quantum symmetric pairs} \label{s:qsp}

We recall the definition of quantum symmetric pairs for quantum Kac-Moody algebras following \cite{Kol14, RV20, AV22a}. 

\subsection{Pseudo-involutions} \label{ss:pseudo}

Let $\Aut({A})$ be the group of diagram automorphisms, \ie the group of permutations $\tau$ of $\IS$ such that $a_{ij}=a_{\tau(i)\tau(j)}$.
Any $\tau \in \Aut(A)$ canonically defines an automorphism of ${\g'}$, given on the generators by $\tau(\cort{i}) = \cort{\tau(i)}$, $\tau(e_i)=e_{\tau(i)}$, and $\tau(f_i)=f_{\tau(i)}$. 
In \cite[Sec.~4.9]{KW92}, Kac and Wang describe a procedure (depending on the choice of the subspace $\h''\subset\h$) to extend a diagram automorphism from ${\h'}$ to ${\h}$. 
Then, we assume that the basis $\{d_1, \dots, d_{\corank(A)}\}$ of $\h''$ is chosen in such a way that the Kac-Wang extension of $\tau$ on $\h$ preserves the extended coroot lattice $\Qvext$ and, by duality, the weight lattice $\Pext$.

\begin{remark}
The problem of the existence of a basis of $\h''$ compatible with a given $\tau \in \Aut(A)$ is discussed in \cite[Prop.~3.2]{AV22a} in the case $\corank(A)=1$. 
If $A$ is of affine type, such a basis exists for all $\tau$, see also \cite[Prop.\ 2.12]{Kol14}. \rmkend
\end{remark}

Let $X \subset\IS$ be a subset of indices such that the corresponding Cartan matrix $A_X$ is of finite type.
We denote by $\g_X\subset\g$ the corresponding semisimple finite-dimensional Lie subalgebra.
Let $\tau$ be an involutive diagram automorphism stabilizing $X$ such that $\tau|_X$ equals $\oi_X$, the \emph{opposition involution} of $X$.
This is the involutive diagram automorphism of $X$ induced by the action of the longest element $w_X$ of the Weyl group $W_X$ on $\Qlat_X$: $w_X(\al_i) = -\al_{\oi_X(i)}$ for all $i \in X$.

We consider the Lie algebra automorphism of $\fkg$ given by
	\begin{align}\label{theta:def}
		\theta = \Ad(w_X) \circ \omega \circ \tau
	\end{align}
where $\omega$ denotes the Chevalley involution on $\g$ and $\Ad(w)$ for $w \in W$ is the corresponding braid group action realized in terms of triple exponentials. 
Note that $\Ad(w_X)^2$ acts on the root space $\g_\la$ ($\la \in \Qlat$) as multiplication by $(-1)^{\la(2\rho^\vee_X)}$, with $\rho^\vee_X$ the half-sum of the positive coroots associated to $\g_X$.
Then $\theta$ fixes $\g_X$ pointwise and exchanges positive and negative root spaces not contained in $\g_X$.
Further, $\theta$ restricts to $\fkh$ as the involutive map $-w_X \circ \tau$; we use the notation $\theta$ also to denote the map on $\h^*$ dual to $\theta|_\h$, so that $\theta\big(\g_\al \big) = \g_{\theta(\al)}$ for all $\al \in \rootsys$.

We call $\theta$ a \emph{pseudo-involution (of the second kind)}.
The datum $(X,\tau)$ can be recovered from $\theta$ since $X=\{ i \in \IS \, | \, \theta(h_i)=h_i \}$ and hence we will use the subscript $\tsat$ for objects explicitly defined in terms of $(X,\tau)$. 
%

Any ${\bm y} \in (\C^\times)^\IS$ may be viewed as a multiplicative character ${\bm y}: \Qlat \to \C^\times$ according to ${\bm y}(\al_i) = y_i$ for $i \in \IS$. 
In turn, it gives rise to a Lie algebra automorphism $\Ad(\bm y)$ of $\fkg$ which acts on the root space $\fkg_\al$ as multiplication by $\bm y(\al)$.
For ${\bm y} \in (\C^\times)^\IS$ such that $y_i = 1$ if $i \in X$, consider the following modification of $\theta$, which coincides with $\theta$ on the Lie subalgebras $\g_X$ and $\h$:
\eq{ \label{phi:def}
\phi = \theta\circ\Ad({\bm y})^{-1}. 
}

\subsection{Pseudo-fixed-point subalgebras} \label{ss:Liesubalgebra}
Let $\k\subset\g$ be the Lie subalgebra generated by $\n^+_X$, $\h^\theta$ and the elements 
\[
f_i + \phi(f_i) = f_i+ y_i \theta(f_i)
\]
for $i\in \IS$. 
We set $\fkk' = \fkk \cap \fkg'$.

Henceforth we assume that $(X,\tau)$ is a generalized Satake diagram, \ie for all $i,j \in \IS$ such that $\theta(\al_i) = -\al_i-\al_j$ we have $a_{ij} \ne -1$.
Further, we assume that, for all $i \in \IS$,
\eq{ \label{y:condition}
y_i = y_{\tau(i)} \qq \text{if} \qq (\al_i,\theta(\al_i)) = 0.
}
With these assumptions, the Lie subalgebra $\k$ resembles the fixed-point subalgebra of an involution in the sense that $\k \cap \h = \h^\theta$, see \cite{RV20,RV22,AV22a}, and we call it a {\em pseudo-fixed-point subalgebra}.

\subsection{Involutions and fixed-point subalgebras}
The genuinely involutive case is described by the following refinement, cf. \cite[Sec.~2.4]{Kol14}.

\begin{lemma} \label{lem:phi:involution}

The following conditions are equivalent:

\vspace{1mm}
\begin{enumerate}\itemsep2mm
\item 
$\phi$ is an involution;

\item \label{parameter:condition:involutivity}
$y_i y_{\tau(i)}^{-1} = (-1)^{\al_i(2\rho^\vee_X)}$ for all $i \in \IS \setminus X$;

\item 
$\k$ is the fixed-point subalgebra $\g^\phi$.

\end{enumerate}

\end{lemma}

Note that condition \ref{parameter:condition:involutivity} above implies that $\alpha_i(\rho^\vee_X)\in\bbZ$ whenever $i = \tau(i) \in \IS \backslash X$, \ie $(X,\tau)$ is a \emph{Satake diagram} as defined in \cite[Def.~2.3]{Kol14}.

\begin{proof}[Proof of Lemma \ref{lem:phi:involution}.]
Since $\theta^2 = \Ad(w_X)^2$, by direct inspection, $\phi$ is an involutive automorphism if and only if \ref{parameter:condition:involutivity} holds.
In this case, \cite[Lem.~2.8]{Kol14} implies that $\k = \g^\phi$. 
Conversely, if $\k$ is contained in $\g^\phi$ then we must have $\phi^2(f_i)=f_i$ for all $i \in \IS$ so that $\phi^2$ fixes $\n^-$ pointwise.
Noting that $\phi|_\h = -w_X \circ \tau$ is an involution, we deduce that $\ad(f_i)$ annihilates $e_j-\phi^2(e_j) \in \n^+$ for all $i,j \in \IS$.
Using \cite[Lemma 1.5]{Kac90} we conclude that $\phi^2(e_j)=e_j$ for all $j \in \IS$.
Hence $\phi$ is an involution.
\end{proof}


\subsection{Quantum pseudo-involutions}\label{ss:qtheta}

We now describe a convenient algebra automorphism $\tsatq$ of $\Uqg$ which quantizes $\tsat$.
This is obtained by choosing a lift for each of the three factors in $\tsat$. 
First, we consider the standard Chevalley involution on $\Uqg$ given by \eqref{eq:chevalley}.
The diagram automorphism $\tau$ extends  to an automorphism of $\Uqg$ given on the generators by $\tau(\Eg{i})=\Eg{\tau(i)}$, $\tau(\Fg{i})=\Fg{\tau(i)}$, and $\tau(\Kg{h})=\Kg{\tau(h)}$.

The action of the Weyl group operator $w_X\in W_X$ is lifted to $\Uqg$ as follows. 
Let $\qWS{X}$ be the braid group operator on modules in $\OintUqg$ corresponding to $w_X$ \cite[Sec.~5]{Lus94}.
More precisely, given a reduced expression $s_{i_1}\cdots s_{i_{\ell}}$ of $w_X$ in terms of fundamental reflections, one sets $\qWS{X}= \qWS{i_1}\cdots\qWS{i_{\ell}}$, where $\qWS{j}=T''_{j,1}$ 
in the notation from \cite[5.2.1]{Lus94}. 
It follows from the braid relations that $\qWS{X}$ is independent of the 
chosen reduced expression.

In \cite[Sec.~4.9]{AV22a}, we introduced a Cartan correction of $\qWS{X}$ given by
\begin{align} \label{braidtheta:def}
	\bt{\tsat} = \tcorr{\tsat}\cdot\qWS{X}
\end{align}
Here, $\tcorr{\tsat}$ is the operator defined on any weight vector of weight $\lambda$ as multiplication by $q^{\iip{\tsat(\lambda)}{\lambda}/2+\iip{\lambda}{\rho_X}}$, with $\rho_X$ the half-sum of the positive roots in $\rootsys_X$, the root system of the finite-dimensional semisimple Lie algebra $\fkg_X$. 
Thus, $\bt{\tsat}$ can be thought of as an element of $\CPOUqg{X}{}$.

In \cite[Sec.~6.6]{AV22a}, we also introduced a Cartan correction of the universal R-matrix $R_X$ of $\UqgX$ given by 
\begin{equation}\label{Rtheta:def}
	R_\theta=\WO{\theta}\cdot\Theta_X\in\End(\FF{}{2})
\end{equation}
where we recall $\WO{\theta} \in \End(\FF{\cW}{2})$ as defined by \eqref{weightop:def} and $\Theta_X$ is the quasi R-matrix of $\UqgX$. 
Then, by \cite[Lemma~6.8]{AV22a}, in $\End(\FF{\pint{X}}{2})$ we have
\begin{equation}\label{Rtheta:factorization}
	R_\theta=(\bt{\tsat} \ten \bt{\tsat} )^{-1}\Delta(\bt{\tsat}).
\end{equation}

Further, we obtain an algebra automorphism $\adbt{\tsat}=\Ad{(\bt{\tsat})}$ of $\COUqg{\pint{X}}{}$, which preserves $\Uqg$ \cite[Lem.~4.3 and 5.2]{AV22a}.
Set 
\begin{equation} \label{thetaq:def}
	\tsatq = \adbt{\tsat}\circ\omega\circ\tau
\end{equation}
As an algebra automorphism of $\Uqg$, $\tsatq$ is the identity on $\UqgX$, satisfies $\tsatq(K_h) = K_{\theta(h)}$ for all $h \in \Qvext$ and satisfies $\tsatq\big(\Uqg_\la \big) = \Uqg_{\theta(\la)}$ for all $\la \in \Qlat$, see, \eg \cite[Sec.~6.7]{AV22a}.

We consider the tuple $\Parc \in (\bsF^\times)^\IS$ as a multiplicative character $\Parc: \Qlat \to \bsF^\times$, yielding a Hopf algebra automorphism $\Ad(\Parc)$ of $\Uqg$ which acts on the root space $\fkg_\al$ as multiplication by $\Parc(\al)$.
In analogy with \eqref{phi:def}, we set
\eq{ \label{phiq:def}
\phi_q = \tsat_{q} \circ \Ad(\Parc^{-1})
} 
so that $\phi_q(F_i) = \Parc_i \tsat_q(F_i)$.


\subsection{QSP subalgebras}\label{ss:qsp}
Let $\Parsetc\subset(\bsF^\times)^{\IS}$ and $\Parsets\subset\bsF^{\IS}$ be the parameter sets defined in \cite[Def.~6.11]{AV22a}.
The {QSP subalgebra} associated to $\tsat$ with parameters $(\Parc,\Pars)\in \Parsetc \times \Parsets$ is the subalgebra $\Uqk\subset\Uqg$ generated by the subalgebra $U_q(\fkn^+_X) \coloneqq \bsF\langle E_j \rangle_{j \in X}$, the commutative subalgebra $U_q(\h^\tsat) \coloneqq \bsF\langle \Kg{h}\;\vert\; h\in(\Qvext)^\tsat\rangle$, and the elements $\Bg{i}$ ($i\in\IS$), given by 
\eq{ \label{Bi:def}
	\Bg{i} =
	\left\{
	\begin{array}{cl}
		\Fg{i} & \mbox{if }\, i\in X ,\\
		\Fg{i} + \phi_q(\Fg{i}) + \pars{i} \, \Kg{i}^{-1}  & \mbox{if }\, i\not\in X.
	\end{array}
	\right.
}
We set 
\eq{
\Uqkp=\Uqk\cap\Uqgp.
} 

\begin{remarks} \label{rmks:QSPdef} \hfill
	\begin{enumerate}\itemsep2mm
	\item 
	Up to reparametrization, the definition of $\Uqk$ coincides with that given in \cite{Let02,Kol14},
	see \cite[Rmk.~6.13]{AV22a}.
	\item It is known that $\Uqk$ is a right coideal, \ie $\Delta(\Uqk)\subset\Uqk\ten\Uqg$, with the additional property $\Uqk\cap\Uqh=U_q(\h^\tsat)$; $\Uqkp$ satisfies analogous properties with respect to $\Uqgp$ and $U_q((\h')^\tsat)$, see \cite{Kol14,AV22a}.
	\item
	If $\corank(A)\leqslant 1$ then $\Uqkp=\Uqk$, see \cite[Sec.~6.2]{AV22a}. 
	\rmkend
	\end{enumerate}
\end{remarks}


\subsection{Auxiliary R-matrices}\label{ss:aux-datum-k}
Fix an extension of $\Parc : \Qlat \to \bbF^\times$ to a character
$\Parc: \Pext\to\bbF^\times$. We consider the diagonal operator $\Parc\in\CWUqg{}$
acting on any weight space $V_\lambda$ as multiplication by $\Parc(\lambda)$.
Moreover, one has $\Del(\Parc) = \Parc \ten \Parc$ in $\End(\FF{\cW}{}\boxtimes\FF{\cW}{})$.
Finally, $\Ad(\Parc)$ is an algebra automorphism  of $\CWUqg{}$, which restricts to a Hopf algebra automorphism of $\Uqg$.\\

For any $V\in\Mod(\Uqg)$, and any algebra automorphism $\psi :\Uqg \to \Uqg$, let ${\Vpsi} = \psi^*(V) \in\Mod(\Uqg)$ be the pullback of $V$ under $\psi$.

\begin{lemma}\label{lem:twistpair:properties}
If $\psi = \phi_q^{-1}$, the following statements hold. 
\begin{enumerate}\itemsep2mm
\item \label{W:psitwist} For any $V\in\WUqg$, ${\Vpsi}\in\WUqg$.
\item \label{W:psitwist2} For any $V\in\POUqg{X}{}$, ${\Vpsi}$ is in the opposite category, \ie the weight module ${\Vpsi}$ is locally finite over $\Uqnm$ and integrable over $\UqgX$.
\item \label{twistpair:coproduct}
For any $V,W\in\OUqg$, the operator 
\[
	J=(\Parc\ten\Parc)\cdot R_\theta\cdot \Delta(\Parc)^{-1}=R_\theta\in\End(\FF{}{2})
\]
yields an intertwiner
\[
	J_{VW}\colon\VWpsi\to {\Vpsi}\ten^{\op}{\Wpsi},
\]
where $\ten^{\op}$ indicates that the $\Uqg$-action is given by the opposite coproduct $\Delta^{\op}$.
\item \label{twistpair:Rmatrix} For any $V,W\in\OUqg$, the action of the universal R-matrix on ${\Vpsi}\ten{\Wpsi}$ is well-defined and it is given by
\begin{equation}\label{eq:twisted-R}
	R_{{\Vpsi}{\Wpsi}}=J_{VW}\circ (R_{21})_{VW}\circ (J_{21})_{VW}^{-1},
\end{equation}
where $(R_{21})_{VW}=(1\,2)\circ R_{WV}\circ (1\,2)$ and $(J_{21})_{VW}=(1\,2)\circ J_{WV}\circ (1\,2)$.
\end{enumerate}
\end{lemma}

\begin{proof} \mbox{}

\begin{enumerate}\itemsep2mm
\item 
It suffices to note that, for any $h\in\Qvext$, $\psi(K_h)=K_{\theta(h)}$. 
Hence, ${\Vpsi}\in\WUqg$.

\item
Since $\tsat_{q}=\adbt{\tsat}\circ\omega\circ\tau$, it is enough to observe that 
the restriction under $\adbt{\tsat}$ preserves $X$-integrable category $\OUqg$ modules, see \eg \cite[Lemma~5.2]{AV22a}.

\item[(c)-(d)] See \cite[Prop.~8.7]{AV22a}. 
In particular, note that
\begin{align}
(\Parc\ten\Parc)\cdot R_\theta = R_\theta\cdot (\Parc\ten\Parc) = R_\theta\cdot \Delta(\Parc)
\end{align}
since $R_\theta$ is a weight preserving operator and $\Parc$ is grouplike.\hfill \qedhere

\end{enumerate}

\end{proof}

\begin{remark}
The identity \eqref{eq:twisted-R} is somewhat surprising.
Whenever $V\in\POUqg{X}{}$, $\Uqnm$ acts locally finitely on $\Vpsi$ by \ref{W:psitwist2}. Thus $R_{{\Vpsi}{\Wpsi}}$ is obviously well-defined.
However, it is not clear \emph{a priori} why the same operator is well-defined on modules which are not $X$-integrable. 
\hfill\rmkend   
\end{remark}

\begin{prop}\label{prop:aux-R-matrices} 
Let $(\Uqg,\Uqk)$ be a quantum symmetric pair. 
Set $\psi = \phi_q^{-1}$.

\begin{enumerate}\itemsep2mm
\item \label{Rpsi:def} 
There is a well-defined operator $\RM{}^{\psi}\in\End(\FF{\cW}{}\boxtimes\FF{}{})$ given by the collection of maps
\eq{
\hspace{12pt} \RM{VW}^{\psi} =\RM{{\Vpsi}W}\in\End(V\ten W), \hspace{70pt} V \in \WUqg, W \in \OUqg. 
}	

\item \label{Rpsi21:def} 
There is a well-defined operator $\RM{21}^{\psi}\in\End(\FF{}{}\boxtimes\FF{\cW}{})$ given by the collection of maps
\eq{
(\RM{21}^{\psi})_{VW} = (1\,2)\circ\RM{WV}^{\psi}\circ(1\,2)\in\End(V\ten W), \hspace{13pt} V \in \OUqg, W \in \WUqg. \hspace{5pt}
}	

\item 
There is a well-defined operator $\RM{}^{\psi\psi}\in\End(\FF{}{}\boxtimes\FF{}{})$ given by the collection of maps
\eq{
\hspace{1pt} \RM{VW}^{\psi\psi} = \RM{{\Vpsi}{\Wpsi}}\in\End(V\ten W), \hspace{64pt} V,W \in \OUqg. \hspace{14pt}
}

\end{enumerate}	

\end{prop}

\begin{proof} \mbox{}

\begin{enumerate}\itemsep2mm

\item
By Section \ref{ss:R-matrix}, $R\in\End(\FF{\cW}{}\boxtimes\FF{}{})$. 
Thus, by Lemma~\ref{lem:twistpair:properties} \ref{W:psitwist}, it follows that $\RM{{\Vpsi}W}$ is well-defined and natural in both $V$ and $W$.

\item
As for \ref{Rpsi:def}, it is enough to observe that 
	\begin{align}
	(\RM{21}^{\psi})_{VW}=(1\,2)\circ\RM{WV}^{\psi}\circ(1\,2)
	=(1\,2)\circ\RM{{\Wpsi}V}\circ(1\,2)
	=(\RM{21})_{V{\Wpsi}}
\end{align}
and $R_{21}\in\End(\FF{}{}\boxtimes\FF{\cW}{})$.

\item
It follows readily from Lemma~\ref{lem:twistpair:properties} \ref{twistpair:Rmatrix}. 
\hfill \qedhere

\end{enumerate}

\end{proof}

\subsection{Basic K-matrices for quantum symmetric pairs} \label{ss:basicKmatrices}

In \cite{AV22a} we pointed out that the QSP subalgebra $\Uqkp$ gives rise to a family of \emph{cylindrical structures} $(\psi,J,K)$ on $\Uqg$; for a purely algebraic review of cylindrical structures, see Section \ref{ss:cyl-bialg}.
In the quantum group setting a cylindrical structure amounts to an invertible element $K$ of a completion of $\Uqg$, called \emph{basic K-matrix}, an algebra automorphism $\psi: \Uqg \to \Uqg$ and an invertible element $J$ of a completion of $\Uqg \ten \Uqg$, such that the intertwining identity
\begin{align}\label{eq:k-intertwiner}
\KM{}\cdot b=\psi(b)\cdot\KM{}\, \qq \text{for all } b \in \Uqkp,
\end{align}
and the coproduct formula
\begin{align}\label{eq:coprod-id}
\Delta(\KM{}) = J^{-1 }\cdot (1\ten\KM{}) \cdot \RM{}^{\psi} \cdot (\KM{} \ten 1)
\end{align}
are satisfied.
Additionally, $(\psi,J)$ should be a \emph{twist pair} for $\Uqg$ (see \ref{ss:qt-hopf}) up to completion, which amounts to the statements of Lemma \ref{lem:twistpair:properties} \ref{twistpair:coproduct}-\ref{twistpair:Rmatrix}.
As a consequence, we obtain the \emph{generalized reflection equation} (we can follow the purely algebraic proof of Proposition \ref{prop:basicK:RE}):
\begin{equation}\label{eq:K-RE}
R^{\psi \psi}_{21} \cdot (1 \ten \KM{}) \cdot R^\psi \cdot (\KM{} \ten 1) = (\KM{} \ten 1) \cdot R^\psi_{21} \cdot (1 \ten \KM{}) \cdot R.
\end{equation}
We now follow the approach of \cite{AV22a}, in which we first construct a distinguished cylindrical structure and then generate new examples by gauge transformations.

\subsection{The quasi K-matrix}\label{ss:univ-kmx}

One solution of \eqrefs{eq:k-intertwiner}{eq:coprod-id} is distinguished by the fact that $K$, in this case also called \emph{quasi K-matrix}, lies in a completion of $U_q(\fkn^+)$. 
Its construction, originally due to Bao and Wang \cite{BW18b}, was subsequently generalized by Balagovi\'c and Kolb in \cite{BK19} to the symmetrizable Kac-Moody setting with constraints on the parameters $(\Parc,\Pars)\in\Parsetc\times\Parsets$. 
In \cite{AV22a} the parameter constraints were lifted and \eqref{eq:k-intertwiner} was rewritten in the indicated form.

\begin{theorem} \cite[Lemma 8.1 \& Prop. 8.3]{AV22a} \label{thm:av-k-mx}
Let $(\Uqg,\Uqk)$ be a quantum symmetric pair and set 
\eq{
\psi = \phi_q^{-1}, \qq J=R_\theta.
} 
Then there is a unique $K = \QK{} \in\COUqg{}{}$ of the form
\eq{
\QK{}=\sum_{\mu \in (\Qp)^{-\tsat}}\QK{\mu} \in\prod_{\mu\in\Qp}\Uqnp_\mu, \qq \QK{0}=1
}
which satisfies \eqrefs{eq:k-intertwiner}{eq:coprod-id}.
\end{theorem}

\begin{remarks}
	\hfill
	\begin{enumerate}\itemsep2mm
	\item 
	By the weight decomposition of $\QK{}$, one has $u\QK{}=\QK{}u$ for any $u\in U_q(\h^\tsat)$. 
	Thus, the identity \eqref{eq:k-intertwiner} for $\Uqkp$ implies the same identity for $\Uqk$, which appears in \cite[Lemma 8.1]{AV22a}.
	\item By Proposition~\ref{prop:aux-R-matrices}, \eqref{eq:coprod-id} and \eqref{eq:K-RE} lie in $\End(\FF{}{}\boxtimes\FF{}{})$.
	\hfill\rmkend
	\end{enumerate}
\end{remarks}

\subsection{Gauged basic K-matrices} \label{ss:gauged-kmx}
Having obtained one solution of \eqrefs{eq:k-intertwiner}{eq:coprod-id} for a given quantum symmetric pair $(\Uqg,\Uqk)$, we now explain how others can be obtained.
In \cite{AV22a,AV22b}, we described a procedure to obtain new cylindrical structures by \emph{gauge-transformations}, see \cite[Rmk.~8.11]{AV22a} and \cite[Cor.~3.6]{AV22b}.

We will set 
\eq{
(\psi,J,K) = (\phi_q^{-1},R_\theta,\QK{})
} 
and use it as the initial datum in this procedure, carefully keeping track of the category of modules on which the new datum acts.

\begin{prop} \cite[Thm.~8.8]{AV22a} \label{prop:gauge-k-matrix}
Fix a subset $Z\subseteq \IS$. 
Let $\Gg_Z$ be the group of invertible elements $\gau\in \CPOUqg{Z}{}$ such that $\Ad(\gau)$ preserves $\Uqg\subset\CPOUqg{Z}{}$. 
For any fixed $\gau\in\Gg_Z$, set
\begin{equation}
	\psi_\gau = \Ad(\gau) \circ \psi,
	\quad
	J_\gau = (\gau\ten\gau) \cdot J \cdot \Delta(\gau)^{-1}
	\quad\mbox{and}\quad
	\KM{\gau} = \gau \cdot \KM{}.
\end{equation}
Then the identities \eqref{eq:k-intertwiner} and \eqref{eq:coprod-id} with $(\psi,J,\KM{})$ replaced by $(\psi_\gau,J_\gau,\KM{\gau})$, are satisfied in $\End(\FF{\pint{Z}}{})$ and $\End(\FF{\pint{Z}}{} \boxtimes \FF{\pint{Z}}{})$, respectively.
\end{prop}

\begin{proof}
It is enough to observe that the results of Lemma~\ref{lem:twistpair:properties} and Proposition~\ref{prop:aux-R-matrices} remain valid for this choice of $\psi$ and $J$, up to replacing $\FF{}{}$ with $\FF{\pint{Z}}{}$.
\end{proof}

%
%


\section{The classical limit of universal K-matrices} \label{ss:class}

In this section we discuss specialization. 
In the formal ($\hbar$-adic) setting, it is clear that the specializations below are well-defined.
We will use the formalism for specialization as discussed in \cite[Sec.~2.4]{Wat21} via subrings of $\bsF$ defined in terms of certain Puiseux series.
The main result of \cite[Sec.~10]{Kol14}, characterizing $\Uqkp$ as a maximal subspace of $\Uqgp$ whose specialization is $U(\fkk')$, directly carries over to this setting.

In this section we rely on this result to characterize $\Uqkp$ as the maximal right coideal subalgebra of $\Uqgp$ contained in a fixed-point subalgebra of the automorphism $\phi_q \circ \Ad(\QK{})$, where we recall $\phi_q = \tsat_{q} \circ \Ad(\Parc^{-1})$. 

\subsection{Maximality for QSP subalgebras}

We denote by $\cl{x}$ the specialization of $x$ at $q=1$ whenever well-defined.
Henceforth, we assume that the QSP parameters $(\Parc,\Pars)\in\Parsetc\times\Parsets$ are specializable.
Note that, unlike in \cite{Kol14}, we do not assume that $\cl{\gamma_i} = 1$.

Note that in Section \ref{ss:Liesubalgebra} we defined subalgebras $\fkk \subseteq \fkg$ and $\fkk' = \fkk \cap \fkg'$ depending on a tuple $\bm y \in (\C^\times)^\IS$ satisfying \eqref{y:condition}.
We recall the following characterization of $\Uqkp$.

\begin{theorem}\cite{Kol14, RV20}\label{thm:kolb-maximality}
Let $(\Uqg,\Uqk)$ be a quantum symmetric pair.
Then $\Uqkp\subset\Uqgp$ is a maximal subspace such that $\cl{\Uqkp}=U(\kp)$ with $\cl{\Parc} = \bm y$.
\end{theorem}
	
\begin{proof} 
In the involutive case, we refer to \cite[Rmk.~6.13]{AV22a} for the equality between $\Uqkp$ and Kolb's subalgebra $B_{{\bf c},{\bf s}}$.
In this case, the statement is proved in \cite[Thms.~10.8, 10.11]{Kol14}. 
The proofs of \emph{loc.~cit.~}naturally extend to generalized Satake diagrams, as was pointed out in \cite{RV20}. 
In the same way, the proofs are still valid for general specializable $\Parc \in \Gamma$, the difference being that the procedure yields the universal enveloping algebra of a Lie subalgebra $\fkk' \subset \fkg'$ depending on $\bm y = \cl{\Parc}$. 
\end{proof}

We consider the automorphism of the completion $\COUqg{\sint}{}$ given by\footnote{
There is \emph{a priori} no reason why $\Uqg$ should be $\Ad(\QK{})$-stable. 
Indeed, it is false even for $\g=\fksl_2$, as is easy to show using the expression given in \cite[Lemma~3.8]{DK19}.
}
\begin{equation}\label{eq:def-xi}
\xi = \Ad(\QK{})^{-1} \circ \phi_q^{-1}.
\end{equation}
Then, we consider 
\eq{ \label{Bprimexi:def}
B'_\xi = \{ x \in \Uqgp \,\vert\, (\xi\ten\id)(\Delta(x)) = \Delta(x) \} .
}

\begin{lemma}\label{lem:Bxi:maximal}
The subalgebra $B'_\xi$ is the maximal right coideal subalgebra of $\Uqgp$ contained in $\Uqgp^\xi$. 
\end{lemma}

\begin{proof} 
Applying $\id \ten \veps$ to the relation $(\xi\ten\id)(\Delta(x))=\Delta(x)$, we obtain $B'_\xi \subseteq \Uqgp^\xi$.
For any $x\in B'_\xi$, we have 
\begin{align}
(\xi\ten\id\ten \id) \big( (\Delta\ten\id)(\Delta(x)) \big) &=(\xi\ten\id\ten \id) \big( (\id \ten \Delta)(\Delta(x)) \big)  \\
&=(\id\ten\Delta)\circ (\xi\ten\id)(\Delta(x))\\
&=(\id\ten\Delta)(\Delta(x))=(\Delta\ten\id)(\Delta(x))\,,
\end{align}
\ie $\Delta(x)\in B'_\xi\ten \Uqgp$. 
Let now $C\subseteq \Uqgp^\xi$ be an arbitrary right coideal subalgebra of $\Uqgp$. 
Then $\Delta(C) \subseteq \Uqgp^\xi \ten \Uqgp$ and therefore $C \subseteq B'_\xi$. 
\end{proof}

By \eqref{eq:k-intertwiner}, $\Uqkp \subseteq B'_\xi$.
By the following key result, this is an equality.

\begin{theorem}\label{thm:QSP-maximal-xi}
Let $(\Uqg,\Uqk)$ be a quantum symmetric pair.
Then 
\eq{
\Uqkp= B'_\xi.
}
\end{theorem}

The proof is carried out in Section~\ref{ss:proof-QSP-maximal-xi}. 
It relies on the classical limit of the inclusion $\Uqkp\subseteq B'_\xi$ and properties of the classical limit of $\QK{}$.

\subsection{Classical limit of the quasi K-matrix} \label{ss:classlimitK}

We shall now study the classical limit of $\QK{}$. 
Note that $\cl{\tsat_{q}}=\theta$.
Since $\Ad(\Parc) \in \Aut_{\sf alg}(\Uqg)$ only depends on the values $\Parc_i$, not on the choice of extension of $\Parc: \Qlat \to \bsF^\times$ to a group homomorphism with domain $\Pext$, we obtain $\cl{\phi_q} = \phi$, defined by \eqref{phi:def} in terms of $\bm y = \cl{\Parc}$.

\begin{prop}\label{prop:cl-QK}
Let $(\Uqg,\Uqk)$ be a quantum symmetric pair.
\vspace{1mm}
\begin{enumerate}\itemsep2mm
	\item There is a unique $\sclQK = \sum_{\mu \in \rootsys^+} \sclQK_\mu \in \prod_{\mu \in \rootsys^+} \n^+_{\mu}$ such that 
	$\cl{\QK{}} = \exp(\sclQK)$.
	\item For any $b\in\k'$, $\exp(\ad(\sclQK))(b) = \phi^{-1}(b)$.
	\item Let $\alpha\in\rootsys^+$ such that $\tsat(\alpha)\neq-\alpha$. 
Then	$\sclQK_\alpha = 0$.
\end{enumerate}

\end{prop}
\begin{proof} \hfill
\begin{enumerate}\itemsep2mm
\item 
Since the classical limit of $\RM{}$ and $\RM{\tsat}$ is trivial, it follows from the coproduct identity \eqref{eq:coprod-id} 
that $\cl{\QK{}}\in\prod_{\mu\geqslant0} U(\n^+)_\mu$ is a grouplike element with $\cl{\QK{}}_0=1$. 
Thus its logarithm is well-defined and primitive.
\item 
The classical limit from the intertwining identity \eqref{eq:k-intertwiner} is
\begin{equation}\label{eq:QK-class-intw}
	\exp(\sclQK) \cdot b \cdot \exp(\sclQK)^{-1} = \phi^{-1}(b) \qq \text{for all } b \in \k'
\end{equation}
from which the desired equation readily follows.
\item
First consider an arbitrary element $\la \in \Qlat^+ \backslash (\Qlat^+)^{-\tsat}$. By Theorem~\ref{thm:av-k-mx}, $\QK{\la}=0$ and therefore $\cl{\QK{}}_{\la}=0$. Since $\cl{\QK{}}=\exp(\sclQK)$, it follows that
\eq{ \label{eq:QK-class-intw:3}
	\sum_{k \ge 1} \frac{1}{k!} \sum_{\substack{\beta_1,\ldots,\beta_k \in \rootsys^+ \\ \beta_1 + \ldots + \beta_k = \la}} \sclQK_{\beta_1} \cdots \sclQK_{\beta_k} = 0.
}
In particular, if $\alpha\in\rootsys^+$ is a simple root such that
$\tsat(\alpha)\neq-\alpha$, then  for $\lambda=\alpha$ the identity \eqref{eq:QK-class-intw:3} reduces to $\sclQK_\alpha=0$, as required.

We now proceed by induction on $\hgt(\alpha)$. 
Let $\alpha\in\rootsys^+$ be a positive root such that $\tsat(\alpha)\neq-\alpha$ and $\hgt(\alpha)>1$.
For any decomposition $\alpha=\beta_1+\cdots+\beta_k$ with $k>1$ and $\beta_1,\dots, \beta_k\in\rootsys^+$ we have $\tsat(\beta_j)\neq-\beta_j$ for some $j \in \{1,\ldots,k\}$.
Since $\hgt(\beta_j)<\hgt(\alpha)$, by the induction hypothesis one has $\sclQK_{\beta_j}=0$. 
Thus, for $\lambda=\alpha$ the identity \eqref{eq:QK-class-intw:3} reduces to
\eq{ 
	\sclQK_{\al} = -\sum_{k \ge 2} \frac{1}{k!} \sum_{\substack{\beta_1,\ldots,\beta_k \in \rootsys^+ \\ \beta_1 + \ldots + \beta_k = \al}} \sclQK_{\beta_1} \cdots \sclQK_{\beta_k} = 0
}
and the result follows. \hfill \qedhere
\end{enumerate}
\end{proof}

The following statement is a generalization, to symmetric pairs of non-diagonal type, of the fact that the classical limit of the quasi R-matrix is $1 \in U(\fkg \oplus \fkg)$.

\begin{prop}\label{prop:cl-QK-2}
$\sclQK=0$ if and only if $\phi$ is an involution.
\end{prop}

\begin{proof}
If $\sclQK=0$ then \eqref{eq:QK-class-intw} implies that $\k'$ is contained in the fixed-point subalgebra of $\phi$.
Re-running the proof of Lemma \ref{lem:phi:involution}, we deduce that $\phi$ is an involution.\\

Conversely, suppose $\phi$ is an involution.
Now Lemma \ref{lem:phi:involution} and \eqref{eq:QK-class-intw} imply that $[\exp(\sclQK),b] = 0$ for all $b \in \fkk'$. 
Since $\fkk'$ contains the elements $f_i$ for all $i \in X$ and $f_i + \phi(f_i)$ for all $i \in I \backslash X$, this in turn implies 
\begin{equation}\label{eq:QK-class-intw:4}
\sum_{k \ge 1} \frac{1}{k!} [\sclQK^k, f_i] = \begin{cases} 
0 & \text{if } i \in X, \\ 
\displaystyle\sum_{k \ge 1} \frac{1}{k!} [\phi(f_i),\sclQK^k] & \text{if } i \in \IS\backslash X, 
\end{cases}
\end{equation}
which induces a recursive relation among the components of $\sclQK$ with respect to the principal grading:
\eq{
\sclQK = \sum_{h \ge 1} \sclQK(h), \qq \sclQK(h) \in \bigoplus_{\substack{\la \in \rootsys^+ \\ \hgt(\la) = h}} \fkn^+_\la \subseteq \fkn^+.
}
To specify this relation, for $\bm \ell = (\ell_1,\ldots,\ell_k) \in \Z_{> 0}^k$ we denote $\sclQK(\bm \ell) = \sclQK(\ell_1) \cdots \sclQK(\ell_k)$, so that $\sclQK^k = \sum_{\bm \ell \in \Z_{> 0}^k} \sclQK(\bm \ell)$.
For $m \in \Z$ consider the set of ordered partitions of $m$ in $k$ parts:
\[
\cP_k(m) = \{ \bm \ell = (\ell_1,\ldots,\ell_k) \in \Z_{> 0}^k \, | \, \ell_1 + \ldots + \ell_k = m \}.
\]
Finally, for $i \in \IS \backslash X$, note that $\phi(f_i)$ is homogeneous of degree $m_i = \hgt(w_X(\al_{\tau(i)})) \in \Z_{>0}$.
With this notation, the component of \eqref{eq:QK-class-intw:4} of degree $m-1$ ($m \in \Z_{> 0}$) is given by
\begin{equation}\label{eq:QK-class-intw:5}
\sum_{k \ge 1} \frac{1}{k!} \sum_{\bm \ell \in \cP_k(m)} [\sclQK(\bm \ell), f_i] = \begin{cases} 0 & \text{if } i \in X, \\ \displaystyle \sum_{k \ge 1} \sum_{\bm \ell \in \cP_k(m - m_i - 1)} \frac{1}{k!} [\phi(f_i),\sclQK(\bm \ell)] & \text{if } i \in \IS\backslash X. \end{cases}
\end{equation}
We now show by induction that $\sclQK(m) = 0$ for all $m \in \Z_{> 0}$, yielding $\sclQK = 0$ as desired.
The case $m=1$ follows directly from \eqref{eq:QK-class-intw:5}: the left-hand side equals $[\sclQK(1),f_i]$; the right-hand side equals 0 since $\cP_k(-m_i) = \emptyset$ for all $k$. 
Hence we obtain $[\sclQK(1),f_i]=0$ for all $i \in \IS$ so that $\sclQK(1) =0$ by \cite[Lemma 1.5]{Kac90}.

Now let $m \in \Z_{\ge 2}$ and suppose that $\sclQK(m') = 0$ for all $m' \in \{ 1, \ldots, m-1 \}$.
Since for all $k > 1$ the elements of $\cP_k(m)$ are tuples whose entries are less than $m$, the induction hypothesis implies that the left-hand side of \eqref{eq:QK-class-intw:5} equals its $k=1$ term, namely $[\sclQK(m),f_i]$.
Also, for all $k \ge 1$ the elements of $\cP_k(m-m_i-1)$ are tuples whose entries are less than $m$. 
Hence the right-hand side of \eqref{eq:QK-class-intw:5} vanishes by the induction hypothesis. 
We obtain that $[\sclQK(m),f_i] = 0$ for all $i \in \IS$ and by \cite[Lemma 1.5]{Kac90} we deduce $\sclQK(m) =0$, which completes the proof.
\end{proof}

\begin{example} 
Examples of nonzero elements $\sclQK$ can be obtained by letting $q \to 1$ in quasi K-matrices for symmetric pairs of finite type AIV, found in \cite[(3.43)]{DK19}. 
For instance, for $n=2$, using \cite[Rmk.~6.13, Eqn.~(8.1)]{AV22a} to present the expression in our conventions, we have
\eq{
\QK{} = 
\bigg( \sum_{k \ge 0} \frac{q^{(k-2)k/2}}{[k]_q!} \overline{\Parc_1^k} [E_1,E_2]_q^k \bigg) 
\bigg( \sum_{k \ge 0} \frac{q^{(k-2)k/2}}{[k]_q!} \overline{\Parc_2^k} [E_2,E_1]_q^k \bigg).
}
Here $\overline{\cdot}$ is the unique field automorphism of $\bsF$ which sends $q^{1/m}$ to $q^{-1/m}$ for all $m\in \Z$.
Clearly, $\QK{}$ specializes to 
\eq{
\exp( y_1 [e_1,e_2]) \exp(y_2 [e_2,e_1]) = \exp\big( (y_1-y_2) [e_1,e_2]\big),
}
where we recall that $y_i = \cl{\Parc_i}$.
In this case we obtain $\mathbf{Y} = (y_1-y_2)[e_1,e_2]$. 
Note that this explicit formula is precisely in line with Lemma \ref{lem:phi:involution} and Proposition \ref{prop:cl-QK-2}. \hfill \rmkend 
\end{example}

\subsection{Proof of Theorem~\ref{thm:QSP-maximal-xi}}\label{ss:proof-QSP-maximal-xi}
Since $B'_\xi$ is a right coideal subalgebra which contains $\Uqkp$, it is enough to prove that $\cl{B'_\xi}=U(\kp)$. 
Indeed, in this case Kolb's maximality theorem, see Theorem~\ref{thm:kolb-maximality}, implies $B'_\xi=\Uqkp$. 
We start with the following basic result.

\begin{lemma} \label{lem:Hopfcoideal}
Let $H$ be a Hopf algebra generated over a field $k$ by primitive elements.
Then every coideal subalgebra $H'$ of $H$ is a Hopf subalgebra.	
\end{lemma}

\begin{proof}
Since $H$ is cocommutative, $H'$ is a subbialgebra.
It remains to prove that every subbialgebra $H' \subseteq H$ is preserved by the antipode $S$, which is a consequence from the following argument (see e.g. \cite[Proof of Prop.~4.2]{EE05} and \cite[Prop.~5.2]{FGB05}). 
Let $\eta$, $\veps$, $m^{(n)}$ and $\Delta^{(n)}$ denote the unit, counit, $n$-fold product and $n$-fold coproduct, respectively.
With respect to the convolution product on $\End_k(H)$, the antipode is the inverse of $\id_H$, the neutral element being $\eta \circ \veps$. 
Since any product of primitive elements is annihilated by $(\eta \circ \veps - \id_H)^{\ten n} \circ \Del^{(n)} \in \Hom_k(H,H^{\ten n})$ for sufficiently large	$n$, the following identity is valid in $\End_k(H)$:
\[
S= \sum_{n \ge 0} m^{(n)} \circ (\eta \circ \veps - \id_H)^{\ten n} \circ \Del^{(n)}.
\]
Hence $S(h)$ for all $h \in H'$ can be expressed in terms of the bialgebra maps of $H'$ and the result follows.
\end{proof}

By \cite[Thm.~5.18, Prop.~6.13]{MM65} we conclude that the coideal subalgebra $\cl{B'_\xi}$ of $U(\gp)$ is equal to $U(\f)$ for some Lie subalgebra $\f\subseteq\gp$.
Moreover, since $\cl{\Uqkp}=U(\kp)$ and $\cl{(\Uqgp)^\xi}=U(\gp)^{\cl{\xi}}$, we have
\begin{align}
\kp\subseteq\f\subseteq(\gp)^{\cl{\xi}}.
\end{align}
Note that $\cl{\xi}=\Ad(\cl{\QK{}})\circ\phi$.
By Proposition~\ref{prop:cl-QK}, $\cl{\QK{}}=\exp(\sclQK)$ for some $\sclQK=\sum_{\mu \in \rootsys^+}\sclQK_\mu\in\prod_{\mu \in \rootsys^+}\n^+_\mu$. Thus,
\eq{ \label{pseudofixedpoint:inclusion}
(\gp)^{\cl{\xi}} = \{x\in\gp\,\vert\, \Ad(\cl{\QK{}})(x)=\phi^{-1}(x)\}=\{x\in\gp\,\vert\, \exp(\ad(\sclQK))(x)=\phi^{-1}(x)\}.	
}

\begin{lemma}\label{lem:same-fixed-points}
$\k=\g^{\cl{\xi}}$.	
\end{lemma}

\begin{proof}
Clearly, $\k \subseteq \g^{\cl{\xi}}$.
Denote $\n^+_\theta = \n^+\cap\theta(\n^-) = \n^+ \cap \phi(\n^-)$. 
By \cite[Cor.~3.10]{RV22}, one has the \emph{Iwasawa decomposition}:
\eq{ \label{Iwasawa}
\g=\k\oplus\h^{-\theta}\oplus\n^+_\theta
} 
(as linear spaces). 
Now suppose $x \in \h^{-\theta}\oplus\n^+_\theta$ is fixed by $\cl{\xi}$, \ie
\eq{ \label{lem:same-fixed-points:eq1}
\exp(\ad(\sclQK))(x)=\phi^{-1}(x).
}
Denote by $\b^\pm = \h \oplus \n^\pm$ the standard Borel subalgebras.
Note that $\exp(\ad(\sclQK))(x) \in \prod_{\mu \in \Qlat^+} \b^+_\mu$ whereas $\phi^{-1}(x) \in \b^-$.
Hence \eqref{lem:same-fixed-points:eq1} implies $x \in \h$, so that $x \in \h^{-\theta} = \h^{-\phi}$.
Rewriting \eqref{lem:same-fixed-points:eq1} as $\exp(\sclQK) x = -x \exp(\sclQK)$ and comparing weight-0 components, we obtain $x=0$.
Now $(\h^{-\theta}\oplus\n^+_\theta)^{\cl{\xi}} = \{ 0 \}$ implies the desired result.
\end{proof}

\noindent
Intersecting with $\gp$, we obtain $\kp = (\gp)^{\cl{\xi}}$ so that $\f=\kp$. 
Theorem~\ref{thm:kolb-maximality} now implies Theorem~\ref{thm:QSP-maximal-xi}.

\begin{remark}
	 If $\phi$ is an involution, Lemma~\ref{lem:same-fixed-points} follows immediately from Proposition~\ref{prop:cl-QK-2}. \rmkend
\end{remark}


\section{Tensor K-matrices} \label{ss:tensorK}

In this Section we discuss the main result of the paper: the construction of tensor K-matrices for quantum symmetric pairs $(\Uqg,\Uqk)$ of Kac-Moody type.
Recall that $\Uqk$ is a coideal subalgebra of $\Uqg$, so that full subcategories of $\Mod_{\Uqk}$ are naturally module categories over (subcategories of) $\Mod_{\Uqg}$.
We are interested in solutions $\TKM{}$ of reflection equations such that $\TKM{}$ lies in a completion of $\Uqk \ten \Uqg$, allowing us to consider a braided structure on such a module category that extends the natural braided structure on certain monoidal subcategories of $\Mod(\Uqg)$ such as $\OUqg$.
The tensor K-matrices have a natural factorization in terms of the datum $(R,\psi,K)$, see \eqref{intro:factorization} and the following comments.
In concrete terms, we obtain a tensor K-matrix whose $\Uqk$-leg can be evaluated on a large category of $\Uqk$-modules (if $\theta=\omega$, on $\Mod(\Uqk)$ itself) and whose $\Uqg$-leg can be evaluated on a category between $\OUqg$ and $\OintUqg$, depending on the choice of gauge transformation.

As in the previous sections, we will give a treatment tailored to the QSP setting, involving completions with respect to certain categories of modules. 
The purely algebraic formalism can be read in parallel in Appendix \ref{ss:cylindrical-algebraic}.

\subsection{The QSP weight lattice $\Pext_\theta$}
The involution $\theta: \fkh \to \fkh$ determines several lattices 
in $\fkh$ and $\fkh^*$ equipped with a natural pairing, see \eg ~\cite{BW18a,Wat24}. 
Adopting a slightly different approach, we consider the quotient lattice
\begin{equation}\label{eq:QSP-lattice}
	\Pext_\theta = \Pext / \Pext^{-\theta},
\end{equation}
with canonical projection $\tproj{\cdot}: \Pext \twoheadrightarrow \Pext_\theta$.
The fixed-point sublattice of the $\Z$-linear automorphism $-\theta$ satisfies
\eq{
	\Pext^{-\theta} = \Pext \cap \operatorname{span}_{\tfrac{1}{2}\Z} \{ \la - \theta(\la)\,\vert\, \la\in\Pext \}.
} 
Given the bilinear pairing $\Pext \times \Qvext \to \Z$, define a pairing $\Pext_\theta \times \Qvext \to \tfrac{1}{2} \Z$ as follows: $\zeta(h) = \tfrac{1}{2} \lambda(h+\theta(h))$ for $\zeta = \tproj{\lambda} \in \Pext_\theta$ with $\la \in \Pext$ and $h \in \Qvext$ (note that it is independent of the choice of $\la$).
This restricts to a nondegenerate pairing $\Pext_\theta \times (\Qvext)^\theta \to \Z$.

\begin{remark} \label{rmk:Watanabe}
	Set $\IS_{\sf ns} = \{ i \in \IS \, | \, \theta(\al_i) = -\al_i \}$ so that $\tproj{\al_i} = 0$ if and only if $i \in \IS_{\sf ns}$. 
	Note that $\Pext_\theta$ identifies with a subgroup of the lattice $X^\imath$ considered in \cite[Sec.~3.2]{Wat24} of index $2^{|\IS_{\sf ns}|}$. \hfill \rmkend
\end{remark}

\subsection{QSP weight modules} \label{ss:weightQSPmodules}

Recall that we have
\[
\Uqht= \bsF\langle \Kg{h}\;\vert\; h\in(\Qvext)^\tsat\rangle=\Uqk \cap \Uqh.
\]

\begin{definition}
A $\Uqk$-module $M$ is a (type $\bf 1$) {\em QSP weight module} if, as a 
 $\Uqht$-module,
\begin{equation}
M=\bigoplus_{\zeta\in\Pext_\theta}M_\zeta
\end{equation}
where the {\em QSP weight spaces} are defined by
\begin{equation}\label{QSPweight:decomposition}
	M_\zeta = \{ m \in M \, | \, \forall h \in (\Qvext)^\theta, \, K_h \cdot m = q^{\zeta(h)} m \}.
\end{equation}
We denote the full subcategory of QSP weight modules 
by $\WUqk \subseteq \Mod(\Uqk)$. 
\end{definition}

\begin{remark} \label{rmk:weightspace:dimension}
	Consider the case of a split quantum symmetric pair, \ie $\theta=\omega$. 
	Then $\Pext_\theta = \{ 0 \}$ and $\WUqk = \Mod(\Uqk)$.
	In particular, every $\Uqk$-module $M$ is a QSP weight module which consists of a single QSP weight space: $M=M_0$.
	\hfill \rmkend
\end{remark}

In the following proposition we collect some elementary properties of $\WUqk$.

\begin{prop}\label{prop:weight-QSP-cat}
Let $(\Uqg,\Uqk)$ be a quantum symmetric pair.
\vspace{1mm}
\begin{enumerate}\itemsep2mm 

\item \label{prop:actiononweightspaces} For any $M\in\WUqk$, $\zeta\in\Pext_\theta$, $i\in \IS$ and $j\in X$, one has
\[
E_j \cdot M_\ze \subseteq M_{\ze + \tproj{\al_j}}.
\qquad \qquad 
B_i \cdot M_\ze \subseteq M_{\ze - \tproj{\al_i}},
\]

\item \label{prop:findimreps} 
All finite-dimensional irreducible $\Uqk$-modules are objects in $\WUqk$ up to twisting by an automorphism of $\Uqk$.

\item \label{prop:modulecategory} 
$\WUqk$ is a right module category over the monoidal category $\WUqg$. 
Further, every weight $\Uqg$-module is a QSP weight module under restriction to $\Uqk$.

\item \label{prop:separatespoints} 
The action of $\Uqk$ on $\WUqk$ is faithful, \ie $\Uqk$ embeds as a subalgebra into $\CWUqk{}$, where $\sff_{\theta}:\WUqk\to\vect$ is the forgetful functor.

\end{enumerate}

\end{prop}

\begin{proof} \mbox{}

\begin{enumerate}\itemsep2mm 

\item
Note that $\pars{i} = 0$ if $\theta(\al_i) \ne -\al_i$. 
Therefore, the result follows from the relations 
\eq{ \label{QSPrelations1}
K_h B_i =  q_i^{-\al_i(h)} B_i K_h, \qq \qq K_h E_j = q_j^{\al_j(h)} E_j K_h
}
for any $i \in \IS$, $j \in X$, $h \in (\Qvext)^\tsat$.

\item 
This is a consequence of Theorem \ref{thm:QSP:findemreps} which we prove in the self-contained appendix \ref{ss:findimQSPmodules}.

\item
Since $\Uqk\subset\Uqg$ is a (right) coideal subalgebra, $\Mod(\Uqk)$ is a right module category over $\Mod(\Uqg)$ with monoidal action given by the tensor product, cf.~Section \ref{ss:module-category}. 
In particular, for any $M \in \WUqk$ and $V \in \WUqg$, one has $M\ten V\in\Mod(\Uqk)$. Then,
\[
M \ten V = \bigoplus_{\ze \in \Pext_\theta} (M \ten V)_\ze, \qq (M \ten V)_\ze = \left(\bigoplus_{\mu \in \Pext} M_{\ze-\tproj{\mu}} \ten V_\mu\right).
\]
Let $h \in (\Qvext)^\tsat$ and $\ze\in\Pext_{\theta}$.
Then, $K_h$ acts on $(M \ten V)_\ze$ as multiplication by $q^{\ze(h)}$.
Namely, write  $\ze = \tproj{\la}$ for some $\la \in \Pext$. Then, it is enough to observe that $K_h$ is grouplike and it acts on $M_{\ze-\tproj{\mu}}$ and $V_\mu$ via multiplication by $q^{(\la-\mu)(h)}$ and $q^{\mu(h)}$, respectively.

Finally, note that the restriction of $V$ to $\Uqk$ is given by the action of $V$
on the trivial representation $\Uqk\subset\Uqg\stackrel{\veps}{\longrightarrow}\bsF$, thus yielding a restriction functor $\WUqg\to\WUqk$.

\item
By Proposition~\ref{prop:weight-QSP-cat} \ref{prop:modulecategory}, every module in $\OUqg$ restricts to $\WUqk$. 
This yields a morphism $\CWUqk{}\to\COUqg{}{}$ and a commutative diagram
\begin{equation}
\begin{tikzcd}
	\Uqk\arrow[r,hookrightarrow]\arrow[d]&\Uqg\arrow[d,hookrightarrow]\\
	\CWUqk{}\arrow[r]&\COUqg{}{}
\end{tikzcd}
\end{equation}
where the right vertical arrow is injective by \cite{ATL24b}.
It follows that the map $\Uqk\to\CWUqk{}$ is also injective. \hfill \qedhere
\end{enumerate}

\end{proof}

\subsection{Tensor QSP weight operators}\label{ss:2-QSP-weight}

Let $f\colon\Pext\to\bsF$ be a $\bbZ$-linear functional, self-adjoint with respect to the pairing on $\Pext$. Suppose that $f$ satisfies $f \circ \theta = f$.
Then, $f(\Pext^{-\theta})=0$ and $\WO{f}$ descends to a function $\Pext_\theta \times \Pext \to \bsF$. 
Proceeding as in Section \ref{ss:2-weight}, $\WO{f}$ determines a diagonal operator on any tensor product $M\ten V$ with $M\in\WUqk$ and $V\in\WUqg$, which is natural in both $M$ and $V$. Let
\[
\sff_\theta\boxtimes\sff_\cW\colon\WUqk\boxtimes\WUqg\to\vect
\]
be the 2-fold forgetful functor $\sff_\theta\boxtimes\sff_\cW(M,V)=M\ten V$.
Then, $\WO{f}\in\End(\sff_\theta\boxtimes\sff_{\cW})$.

\subsection{Tensor Drinfeld algebras}\label{ss:2-Drinfeld-algebra}
For any $\beta\in\Qlat^+$, let $\cB_{\beta}$ be a basis of $\Uqnp_{\beta}$
and set $\cB=\bigsqcup_{\beta\in\Qlat^+}\cB_\beta$. 
Consider the \emph{Drinfeld algebra}
\begin{equation}\label{eq:Drt-def}
	\Dr=\left\{\left.\sum_{x\in\cB}c_xx\,\right\vert\, c_x\in\Uqbm\right\}=\prod_{\beta\in\Qlat^+}\Uqbm\Uqnp_{\beta} \qq \supseteq\Uqg.
\end{equation}
One readily checks that $\Dr$ has a unique algebra structure which extends that of $\Uqg$. 
Moreover, every element in $\Dr$ naturally acts on objects in $\OUqg$, since it reduces to a finite sum due to condition \ref{cond:O}. 
By \cite[Sec.~3.1]{ATL24b}, $\Dr$ embeds as a subalgebra in both $\COUqg{}{}$ and $\COUqg{\sint}{}$, where $\FF{}{}:\OUqg\to\vect$ and $\FF{\sint}{}\colon \OintUqg \to\vect$ are the forgetful functors. 

\begin{remark}
By construction, the quasi K-matrix $\QK{}$ is an element of $\Dr$, see Theorem~\ref{thm:av-k-mx}. \hfill\rmkend
\end{remark}

In a similar way, we equip
\begin{align}
\Drt &= \left\{\left.\sum_{x\in\cB}b_x\ten c_xx\,\right\vert\, b_x\in\Uqk\,, \,c_x\in\Uqbm\right\}\\
&= \prod_{\beta\in\Qlat^+}\Uqk\ten\Uqbm\Uqnp_{\beta} \qq \supseteq \Uqk\ten\Uqg
\end{align}
with the unique algebra structure extending that of $\Uqk\ten\Uqg$. 
Every element in $\Drt$ acts in particular on tensor products of the form $M\ten V$ with $M\in\WUqk$ and $V\in\OUqg$. 
The action is natural in $M$ and $V$, thus yielding a morphism of algebras $\Drt\to\End(\sff_\theta\boxtimes\sff)$. 
Proceeding as in \cite[Thm.~3.1]{ATL24b}, one further shows that the action of $\Drt$ is faithful on objects in $\OintUqg$. 
Thus, we obtain a commutative diagram
\begin{equation}\label{eq:Drt-End}
	\begin{tikzcd}
	&\Drt \arrow[dl, hook] \arrow[d, hook] \arrow[dr, hook]&\\
	\End(\FF{}{}\boxtimes\FF{}{}) \arrow[r]& \End(\FF{\theta}{}\boxtimes\FF{}{})  \arrow[r]& \End(\FF{\theta}{}\boxtimes\FF{\sint}{})
	\end{tikzcd}
\end{equation}
where the first and second horizontal arrows are induced by the restriction functor $\OUqg\to\WUqk$ and the inclusion $\OintUqg\hookrightarrow\OUqg$, respectively.\\

Finally, set $\cB_X=\bigsqcup_{\beta\in\Qlat_X^+}\cB_\beta$. We consider the \emph{opposite} algebra
\begin{align}
	\DrX &=\left\{\left.\sum_{x\in\cB_X}a_x\ten c_x\omega(x)\,\right\vert\, a_x\in\UqgX\,, \,c_x\in\UqbpX\right\}\\
	&=\prod_{\beta\in\Qlat_X^+}\UqgX\ten\UqbpX\UqnmX_{-\beta} \qq \supseteq \UqgX\ten\UqgX
\end{align}

We observe that $\DrX$ acts on any tensor product $V\ten W$ with $V\in\Mod(\UqgX)$ and $W\in\POUqg{X}$. 
Namely, by \cite[Thm.~6.2.2]{Lus94}, any object in $\OintUqg$ is completely reducible. 
Since $X\subset \IS$ is a subdiagram of finite type, $W$ decomposes as a (possibly infinite) direct sum of finite-dimensional irreducible $\UqgX$-modules. 
Thus, the action of $\DrX$ on $V\ten W$ is well-defined.
Through the restriction functor $\POUqg{X}\to\cO_\infty^{\sf int}(\UqgX)$ and the inclusion $\UqgX\subset\Uqk$, we obtain an embedding
\begin{equation}\label{eq:DrX-End}
	\begin{tikzcd}
		\DrX\arrow[r,hook] & \End(\FF{\theta}{}\boxtimes\FF{\pint{X}}{}).
	\end{tikzcd}
\end{equation}


\subsection{Tensor K-matrices for quantum symmetric pairs} \label{ss:tensorKmatrices}
We now extend the discussion of Sections \ref{ss:basicKmatrices}-\ref{ss:gauged-kmx} about cylindrical structures and basic K-matrices to reflection structures and tensor K-matrices. 
Again, the datum is a quantum symmetric pair $(\Uqg,\Uqk)$.
In Appendix \ref{ss:reflection-bialgebras} the algebraic formalism is given, generalizing formalisms from \cite{Enr07,Kol20} to make them applicable in the quantized Kac-Moody setting; the corresponding formalism for comodule algebras appeared independently in \cite{LBG23,Lem23}.\\

The objects of interest are as follows: we keep the same twist pair $(\psi,J)$ as in Section \ref{ss:qsp} and replace the basic K-matrix (lying in a completion of $\Uqg$) by a tensor K-matrix (lying in a completion of $\Uqk \ten \Uqg$). 
Instead of the identities \eqrefs{eq:k-intertwiner}{eq:coprod-id} we are interested in the intertwining identity
\begin{equation}\label{eq:2K-intw-QSP} 
\TKM{} \cdot\Delta(b) = (\id\ten\psi)(\Delta(b)) \cdot \TKM{} \qq \text{for all } b \in \Uqk
\end{equation}
and the coproduct identities 
\begin{align}
\label{eq:2K-coprod-QSP-1} (\Delta\ten\id)(\TKM{}) &= R^\psi_{32} \cdot \TKM{13} \cdot R_{23},\\
\label{eq:2K-coprod-QSP-2} (\id\ten\Delta)(\TKM{}) &= J^{-1}_{23} \cdot \TKM{13} \cdot R^\psi_{23} \cdot \TKM{12}.
\end{align}
Such a triple $(\psi,J,\TKM{})$ is called a \emph{reflection structure} on $(\Uqg,\Uqk)$
As before, we obtain a version of \eqref{eq:K-RE} in a completion of a triple tensor product (the proof of Proposition \ref{prop:tensorK:RE} applies here):
\eq{ \label{eq:2K-RE-QSP}	
R^{\psi\psi}_{32} \cdot \TKM{13} \cdot R^\psi_{23} \cdot \TKM{12} = \TKM{12} \cdot  R^\psi_{32} \cdot \TKM{13} \cdot R_{23}.
}

\subsection{The tensor K-matrix associated to the quasi K-matrix}
From the distinguished cylindrical structure $(\psi,J,K)$ constructed in Theorem~\ref{thm:av-k-mx}, we construct our first example of a tensor K-matrix.
Recall the operator $R_{21}^\psi\in\End(\FF{}{}\boxtimes\FF{}{})$ defined in Proposition~\ref{prop:aux-R-matrices} \ref{Rpsi21:def}. 

\begin{theorem} \label{thm:reflectionalgebra:QSP}
Let $(\Uqg,\Uqk)$ be a quantum symmetric pair and set 
\begin{equation}
	\psi=\phi_q^{-1}, \qq J=R_\theta, \qq \KM{}=\QK{}.
\end{equation}
Then the operator 
\begin{equation}\label{eq:tensor-K-def}
	\TKM{} \coloneqq R_{21}^\psi\cdot (1\ten \KM{}) \cdot R \in \End(\FF{}{}\boxtimes\FF{}{})
\end{equation} 
has the following properties.
\vspace{1mm}
\begin{enumerate}\itemsep2mm
\item \label{axioms:O} It satisfies the intertwining identity \eqref{eq:2K-intw-QSP} in $\End(\FF{}{}\boxtimes\FF{}{})$ and the coproduct identities \eqref{eq:2K-coprod-QSP-1} and \eqref{eq:2K-coprod-QSP-2} in $\End(\FF{}{}\boxtimes\FF{}{}\boxtimes\FF{}{})$.
\item \label{support:WthetaW} There are elements $\TQK{X}\in\DrX$ and $\TQK{}\in\Drt$ such that 
\[
\TKM{}=\WO{\id+\tsat}\cdot\TQK{X}\cdot\TQK{}.
\]
In particular, $\TKM{}$ admits a canonical lift\footnote{
By construction, $\TKM{}$ acts on any tensor product $V\ten W$ with $V,W\in\OUqg$. 
This means that the action of $\TKM{}$ is \emph{extended} to any tensor product $M\ten W$ with $M\in\WUqk$ but $W\in\POUqg{X} \subseteq \OUqg$.} in $\End(\FF{\theta}{}\boxtimes\FF{\pint{X}}{})$, the identity \eqref{eq:2K-intw-QSP} is valid in $\End(\FF{\theta}{}\boxtimes\FF{\pint{X}}{})$, and\footnote{Since $\WUqk$ is a module category over $\OUqg$, the coproduct $\Delta$ naturally lifts to a morphism $\End(\FF{\theta}{})\to\End(\FF{\theta}{}\boxtimes\FF{}{})$, cf.~Section \ref{ss:R-matrix}.} the identities \eqref{eq:2K-coprod-QSP-1} and \eqref{eq:2K-coprod-QSP-2}, as well as \eqref{eq:2K-RE-QSP}, are valid in $\End(\FF{\theta}{}\boxtimes\FF{\pint{X}}{}\boxtimes\FF{\pint{X}}{})$.
\end{enumerate}
\end{theorem}


\begin{proof}
One observes that all factors on either side of \eqref{eq:2K-intw-QSP} are well-defined elements of $\End(\FF{}{}\boxtimes\FF{}{})$.
Similarly, all factors on either side of \eqref{eq:2K-coprod-QSP-1} and \eqref{eq:2K-coprod-QSP-2} lie in $\End(\FF{}{}\boxtimes\FF{}{}\boxtimes\FF{}{})$.
Hence, \ref{axioms:O} follows as in the proof of Proposition \ref{prop:cylindricaltoreflection}. 


Since $\theta|_\fkh$ is involutive, the weight operator $\WO{\id+\theta}\in\End(\FF{\cW}{}\boxtimes\FF{\cW}{})$ has a canonical lift in $\End(\FF{\theta}{}\boxtimes\FF{\cW}{})$, see Section \ref{ss:2-QSP-weight}.
We have to prove that 
\[
\WO{-\id-\theta}\cdot\, \TKM{} = \TQK{X}\cdot\TQK{}
\]
for some $\TQK{X}\in\DrX$ and $\TQK{}\in\Drt$. Then, by \eqref{eq:Drt-End} and \eqref{eq:DrX-End}, $\TKM{}$ admits a suitable canonical lift to $\End(\FF{\theta}{}\boxtimes\FF{\pint{X}}{})$.
The operators $ \TQK{X}$ and $\TQK{}$ are obtained as follows.
The universal R-matrix is equipped with the factorization $R=\WO{\id}\cdot\,\Theta$, where 
\[
\Theta=\sum_{\beta\in\Qlat^+}\Theta_\beta\in\prod_{\beta\in\Qlat^+}
\Uqnm_{-\beta}\ten\Uqnp_\beta\,
\]
see Section \ref{ss:R-matrix}. 
By Lemma~\ref{lem:twistpair:properties} \ref{W:psitwist}, 
one readily checks that
\[
\WO{\id}_{V{\Wpsi}}=\WO{\theta}_{VW}
\]
for any $V,W\in\WUqg$. 
Hence, $R_{21}^{\psi} = \WO{\theta}\cdot\,\Theta_{21}^{\psi}$.
Thus, we have
\begin{equation}
	\WO{-\id-\theta}\cdot\,R_{21}^{\psi} \cdot (1\ten \KM{}) \cdot R = \Ad(\WO{-\id})(\Theta^{\psi}_{21}) \cdot \Ad(\WO{-\id})(1 \ten \KM{}) \cdot \Theta.
\end{equation}
Let $\RM{X}$ be the universal R-matrix of $\UqgX$ with the corresponding factorization $\RM{X}=\WO{\id}\cdot\Theta_X$.
By \cite[Prop.~4.3]{AV22a},
\begin{align}
\wt{\Theta} = \Theta_X^{-1}\cdot\Theta\in\prod_{\beta\in\Qlat^+\setminus\Qlat_X^+}
\Uqnm_{-\beta}\ten\Uqnp_\beta.
\end{align}
Since $\psi$ is the identity on $\UqgX$, we have
\begin{equation}\label{eq:RX-decomposition}
	\Theta_{21}^{\psi}=\Theta_{X,21}\cdot\wt{\Theta}_{21}^{\psi}\,.
\end{equation}
Set 
\begin{align}
\TQK{X}=\Ad(\WO{-\id})(\Theta_{X,21})
\quad\mbox{and}\quad
\TQK{}=\Ad(\WO{-\id})(\wt{\Theta}^{\psi}_{21}) \cdot \Ad(\WO{-\id})(1 \ten \KM{}) \cdot \Theta.
\end{align}
Clearly, $\TQK{X}\in\DrX$. It remains to prove that $\TQK{}\in\Drt$. Set
\[
\Dr^{(2)} = \left\{\left.\sum_{x\in\cB}a_x\ten c_xx\,\right\vert\, a_x\in\Uqg\,, \,c_x\in\Uqbm\right\} \supset \Drt.
\]
We claim that $\TQK{}\in\Dr^{(2)}$. 
Clearly, $\Ad(\WO{-\id})(1 \ten \KM{}) \cdot \Theta\in\Dr^{(2)}$.  
Then, it is enough to observe that $-\theta$ preserves $\Qlat^+\setminus\Qlat_X^+$. 
Therefore,
\[
\wt{\Theta}_{21}^{\psi} \in \prod_{\beta\in\Qlat^+\setminus\Qlat_X^+} \Uqnp_{\beta} \ten \Uqnp_{\beta} \subset \Dr^{(2)}.
\]
Finally, set $\xi = \Ad(\KM{})^{-1}\circ\psi$ as in \eqref{eq:def-xi}. 
By Theorem~\ref{thm:QSP-maximal-xi},
\[
\Uqkp= \{ u \in \Uqgp \,\vert\, (\xi \ten \id)(\Delta(u)) = \Delta(u)\}.
\]

It now remains to prove that $\TQK{}\in\Drt$. 
Consider the following equation for an element $\Pi \in \End(\FF{}{}\boxtimes\FF{}{})$:
\begin{equation}\label{eq:xi-cond}
\big( (\xi \ten \id \ten \id) \circ (\Delta \ten \id) \big) (\Pi) = (\Delta \ten \id)(\Pi).
\end{equation}
Note that $\TQK{} \in \Drt$ if and only if \eqref{eq:xi-cond} holds with $\Pi = \TQK{}$.
Also, \eqref{eq:xi-cond} is automatically satisfied for $\Pi = \TQK{X}$, since $\xi$ is the identity on $\Uqk$. 
Moreover, it is also satisfied for $\Pi = \WO{\id+\theta}$, since

\vspace{1mm}
\begin{itemize}\itemsep2mm

\item 
the coproduct of $\WO{\id+\theta}$ is provided by \eqref{C:coproduct};

\item 
for any $V,W\in\WUqg$, one has $\WO{\id+\theta}_{{\Vpsi}W} = \WO{\id+\theta}_{VW}$;

\item 
$\WO{\id+\theta}$ commutes with $\KM{} = \QK{}$ since the latter is supported on $(\Qlat^+)^{-\theta}$, see Theorem \ref{thm:av-k-mx}.

\end{itemize} 
\vspace{1.5mm}

Thus, $\TQK{}\in\Drt$ if and only if $\Pi = \TKM{}$ satisfies \eqref{eq:xi-cond}, and the latter follows straightforwardly from the generalized reflection equation for $\KM{}$:
\begin{align}
\hspace{15pt} &\hspace{-15pt} \big( (\xi\ten\id\ten\id) \circ (\Delta\ten\id) \big) (\TKM{}) =\\
&\stackrel{\eqref{eq:2K-coprod-QSP-1}}{=} (\xi \ten \id \ten \id)(R^{\psi}_{32} \cdot R^{\psi}_{31} \cdot (1 \ten 1 \ten \KM{}) \cdot R_{13} \cdot R_{23} ) \\
&\stackrel{\eqref{eq:def-xi}}{=} R^{\psi}_{32}\cdot (\KM{}^{-1} \ten 1 \ten 1) \cdot R^{\psi\psi}_{31}\cdot (1 \ten 1 \ten \KM{}) \cdot R^{\psi}_{13}\cdot (\KM{}\ten 1\ten 1) \cdot R_{23} \\
&\stackrel{\eqref{eq:K-RE}}{=} R^{\psi}_{32} \cdot R^{\psi}_{31} \cdot (1 \ten 1 \ten \KM{}) \cdot R_{13} \cdot R_{23} \\
&\stackrel{\eqref{eq:2K-coprod-QSP-1}}{=} (\Delta \ten \id)(\TKM{}).
\end{align}
This completes the proof of 
\ref{support:WthetaW}.
\end{proof}

\subsection{Gauged tensor K-matrices}

The construction of the tensor K-matrix is compatible with gauge transformations (see also Appendices \ref{ss:reflection-bialgebras} and \ref{ss:gauge}).
This will be crucial in Section \ref{ss:spectraltensorK}, where the tensor K-matrix built from the quasi K-matrix is not always suitable.
Recall that in Section \ref{ss:gauged-kmx} we fixed
\eq{
(\psi,J,K) = (\phi_q^{-1},R_\theta,\QK{}).
}
Also, recall from Proposition \ref{prop:gauge-k-matrix} the group $\Gg_Z$ of invertible elements $\gau\in \CPOUqg{Z}{}$ such that $\Ad(\gau)$ preserves $\Uqg\subset\CPOUqg{Z}{}$.
From Theorem~\ref{thm:reflectionalgebra:QSP} we immediately obtain the following analogue for tensor K-matrices. 

\begin{prop}\label{prop:gauge-tensor-k-matrix}
Let $(\Uqg,\Uqk)$ be a quantum symmetric pair and let $Z \subseteq \IS$ be any subset.
For any $\gau\in\Gg_Z$, set
\begin{equation}
\psi_\gau = \Ad(\gau) \circ \psi, \qq J_\gau = (\gau\ten\gau) \cdot J \cdot \Delta(\gau)^{-1}, \qq \KM{\gau} = \gau \cdot \KM{}
\end{equation}
and 
\begin{equation}\label{eq:tensor-K-def-gauge}
\TKM{\gau} = (1 \ten \gau) \cdot \TKM{} = R_{21}^{\psi_\gau} \cdot (1\ten \KM{\gau}) \cdot R \in\End(\FF{}{}\boxtimes\FF{\pint{Z}}{}).
\end{equation} 
Then $\TKM{\gau}$ admits a canonical lift in $\End(\FF{\theta}{}\boxtimes\FF{\pint{X \cup Z}}{})$ which satisfies the identities \eqref{eq:2K-intw-QSP} in $\End(\FF{\theta}{}\boxtimes\FF{\pint{X \cup Z}}{})$ and \eqref{eq:2K-coprod-QSP-1}, \eqref{eq:2K-coprod-QSP-2} and \eqref{eq:2K-RE-QSP} in $\End(\FF{\theta}{}\boxtimes\FF{\pint{X \cup Z}}{}\boxtimes\FF{\pint{X \cup Z}}{})$.
\end{prop}


\subsection{Boundary bimodule structure} \label{ss:boundarybimodulestructure}
In Appendix \ref{ss:bdrybimodcat}, we describe a categorical framework which encodes the defining properties of the tensor K-matrix $\TKM{}$.
In particular, in Section \ref{ss:boundary-cat}, we introduce the notion of a \emph{boundary bimodule category}, which is a generalization of that of braided module category, see \eg \cite{Enr07, Bro13, Kol20}.
These generalized notions are tailored to encode the datum of the auxiliary R-matrix $R^\psi$, the cylindrical structure $(\psi, J, K)$, and the corresponding tensor K-matrix $\TKM{}$. 
Hence, we immediately obtain the following.	

\begin{corollary}\label{cor:refl-str}
Let $(\Uqg,\Uqk)$ be a quantum symmetric pair and recall the distinguished cylindrical structure $(\psi,J,K) = (\phi_q^{-1},R_\theta,\QK{})$.
\vspace{1mm}
\begin{enumerate}\itemsep2mm

\item 
The category $\WUqk$ has the following natural bimodule structure over $\OUqg$.

\vspace{1.5mm}
\begin{itemize}\itemsep2mm

\item 
The right action is given by $M \monactp V = M \ten V$, with trivial associativity constraint, for any $M \in \WUqk$ and $V \in \OUqg$.

\item 
The left action is given by $M \monactm V = M \ten \Vpsi$, with associativity constraint
\[
M \monactm (V \ten W) \to (M \monactm W) \monactm V
\]
induced by $J_{VW}$, for any $M \in \WUqk$ and $V, W \in \OUqg$.

\item 
The commutativity constraint
\[
(M\monactp V)\monactm W\to (M\monactm W)\monactp V
\]
is induced by $R^\psi_{VW}$,  where $M\in\WUqk$ and $V, W\in\OUqg$.
\end{itemize}

\item 
The tensor K-matrix $\TKM{}$ defined in Theorem~\ref{thm:reflectionalgebra:QSP} is a boundary structure on the bimodule category $\WUqk$ over $\POUqg{X}$.

\end{enumerate}

Moreover, for any $Z \subseteq \IS$, the group $\Gg_Z$ acts simultaneously on both the bimodule and the boundary structures of $\WUqk$, producing structures over $\POUqg{X \cup Z}$.

\end{corollary}

%
%


\section{Trigonometric tensor K-matrices} \label{ss:trigtensorK}
In \cite{AV22b} we showed that, for $\fkg$ of untwisted affine type and for suitable twist automorphisms $\psi$, the universal K-matrix $\KM{}$ gives rise to a formal Laurent series with a well-defined action on finite-dimensional modules of $\Uqg$. 
Here we consider similar questions for the tensor K-matrix $\TKM{}$.

\subsection{Quantum loop algebras and finite-dimensional modules}

Let $\g$ be an untwisted affine Kac-Moody algebra. 
By \cite{Kac90}, $\g'$ is a central extension of the loop algebra $\fkL=\gfin[t,t^{-1}]$, where $\gfin = \g_{\IS \backslash \{ 0 \}} \subset\g$ is a distinguished finite-dimensional Lie subalgebra with the node $0 \in \IS$ chosen as in \cite{Kac90}. 
Let $\fdUqL$ be the category of finite-dimensional (type $\bf 1$) $\Uqgp$-modules and $\FF{\fdUqL}{n} \colon \fdUqL^{\boxtimes n} \to \vect$ the corresponding $n$-fold forgetful functor.

Let $\del = \sum_{i \in \IS} a_i \al_i$ be the minimal imaginary root, \ie the unique element of $\Ker(A) \cap \Qlat^+$ with setwise coprime coefficients, and set $K_\del = \prod_{i \in \IS} K_i^{a_i} \in \Uqgp$. 
The element $K_\del$ is central and acts trivially on any module in $\fdUqL$. 
Therefore, $\fdUqL$ is equivalent to the category of finite-dimensional (type $\bf 1$) modules over the quotient Hopf algebra $\UqLg=\Uqgp/(K_\del-1)$, which is commonly referred to as the \emph{quantum loop algebra}.

\subsection{Trigonometric R-matrices} \label{ss:UqL:Rmatrices} 
The references for this subsection are \cite{Dri86, FR92, KS95}.
Let $z$ be a formal variable and consider the $\tau$-minimal grading shift $\Sigma_z: \UqLg \to \UqLg[z,z^{-1}]$ defined by 
\[
\Sigma_z(E_i) = z^{s_i}E_i,\quad \Sigma_z(F_i) = z^{-s_i} F_i\,\quad\mbox{and}\quad \Sigma_z(K_i)=K_i
\] 
where $s_0 = s_{\tau(0)} = 1$ and $s_i=0$ if $i \notin \{ 0, \tau(0) \}$.
For any $\UqLg$-module $V$, we consider the modules $V(z)=V\ten\bsF(z)$ and 
$\Lfml{V}{z}=V\ten\Lfml{\bsF}{z}$ with action given by the pullback through
$\Sigma_z$ and extension of scalars.\\

Let $R\in\End(\FF{}{}\boxtimes\FF{}{})$ be the universal R-matrix of $\Uqg$, see Section \ref{ss:R-matrix}. Then,
\eq{ \label{spectralR:def}
R(z) = (\id \ten \Sigma_z)(R) = (\Sigma_{z^{-1}} \ten \id)(R) \in \fml{\End(\FF{\fdUqL}{} \boxtimes \FF{\fdUqL}{})}{z}
}
that is $R(z)$ is a formal power series in $z$ whose coefficients have a well-defined and natural action on any tensor product in $\fdUqL$. Moreover, 
 $R(z)$ satisfies the \emph{spectral Yang-Baxter equation}
\eq{ \label{YBE:spec}
R(w)_{12} \cdot R(wz)_{13} \cdot R(z)_{23} = R(z)_{23} \cdot R(wz)_{13} \cdot R(w)_{12}
}
in $\fml{\End(\FF{\fdUqL}{3})}{w,z}$. 
Thus, for any $U,V,W \in \fdUqL$, one obtains a matrix-valued formal series satisfying \eqref{YBE:spec} in $\fml{\End(U \ten V \ten W)}{w,z}$.

By \cite[Sec.\ 4.2]{KS95} and \cite[Thm.\ 3]{Cha02}, when $V, W \in \fdUqL$ are irreducible, the tensor product $V\ten W$ is generically irreducible, \ie $V \ten \Lfml{W}{z}$ remains irreducible as a module over $\Lfml{\UqLg}{z}$. 
This yields a factorization, uniquely defined up to a factor in $\bsF^\times$, of the form
\eq{
	R_{VW}(z) = f_{VW}(z) \cdot \rRM{VW}{z},
}
where $f_{VW}(z) \in \Lfml{\bsF}{z}$ and $\rRM{VW}{z} \in \End(V \ten W)[z]$ is a non-vanishing operator, commonly referred to as a \emph{trigonometric R-matrix}.

Since irreducible modules are highest weight and the R-matrix is weight zero, it is natural to normalize $\rRM{VW}{z}$ such that the tensor product of the highest weight vectors has eigenvalue $1$. 
This yields a non-vanishing operator $\rnRM{VW}{z}$ satisfying the unitarity relation
\begin{equation}\label{R:unitarity}
	\rnRM{VW}{z}^{-1}=(1\,2)\circ\rnRM{WV}{z^{-1}}\circ(1\,2).
\end{equation}
Moreover, if  $V \ten W(\ze)$ is irreducible for some $\ze \in \bsF^\times$, then $\rnRM{VW}{\ze}$ is well-defined and invertible.

\subsection{Quantum affine symmetric pairs and categories of representations}
Let $\theta$ be an automorphism of the second kind of $\fkg$ of the form \eqref{theta:def}, $\theta_q$ the corresponding automorphism of $\Uqg$ defined in \eqref{thetaq:def}, and $\Uqk\subseteq\Uqg$ the corresponding QSP subalgebra. 

Since $\del$ is fixed by Weyl group elements and by diagram automorphisms, one has $\theta(\del)=-\del$. 
In particular, $[\del]_\theta = [0]_\theta$, where $[\cdot]_\theta$ is the projection to the quotient lattice $\Pext_{\theta}$, see \eqref{eq:QSP-lattice}.
Therefore, $\theta_q$ descends to an automorphism of $\UqLg$, and $\Uqk$ embeds in $\UqLg$ as a coideal subalgebra. 
We obtain the following analogue of Proposition \ref{prop:weight-QSP-cat} \ref{prop:modulecategory}. 

\begin{lemma}
The category $\WUqk$ of QSP weight modules is a right module category over the category $\fdUqL$ of finite-dimensional  (type $\bf 1$) $\UqLg$-modules. 
\end{lemma}

By Proposition \ref{prop:weight-QSP-cat} \ref{prop:findimreps}, up to twisting, irreducible finite-dimensional $\Uqk$-modules are objects in $\WUqk$.
Now denote by $\fdUqk$ the full subcategory of finite-dimensional $\Uqk$-modules in $\WUqk$.
Clearly, it is a right module category over $\fdUqL$.

\subsection{Spectral tensor K-matrices} \label{ss:spectraltensorK}

In order to study grading-shifted tensor K-matrices, we need to restrict the gauge $\gau$ to a subset of $\Gg_{X \cup Y}$, for some $Y \subseteq \IS \backslash \{ 0,\tau(0) \}$.
Recall from Section \ref{ss:qtheta} the braid group operator $S_X$ defined in terms of longest elements of the finite Weyl group $W_X$.
Let $\Gg_{\theta,\Parc}$ be the set of gauge transformations of the form $\gau = S_Y^{-1} S_X \bmb^{-1}$ where

\vspace{1mm}
\begin{itemize}\itemsep2mm
\item $Y \subseteq \IS \backslash \{ 0,\tau(0) \}$;
\item $\bmb\colon \Pext \to \bsF^\times$ satisfies $\bmb(\delta) = \Parc(\delta)$.
\end{itemize}
\vspace{1mm}

Recall the distinguished cylindrical structure $ (\phi_q^{-1},R_\theta,\QK{})$ where the universal K-matrix is given by the quasi K-matrix.
Whenever $0\in X$, we cannot use this cylindrical structure, the key obstacle being Lemma \ref{lem:Sigmaz:properties} \ref{spectralinversion} below.
Instead, we fix $\gau\in\Gg_{\theta,\Parc}$ and the corresponding transformed cylindrical structures.
\eq{ \label{cylindricalstructure:affine}
	(\psi, J, \KM{}) = ( \Ad(\gau) \circ \phi_q^{-1}, (\gau\ten\gau) \cdot R_\theta \cdot \Delta(\gau)^{-1}, \gau \cdot \QK{}).
}
Note indeed that $1 \in \Gg_{\theta,\Parc}$ if and only if $0\not\in X$.

\begin{lemma} \label{lem:Sigmaz:properties}
The following properties are satisfied by the cylindrical structure \eqref{cylindricalstructure:affine}.
\vspace{1mm}
\begin{enumerate}\itemsep2mm
\item \label{spectralinversion}
$\Sigma_z \circ \psi = \psi \circ \Sigma_{z^{-1}}$.
\item 
$(\Sigma_z \ten \id)(R^{\psi}) = R(z)^{\psi}$.
\item 
$(\Sigma_z \ten \id)(J) = J$.
\item 
$\KM{}(z) := \Sigma_z(\KM{}) \in \Lfml{\End(\FF{\fdUqL}{})}{z}$. 
If $0\not\in X$, $\KM{}(z) \in \fml{\End(\FF{\fdUqL}{})}{z}$.
\end{enumerate}
\end{lemma}

\begin{proof} \mbox \\
\begin{enumerate}\itemsep2mm
\item 
It is enough to observe that $\Sigma_z$ is the $\tau$-minimal grading shift.

\item 
This follows from \ref{spectralinversion} and \eqref{spectralR:def}.

\item 
Note that $\gau = S_Y^{-1} S_X \bmb^{-1}$ and $R_\theta = (\bt{\tsat}\ten\bt{\tsat})^{-1} \cdot \Delta(\bt{\tsat})$ by \eqref{Rtheta:factorization}.
Considering \eqref{braidtheta:def}, it follows that the Drinfeld twist $J = (\gau\ten\gau)\cdot R_\theta \cdot \Delta(\gau)^{-1}$ is given by a Cartan correction of the quasi R-matrix $\Theta_Y$ of $U_q(\g_Y)$. 
The result follows, since $0,\tau(0)\not\in Y$. 

\item 
This is readily verified by direct inspection, see \eg \cite[Sec.~4]{AV22b}. \hfill \qedhere
\end{enumerate}
\end{proof}

From Proposition \ref{prop:gauge-tensor-k-matrix} we recall the tensor K-matrix associated to the cylindrical structure \eqref{cylindricalstructure:affine}, which in the current notation reads:
\eq{
\TKM{} = R_{21}^{\psi} \cdot (1\ten \KM{}) \cdot R \in\End(\FF{}{}\boxtimes\FF{X\cup Y}{}).
} 
Let $\Del_z = (\id \ten \Sigma_z) \circ \Del$ be the shifted coproduct. 
We have the following result for the shifted tensor K-matrix $\TKM{}$.

\begin{theorem} \label{thm:spectraltensorK}
Set 
\eq{  \label{tensorK:grading}
\TKM{}(z) =  (\id \ten \Sigma_z)(\TKM{}) = R(z)_{21}^{\psi} \cdot (1 \ten \KM{}(z)) \cdot R(z)
}
in $\Lfml{\End(\FF{\fdUqL}{}\boxtimes\FF{\fdUqL}{})}{z}$.
Then $\TKM{}(z)$ has a canonical lift in $\Lfml{\End(\FF{\theta}{}\boxtimes\FF{\fdUqL}{})}{z}$, satisfying:
\begin{itemize}\itemsep2mm
\item the following intertwining identity in $\Lfml{\End(\FF{\theta}{}\boxtimes\FF{\fdUqL}{})}{z}$:
\eq{ \label{formalK:intw}
\TKM{}(z) \cdot \Delta_z(b) = (\id \ten \psi)(\Delta_{z^{-1}}(b)) \cdot \TKM{}(z)
\qquad \text{for all }b \in \Uqk
}
\item the coproduct identities
\begin{align}
(\Delta_{z_1} \ten \id)(\TKM{}(z_2)) &= R(z_1z_2)^{\psi}_{32} \cdot \TKM{}(z_2)_{13} \cdot R(\tfrac{z_2}{z_1})_{23}, \\
(\id \ten \Delta_{z_2/z_1})(\TKM{}(z_1)) &= J_{23}^{-1} \cdot \TKM{}(z_2)_{13} \cdot R(z_1z_2)^\psi_{23} \cdot \TKM{}(z_1)_{12},
\end{align}
in $\Lfml{\End(\FF{\theta}{}\boxtimes\FF{\fdUqL}{}\boxtimes\FF{\fdUqL}{})}{z_1, z_2/z_1}$;
\item the generalized spectral reflection equation
\begin{align}
R(\tfrac{z_2}{z_1})^{\psi \psi}_{32} \cdot \TKM{}(z_2)_{13} \cdot R(z_1z_2)^{\psi}_{23} \cdot \TKM{}(z_1)_{12} = \TKM{}(z_1)_{12} \cdot R(z_1z_2)^{\psi}_{32} \cdot  \TKM{}(z_2)_{13} \cdot R(\tfrac{z_2}{z_1})_{23}
\end{align}
in $\Lfml{\End(\FF{\theta}{}\boxtimes\FF{\fdUqL}{}\boxtimes\FF{\fdUqL}{})}{z_1, z_2/z_1}$.
\end{itemize}
\vspace{1.5mm}
Moreover, if $0\not\in X$, then $\TKM{}(z)\in\fml{\End(\FF{\theta}{}\boxtimes\FF{\fdUqL}{})}{z}$.
\end{theorem}

\begin{proof}
It is enough to observe that $S_X, S_Y$ are well-defined operators on finite-dimensional $\UqLg$-modules, $\gau$ acts on $V(z)$ through Laurent polynomials, and therefore $\TKM{}(z)\in\Lfml{\End(\FF{\fdUqL}{}\boxtimes\FF{\fdUqL}{})}{z}$.
Then, the result readily follows from Theorem~\ref{thm:reflectionalgebra:QSP}.
\end{proof}

\subsection{Trigonometric tensor K-matrices} \label{ss:trigtensorK:sub}

The construction of the spectral tensor K-matrix in Theorem \ref{thm:spectraltensorK} implies the following intermediate result.

\begin{lemma}
Let $M\in\fdUqk$ and $V\in\fdUqL$. 
There exists a QSP intertwiner
\[
\rTKM{MV}{z} \colon M \ten V(z) \to M \ten V^{\psi}(z^{-1}).
\] 
\end{lemma}

\begin{proof}
This is a standard argument, see \eg \cite[Lemma~5.1.1]{AV22b}.
Namely, such an intertwiner $\rTKM{MV}{z}$ by definition satisfies a finite linear system.
Owing to \eqref{formalK:intw}, this system is consistent.
Since the system is defined over $\C(z)$, it has a solution defined over $\C(z)$.
\end{proof}

At present, in general it is unclear if $\rTKM{MV}{z}$ is obtained by normalizing the action of $\TKM{}(z)$ on $M \ten \Lfml{V}{z}$. 
When $M$ and $V$ are irreducible, it would follow automatically if the tensor product $M \ten V$ is (generically) irreducible over $\Uqk$.
The description of such modules is currently out of reach.

Instead, by relying on the construction of the trigonometric K-matrices carried out in \cite{AV22b}, we can prove the appearance of trigonometric tensor K-matrices when $M\in\fdUqk$ is the restriction of an irreducible $\UqLg$-module. 
Namely, in \cite[Thm.~5.1.2 (2)]{AV22b}, we proved that, for any irreducible $V\in\fdUqL$, there exist a formal Laurent series $g_V(z)\in\Lfml{\bsF}{z}$ and a non-vanishing operator-valued polynomial $\rKM{V}{z}\in\End(V)[z]$, uniquely defined up to a scalar in $\bsF^\times$, such that
\begin{equation}\label{eq:av-trig-k}
\KM{V}(z)=g_V(z)\cdot\rKM{V}{z}.
\end{equation}
Then, we have the following.

\begin{theorem} \label{thm:trigtensorK:1}
Let $M\in\fdUqk$ be the restriction of an irreducible object in $\fdUqL$. 
For any irreducible $V\in\fdUqL$, set
\begin{equation}
\rTKM{MV}{z} = \rnRM{V^\psi M}{z}_{21} \cdot (\id_M \ten \rKM{V}{z}) \cdot \rnRM{MV}{z}
\end{equation}
in $\End(M \ten V)(z)$, where $\rnRM{}{z}$ denotes the unitary R-matrix from Section~\ref{ss:UqL:Rmatrices}.
\vspace{1mm}
\begin{enumerate}\itemsep2mm
\item
The operator-valued rational function $\rTKM{MV}{z}$ is non-vanishing and there exist a formal Laurent series $h_{MV}(z)\in\Lfml{\bsF}{z}$ such that
\begin{equation}\label{eq:trig-tensor-K}
\TKM{MV}(z)=h_{MV}(z)\cdot\rTKM{MV}{z}.
\end{equation}

\item 
For any irreducible $V,W\in\fdUqL$, the generalized spectral reflection equation hold
\eq{ \label{spectralRE:tensor}
\begin{aligned}
& {\rnRM{{\Wpsi} {\Vpsi}}{\tfrac{z_2}{z_1}}}_{32} \cdot  \rTKM{MW}{z_2}_{13} \cdot \rnRM{{\Vpsi} W}{z_1z_2}_{23} \cdot \rTKM{MV}{z_1}_{12} = \\
& \qq = \rTKM{MV}{z_1}_{12} \cdot \rnRM{{\Wpsi} V}{z_1z_2}_{32} \cdot \rTKM{MW}{z_2}_{13} \cdot \rnRM{VW}{\tfrac{z_2}{z_1}}_{23}.
\end{aligned}
}

\end{enumerate}
\end{theorem}

\begin{proof}
The result follows immediately from \eqref{tensorK:grading} and \cite[Thm.~5.1.2]{AV22b}.
\end{proof}

\begin{remarks}\hfill
\vspace{1mm}
\begin{enumerate}\itemsep2mm
\item 
Consider the q-Onsager algebra, which is the QSP subalgebra $\Uqk \subset U_q(\wh\fksl_2)$ with $\theta=\omega$, see \cite{BK05}. 
In \cite{IT09}, Ito and Terwilliger proved that, up to isomorphism, every finite-dimensional irreducible module over $\Uqk$ is the restriction of a $U_q(\wh\fksl_2)$-module. 
This implies that, in this case, Theorem~\ref{thm:trigtensorK:1} is valid for every irreducible module in $\fdUqk$.

\item 
For more general modules in $M\in\fdUqk$, Theorem~\ref{thm:trigtensorK:1} follows at once whenever the space of QSP intertwiners
\[
M \ten \Lfml{V}{z} \to M \ten \Lfml{V^{\psi}}{z^{-1}}
\]
is proved to be one-dimensional. 
\hfill\rmkend
\end{enumerate}
\end{remarks}

\subsection{Unitary tensor K-matrices} \label{ss:unitensorK}
The construction of unitary K-matrices requires to consider only $V\in\fdUqL$
such that  $V^{\psi^2}=({\Vpsi})^\psi$ is isomorphic to $V$ as a $\Uqk$-module.
This constraint yields no loss of generality. 
Choosing the gauge transformation $\gau=\bt{\tsat}\cdot\Parc^{-1}\in\Gg_{\theta,\Parc}$, we have $\psi=\omega\circ\tau$.
In particular, $\psi$ is an involution so that $V^{\psi^2}=V$ for every $V\in\fdUqL$.

In \cite[Prop.\ 5.4.1]{AV22b}, we proved that, in this setting, the operators $\rKM{V}{z}$ and $\rKM{V^{\psi}}{z}$ can be normalized to a pair of unitary K-matrices $\rnKM{V}{z}$ and $\rnKM{V^{\psi}}{z}$. This yields the following.

\begin{prop} \label{prop:unitensorK:1}
	Let $M\in\fdUqk$ be the restriction of an irreducible $\UqLg$-module. 
	For any irreducible $V\in\fdUqL$ such that $V^{\psi^2} = V$, set
	\begin{equation}
		\rnTKM{MV}{z}=(\rnRM{V^\psi M}{z})_{21}\cdot (\id_M \ten \rnKM{V}{z}) \cdot \rnRM{MV}{z}
	\end{equation}
	in $\End(M \ten V)(z)$. Then, 
	\begin{equation}
		\rnTKM{MV}{z}^{-1}=\rnTKM{M\Vpsi}{z^{-1}}\,.
	\end{equation}
	Moreover, if $V(\ze)$ is $\Uqk$-irreducible and $M \ten V(\ze)$, $\Vpsi(\ze) \ten M$ are $\UqLg$-irreducible for some $\ze\in\bsF^\times$, then 
	$\rnTKM{MV}{z}$ is well-defined and invertible.
\end{prop}

As before, for more general modules in $M\in\fdUqk$, Proposition~\ref{prop:unitensorK:1} follows at once whenever the space of QSP intertwiners $M\ten\Lfml{V}{z}\to M\ten\Lfml{V^{\psi}}{z^{-1}}$ is shown to be one-dimensional.


\appendix

\section{Finite-dimensional $\Uqk$-modules} \label{ss:findimQSPmodules}

In this appendix we will prove Proposition \ref{prop:weight-QSP-cat} \ref{prop:findimreps}: up to twisting by algebra automorphisms of $\Uqk$, finite-dimensional irreducible $\Uqk$-modules are QSP weight modules.
The full description of $\Modfd(\Uqk)$ for general $\fkk$ is an open problem which we do not discuss here (for the case $\dim(\fkg)<\infty$, see \cite{Wat21} for recent developments and a survey). 

Recall that on any finite-dimensional module $V$ of $\Uqg$ with $\dim(\fkg) < \infty$ the generators $K_i$ act via scalar multiplication by signed integer powers of $q$, see \eg \cite[Sec.~10.1]{CP95}. 
Given $\bm \eta_i \in \{ \pm 1\}$, consider the unique algebra automorphism $\Psi_{\bm \eta}$ of $\Uqg$ such that
\eq{
\Psi_{\bm \eta}(E_i) = \eta_i E_i, \qq \Psi_{\bm \eta}(F_i) = F_i, \qq \Psi_{\bm \eta}(K_i^{\pm 1}) = \eta_i K_i^{\pm 1}.
}
As a consequence, for all finite-dimensional modules $V$ there exists $\bm\eta \in \{ \pm 1 \}^I$ such that $\Psi_{\bm \eta}^*(V)$ is a type $\bf 1$ weight module.\\

We now set out to give a QSP analogue of this statement for \emph{irreducible} finite-dimensional $\Uqk$-modules (we do not restrict to the case $\dim(\fkg) < \infty$), with $\WUqk$ substituting for the category of type $\bf 1$ weight modules of $\Uqg$.
We will sometimes write $U_q(\fkk)_{\bm \ga}$ to emphasize the dependence of the QSP subalgebra on the tuple $\bm \ga \in \Gamma$.

\subsection{Weak QSP weight vectors}

First, we consider a weaker notion of QSP weight module.
Namely, for a $\Uqk$-module $M$ and a multiplicative character $\beta \in \Hom ( (\Qvext)^\tsat, \bsF^\times)$, denote the corresponding \emph{weak QSP weight space} by
\eq{
M(\beta) =  \{ m \in M \, | \, K_h \cdot m = \beta(h) m \text{ for all } h \in (\Qvext)^\tsat \}.
}
We call $m \in M$ a \emph{weak QSP weight vector} if it is nonzero and lies in $M(\beta)$ for some $\beta \in \Hom ( (\Qvext)^\tsat, \bsF^\times)$ and we call $M$ a \emph{weak QSP weight module} if as a $U_q(\fkh^\theta)$-module it decomposes as a direct sum of weak QSP weight spaces.
As in Proposition \ref{prop:weight-QSP-cat} \ref{prop:actiononweightspaces} one obtains 
\eq{ \label{weak-weight-QSP-action}
B_i \cdot M(\beta) \subseteq M(q^{-[\al_i]_\theta}\beta), \qq E_j \cdot M(\beta) \subseteq M(q^{[\al_j]_\theta}\beta)
}
for all $i \in I$, $j \in X$, where we note $[\al_i]_\theta \in (\Qvext)^\tsat$ for all $i \in I$.
Quite straightforwardly one obtains, relying on the algebraic closedness of $\bsF$, any irreducible finite-dimensional $\Uqk$-module has a nonzero weak QSP weight space and hence, by \eqref{weak-weight-QSP-action}, is a weak QSP weight module. 

We will now deduce a more powerful statement, still in full generality.
For this it is convenient to describe $(\Qvext)^\tsat$, and hence $U_q(\fkh^\theta)$, more explicitly as follows (see e.g. \cite[Lem.~4.8]{RV22}).
Fix a set $I_\theta$ of representatives of the \emph{nontrivial} $\tau$-orbits of $I \backslash X$. 
Then 
\[
(\Qvext)^\tsat = \Sp_\Z \big\{ [\al_i]_\theta \, \big| \, i \in I_\theta \cup X \big\}
\]
and the commutative subalgebra $U_q(\fkh^\theta) \subset \Uqk$ is generated by $K_j^{\pm 1}$ ($j \in X$) and $(K_i K_{\tau(i)}^{-1})^{\pm 1}$ ($i \in I_\theta$).
We call the elements of
\[
\{ E_j, F_j, K_j^{\pm 1} \, | \, j \in X \} \cup \{ B_{\tau(i)}, B_i, (K_i K_{\tau(i)}^{-1})^{\pm 1} \, | \, i \in I_\theta \} \cup \{ B_i \, | \, i \in I \backslash X,  \tau(i)=i \} 
\]
the \emph{canonical generators} of $\Uqk$. 

Consider $(\Qvext)^{\tsat,+} = \Sp_{\Z_{\ge 0}} \big\{ [\al_i]_\theta \, \big| \, i \in I_\theta \cup X \big\}$.
Define a partial order $\ge_\theta$ on $\Hom( (\Qvext)^\tsat, \bsF^\times)$ as follows
\[
\beta \ge_\theta \beta' \qq \Longleftrightarrow \qq \beta/\beta' \in q^{(\Qvext)^{\tsat,+}}.
\]

\begin{lemma} \label{lem:QSPfindim}
Let $M \in \Modfd(\Uqk)$.
Then the generators $B_{\tau(i)}$, $B_i$ ($i \in I_\theta$) and $E_j$, $F_j$ ($j \in X$) act locally finitely on $M$.

Now assume $M$ is irreducible.
Then it is a weak QSP weight module.
Furthermore, $M$ is generated by a weak QSP weight vector $m_0$ with the following properties:

\vspace{1mm}
\begin{enumerate}\itemsep2mm
\item \label{lem:annihilation}
$m_0$ is annihilated by the action of $B_{\tau(i)}$ and $E_j$ for all $i \in I_\theta$ and $j \in X$;
\item \label{lem:eigenvalues}
the eigenvalues of the action of the following elements of $U_q(\fkh^\theta)$ on $m_0$ lie in $\pm q^\Z$: $K_j$ for all $j$ in $X$ and $K_i K_{\tau(i)}^{-1}$ for all $i$ in
\[
I_\theta^0 \coloneqq \big\{ i \in I_\theta \, \big| \, ( \theta(\al_i), \al_i) = 0 \big\}.
\]
\end{enumerate}

\end{lemma}

\begin{proof}
This is a variation on a standard argument.
Note that $M$ decomposes as a direct sum of generalized weak QSP weight spaces
\eq{
M^{\sf gen}(\beta) = \{ m \in M \, | \, \text{for all } h \in (\Qvext)^\tsat \text{ and all } r>>0, (K_h - \beta(h))^r \cdot m = 0 \}
}
where $\beta \in \Hom ( (\Qvext)^\tsat, \bsF^\times )$.
The relations \eqref{QSPrelations1} imply that
\[
B_i \cdot M^{\sf gen}(\beta) \subseteq M^{\sf gen}\big(q_i^{-[\al_i]_\theta} \beta \big), \qq
E_j \cdot M^{\sf gen}(\beta) \subseteq M^{\sf gen}\big(q_j^{[\al_j]_\theta} \beta \big)
\]
for all $i \in I$, $j \in X$ and $\beta \in \Hom ( (\Qvext)^\tsat, \bsF^\times )$.
Note that $-[\al_{\tau(i)}]_\theta = [w_X(\al_i)]_\theta$ for $i \in I_\theta$. 
Owing to relations \eqref{QSPrelations1} the action of $B_{\tau(i)}$ and $E_j$ on generalized weak QSP weight spaces will strictly increase the weight with respect to $\ge_\theta$ and the action of $B_i$ and $F_j$ will strictly decrease the weight, where $i$ is an arbitrary element of $I_\theta$ and $j$ is an arbitrary element of $X$. 
Since $\dim(M)<\infty$, we obtain the local nilpotency and the existence of a nonzero element of $M$ annihilated by, say, $B_{\tau(i)}$ for all $i \in I_\theta$ and $E_j$ for all $j \in X$.

Consider the nonzero subspace 
\[
M_0 = \{ m \in M \, | \, B_{\tau(i)} \cdot m = E_j \cdot m = 0 \text{ for all } i \in I_\theta, \, j \in X \}.
\] 
By \eqref{QSPrelations1}, $M_0$ is a $U_q(\fkh^\theta)$-submodule.
Since $\dim(M_0) < \infty$ and $\bsF$ is algebraically closed, $M_0$ contains a joint eigenvector $m_0$ of all $K_h$ ($h \in (\Qvext)^\tsat$), i.e. $m_0$ is a weak QSP weight vector.
Since the canonical generators of $\Uqk$ send weak QSP weight spaces to weak QSP weight spaces, we obtain that $\Uqk \cdot m_0$ is a weak QSP weight submodule of $M$, which must equal $M$ by irreducibility.
We obtain \ref{lem:annihilation}.

Finally, for $i \in I_\theta^0$, \cite[Thm.~3.6]{BK15} implies the relation
\eq{ \label{QSPSerre:special}
[B_{\tau(i)},B_i] = c \frac{K_i K_{\tau(i)}^{-1} - K_i^{-1}K_{\tau(i)} }{q_i-q_i^{-1}}
}
where $c \in \bsF^\times$ is proportional to ${\bm \ga}_i = {\bm \ga}_{\tau(i)}$.
Hence, for all $i \in I_\theta^0$ the subalgebra 
\eq{
\bsF\langle B_{\tau(i)}, \, B_i, \, (K_i K_{\tau(i)}^{-1})^{\pm 1} \rangle \subseteq \Uqk
}
is isomorphic to $U_q(\fksl_2)$.
Hence, a standard $U_q(\fksl_2)$-argument, see e.g.\ \cite[Prop.~2.3]{Jan96}, implies \ref{lem:eigenvalues}.
\end{proof}

\subsection{Defining relations of QSP algebras}

Recall from the definition of $\Uqg$ the noncommutative polynomial $\Serre_{ij}(F_i,F_j)$ of degree $(1-a_{ij})\al_i+\al_j$.
By \cite[Thms.~7.1 \& 7.3]{Kol14}, the QSP algebra $U_q(\fkk)_{\bm \ga}$ has a presentation as follows. 
It is generated over $U_q(\fkn_X^+)U_q(\fkh^\theta)$ by elements $B_i$ ($i \in I$) subject to the relations
\begin{align}
\label{QSP:Cartan} K_h B_i &= q_i^{-\al_i(h)} B_i K_h && \text{for all } h \in (\Qvext)^\tsat, \, i \in I, \\
\label{QSP:Uqsl2} [E_j,B_i] &= \delta_{ij} \frac{K_i-K_i^{-1}}{q_i-q_i^{-1}} && \text{for all } i \in I, \, j \in X, \\
\label{QSP:Serre} \Serre_{ij}(B_i,B_j) &= C_{ij} && \text{for all } i,j \in I, \, i \ne j, 
\end{align}
where $C_{ij}$ is a noncommutative polynomial over $U_q(\fkn_X^+)U_q(\fkh^\theta)$ in $B_i$ and $B_j$, of degree strictly less than $(1-a_{ij})\al_i+\al_j$.
If $i \in X$ or if $i \not\in \{ \tau(i) , \tau(j) \}$ then by \cite[Lem.~5.11, Thm.~7.3]{Kol14} we have $\Serre_{ij}(B_i,B_j) = 0$.
For other $(i,j)$, various more explicit forms have been provided for the relation \eqref{QSP:Serre}, see \cite{Let02,Kol14,BK15,DC19,CLW21a,CLW21b,KY21}.

Drawing on these works, here we give properties of the $C_{ij}$ which are sufficient to our purposes.
For all $i \in I \backslash X$ there exists $Z_i \in U_q(\fkn^+_X)_{w_X(\al_i)-\al_i}$ such that the following statements are satisfied.

\vspace{1mm}
\begin{enumerate}\itemsep2mm
\item 
By \cite[Cor.~3.5]{DC19}, for all $i,j \in I \backslash X$ such that $\tau(i)=i \ne j$, $\Serre_{ij}(B_i,B_j)$ is a $\bsF$-linear combination of products of one factor $B_j$, $M$ factors $B_i$ and $(1-a_{ij}-M)/2$ factors $Z_i$, where $M$ runs through $\{ 0,1,\ldots, -1-a_{ij} \}$. 
\item 
By \cite[Cor.~3.9]{DC19}, for all $i \in I \backslash X$ and $j \in X$ such that $\tau(i)=i$, $\Serre_{ij}(B_i,B_j)$ is an $\bsF$-linear combination of products of one factor $B_j$, $M$ factors $B_i$ and either $(1-a_{ij}-M)/2$ factors $Z_i$ or $(-1-a_{ij}-M)/2$ factors $Z_i$ and one factor $W_{ij}K_j$, where $M$ runs through $\{ 0,1,\ldots, -1-a_{ij} \}$ and $W_{ij} \in U_q(\fkn^+_X)_{w_X(\al_i)-\al_i-\al_j}$. 
\item By \cite[Thm.~3.6]{BK15}, for all $i \in I_\theta \cup \tau(I_\theta)$, $\Serre_{i \tau(i)}(B_i,B_{\tau(i)})$ is an $\bsF$-linear combination, with coefficients independent of $\bm \ga$, of $\ga_i B_i^{-a_{i \tau(i)}} Z_i K_i^{-1} K_{\tau(i)}$ and $\ga_{\tau(i)} B_i^{-a_{i \tau(i)}} Z_{\tau(i)} K_i K_{\tau(i)}^{-1}$.  
\end{enumerate}
This completes the enumeration of all $(i,j) \in I \times I$ such that $i \ne j$, $i \in \{ \tau(i), \tau(j)\}$ and $i \notin X$.

Before we use these relations to describe certain automorphisms of $\Uqk$, we look at a special case which shows that in Lemma \ref{lem:QSPfindim} the irreducibility assumption is necessary: for some $\Uqk$ there exist reducible finite-dimensional $\Uqk$-modules which are not weak QSP weight modules.

\begin{example} \label{exam:augqOnsmodule}
Suppose $I = \{ 0, 1 \}$, $a_{01}=a_{10}=-2$, so that $\fkg' \cong \wh\fksl_2$. 
Let $\tau$ be the nontrivial diagram automorphism, and consider the Satake diagram $(\emptyset,\tau)$.
Specializing the definitions in Sections \ref{ss:qtheta}-\ref{ss:qsp} to this case, we obtain the QSP subalgebra $\Uqk \subset U_q(\wh\fksl_2)$ generated by
\eq{
B_0 = F_0 - q \gamma_0 E_1 K_0^{-1}, \qq B_1 = F_1 - q \gamma_1 E_0 K_1^{-1}, \qq L^{\pm 1}
}
where $\ga_0,\ga_1 \in \bsF^\times$ and $L = K_0K_1^{-1}$.
It is also known as \emph{augmented q-Onsager algebra}.
By \eqrefs{QSP:Cartan}{QSP:Serre}, relying on \cite[Thm.~3.6]{BK15} for the explicit expression of $C_{ij}$ in this case (also see \cite{BB17}), this is isomorphic to the algebra generated over $\C(q)$ by $B_0$, $B_1$ and invertible $L$ subject to the relations
\eq{
\begin{aligned}
L B_0 &= q^{-4} B_0 L, \hspace{30pt} L B_1 = q^4 B_1 L, \\
\Serre_{01}(B_0,B_1) &= (q+q^{-1})(q^3-q^{-3}) B_0 \big( \ga_0 L^{-1} - \ga_1 L \big) B_0, \\
\Serre_{10}(B_1,B_0) &= (q+q^{-1})(q^3-q^{-3}) B_1 \big( \ga_1 L - \ga_0 L^{-1} \big) B_1.
\end{aligned}
}
Any finite-dimensional $\C(q)$-linear space $M$ becomes a $\Uqk$-module if we let $B_0$, $B_1$ act by the zero map and $L$ by any invertible operator $L_M$. 
Clearly, this is a weak QSP weight module if and only if $L_M$ is diagonalizable.
\hfill \rmkend
\end{example}

\subsection{Some automorphisms of QSP algebras}

We will use these defining relations to describe certain automorphisms of the QSP algebra which act on the canonical generators by scalar multiplication.
The only defining relations which constrain any such algebra automorphism $f$ of $\Uqk$ are the non-homogeneous relations: \eqref{QSP:Uqsl2} in the case $i=j \in X$, constraining only the scalar factor for $E_j$, $F_j$ and $K_j^{\pm 1}$ ($j \in X$), the q-Serre relations \eqref{QSP:Serre} in the case $i \ne \tau(i)=j \in I \backslash X$, which constrains the products of the scalar factors appearing in $f(B_i)$ and $f(B_{\tau(i)})$ in terms of the induced scalar factors of $Z_i$ and $Z_{\tau(i)}$, and the q-Serre relations \eqref{QSP:Serre} in the case $\tau(i)=i \in I \backslash X$, which only constrains the scalar factor of $B_i$ in terms of the scalar factors appearing in $f(Z_i)$ and $f(W_{ij}K_j)$.

The following is a special case of \cite[Lem.~2.5.1]{Wat21}, which we reproduce here in the present conventions.

\begin{lemma} \label{lem:QSP:aut1}
Let $i \in I_\theta \backslash I_\theta^0$ and $\ka \in \bsF^\times$.
Given ${\bm \ga} \in \Gamma$, define ${\bm \ga}' \in \Gamma$ by 
\eq{
{\bm \ga}'_j = \begin{cases} \ka \bm \ga_i & \text{if } j=i, \\ \ka^{-1} \bm \ga_{\tau(i)} & \text{if } j=\tau(i), \\ \bm \ga_j & \text{otherwise}. \end{cases}
}
There is a unique isomorphism
\[ 
U_q(\fkk)_{\bm \ga} \stackrel{f_{i,\ka}}{\longrightarrow} U_q(\fkk)_{\bm \ga'}
\]
such that 
\eq{
f_{i,\ka}\big( (K_i K_{\tau(i)}^{-1})^{\pm 1}\big) = \ka^{\pm 1} (K_i K_{\tau(i)}^{-1})^{\pm 1}
}
and $f_{i,\ka}$ fixes all other canonical generators.
\end{lemma}

\begin{proof}
This follows from a direct inspection of the q-Serre relations given in \cite[Thm.~3.6]{BK15}, see point (3) above.
\end{proof}

We can also define certain algebra automorphisms of $U_q(\fkk)_{\bm \ga}$ analogous to the automorphisms $\Psi_{\bm \eta} : \Uqg \to \Uqg$ defined above.
First note that for any choice of ${\bm \eta}_j \in \{ \pm 1 \}$ ($j \in X$), the assignments
\[
E_j \mapsto {\bm \eta}_j E_j, \qq F_j \mapsto F_j, \qq K_j^{\pm 1} \mapsto {\bm \eta}_j K_j^{\pm 1}
\]
extend to an algebra automorphism of $U_q(\fkg_X)$.
Now define ${\bm \eta}_X \in \Hom(\Qlat_X, \bsF^\times)$ by ${\bm \eta}_X(\al_j) = {\bm \eta}_j$.
For each $\tau$-orbit outside $X$, i.e. each element of 
\[
I^*_\theta = I_\theta \cup \{ i \in I \backslash X \, | \, \tau(i)=i \}
\]
we need a correcting factor depending on $X$, in addition to a free choice of sign.
If $i \in I_\theta$ set $C(X,i) \coloneqq {\bm \eta}_X(w_X(\al_i)-\al_i) \in \{ \pm 1 \}$. 
On the other hand, if $i \in I \backslash X$ satisfies $\tau(i)=i$ then denote by $C(X,i)$ a fixed square root of ${\bm \eta}_X(w_X(\al_i)-\al_i) \in \{ \pm 1 \}$.

\begin{lemma} \label{lem:QSP:aut2}
Given ${\bm \eta} \in \{ \pm 1 \}^{X \cup I^*_\theta}$, the above algebra automorphism of $\UqgX$, depending only on $( {\bm \eta}_j )_{j \in X}$, extends to an algebra automorphism 
\[
\Psi_{\theta,{\bm \eta}}: \Uqk \to \Uqk
\]
by means of the assigments
\begin{align*}
B_i &\mapsto {\bm \eta}_i C(X,i) B_i , && i \in I \backslash X, \, \tau(i)=i, \\
(K_i K_{\tau(i)}^{-1})^{\pm 1} & \mapsto {\bm \eta}_i (K_i K_{\tau(i)}^{-1})^{\pm 1}, && i \in I_\theta\\
B_{\tau(i)} &\mapsto {\bm \eta}_i C(X,i) B_{\tau(i)}, && i \in I_\theta\\
B_i &\mapsto B_i, && i \in I_\theta.
\end{align*}
\end{lemma}

\begin{proof}
Let $j \in X$. 
As was the case for the automorphism $\Psi_{\bf \eta}$ of $\Uqg$, given that $F_j$ is fixed we readily obtain from \eqref{QSP:Uqsl2} the form of the assignments for $E_j$ and $K_j^{\pm 1}$ and we obtain an algebra automorphism $\Psi_{({\bm \eta}_j)_{j \in X}}$ of $\UqgX$.
Let $i \in I \backslash X$ and $j \in X$ be arbitrary. 
Since $Z_i \in U_q(\fkn^+_X)_{w_X(\al_i)-\al_i}$, $W_{ij} \in U_q(\fkn^+_X)_{w_X(\al_i)-\al_i - \al_j}$ and $w_X(\al_i)-\al_i$ is $\tau$-invariant, we see that $\Psi_{({\bm \eta}_j)_{j \in X}}$ acts on $Z_i$, $Z_{\tau(i)}$ and $W_{ij}K_j$ by scalar multiplication by the same sign ${\bm \eta}(w_X(\al_i)-\al_i)$. 

Let $i \in I_\theta$. 
It follows that $C(X,i) = C(X,\tau(i))$. 
Given that $B_i$ is fixed, the q-Serre relations \eqref{QSP:Serre} in the case $i \ne \tau(i)=j \in I \backslash X$ require the form of the assignments for $B_{\tau(i)}$ and $(K_i K_{\tau(i)}^{-1})^{\pm 1}$.

Finally, let $i \in I \backslash X$ such that $\tau(i)=i$. 
The q-Serre relations \eqref{QSP:Serre} in the case $\tau(i)=i \in I \backslash X$ require the form of the assignments for $B_i$.
\end{proof}

\subsection{From weak QSP weight vectors to QSP weight vectors}

Putting it all together we arrive at the following result.

\begin{theorem} \label{thm:QSP:findemreps}
Let $M \in \Modfd(\Uqk)$ be irreducible and let $m_0$ be as in Lemma \ref{lem:QSPfindim}.
Then there exists an algebra automorphism $f$ of $\Uqk$ such that $m_0 \in f^*M$ is an eigenvector of the action of $K_j$ ($j \in X$) and $K_i K_{\tau(i)}^{-1}$ ($i \in I_\tau$) with eigenvalues in $q^\Z$.
\end{theorem}

\begin{proof}
Let $M$ be a finite-dimensional irreducible $\Uqk$-module with $m_0 \in M$ as in Lemma \ref{lem:QSPfindim}.
By Lemma \ref{lem:QSP:aut1}, there exists a tuple $\bm \ka = ({\bm \ka}_i) \in (\bsF^\times)^{I_\theta \backslash I_\theta^0}$ such that $(f_{\bm \ka})^*M$ is a weak QSP weight module with the joint action of $U_q(\fkh^\theta)$ on $m$ given by scalar multiplication by signed powers of $q$, where
\eq{
f_{\bm \ka} \coloneqq \prod_{i \in I_\theta \backslash I_\theta^0} f_{i,{\bm \ka}_i}.
}
Now we can use Lemma \ref{lem:QSP:aut2} to remove signs from the eigenvalues.
\end{proof}

\section{Cylindrical and reflection bialgebras}\label{ss:cylindrical-algebraic}

We outline a generalization of the algebraic approach of \cite[Sec.~2]{Kol20} to a braided structure on the category of modules of certain coideal subalgebras of quasitriangular bialgebras. 
This was developed independently by Lemarthe, Baseilhac and Gainutdinov in \cite{LBG23,Lem23} (for comodule algebras over quasitriangular algebras).
We also connect with a formalism for comodule algebras over a \emph{weakly} quasitriangular bialgebra extending its braided structure, as introduced by Kolb and Yakimov in \cite{KY20}, which we call here \emph{weakly cylindrical}.

\subsection{Quasitriangular bialgebras \cite{Dri86}} \label{ss:qt-hopf}

Fix a base field $\bbF$.
A bialgebra $A$ with coproduct $\Del$ and opposite coproduct $\Delta^{\op} = (12) \circ \Delta$ is called \emph{quasitriangular} if there exists a \emph{universal R-matrix} $R \in (A \ten A)^\times$, \ie an element satisfying the intertwining identity
\begin{gather}
	\label{eq:R-intw} R \cdot \Delta(x) = \Delta^{\op}(x) \cdot R \qq (x\in A)
\end{gather}
and the coproduct identities
\begin{gather}
	\label{eq:R-coprod} (\Del \ten \id)(R) = R_{13}\cdot R_{23} \qq\mbox{and} \qq (\id \ten \Del)(R) = R_{13}\cdot R_{12}
\end{gather}
where $\Del^{\op} \coloneqq (12) \circ \Del$.
Then $R$ is a solution of the {Yang-Baxter equation}
\begin{equation}\label{eq:YBE}
	R_{12}\cdot R_{13}\cdot R_{23}=R_{23}\cdot R_{13}\cdot R_{12}.
\end{equation}
Moreover, $(\veps \ten \id)(R) = (\id \ten \veps)(R) = 1$, where $\veps$ is the counit of $A$.
Note also that the \emph{co-opposite} $A^{\cop}$ is a quasitriangular bialgebra, where $\Delta$ has been replaced by $\Delta^{\op}$ and $R$ by $R_{21} = (12)(R)$.\\

A quasitriangular bialgebra can be twisted as follows, also see \cite{Dav07}.

\vspace{1.5mm}
\begin{itemize}\itemsep2mm
\item 
Let $J\in A\ten A$ be a Drinfeld twist, \ie an invertible element satisfying 
\[
(J \ten 1) \cdot (\Delta \ten \id)(J) = (1 \ten  J) \cdot (\id \ten \Delta)(J), \qq (\veps \ten \id)(J) = (\id \ten \veps)(J) = 1.
\] 
There is a unique quasitriangular bialgebra $A_J$ such that $A_J=A$ as an algebra, its coproduct is $\Delta_J=\Ad(J)\circ\Delta$, the counit is $\veps$, and its universal R-matrix is $R_J= J_{21}\cdot R\cdot J^{-1}$.

\item 
Let $\psi:A\to A$ be an algebra automorphism. 
There is a unique quasitriangular bialgebra $A^\psi$ such that $A^\psi=A$ as an algebra, its coproduct is $\Delta^{\psi} = (\psi \ten \psi) \circ \Delta \circ \psi^{-1}$, the counit is $\veps^{\psi}=\veps\circ\psi^{-1}$, and its universal R-matrix is $R^{\psi\psi} = (\psi \ten \psi)(R)$. 
By construction, $\psi$ is an isomorphism of quasitriangular bialgebras $A \to A^{\psi}$.
\end{itemize}
\vspace{1.5mm}
If $A^{\cop,\psi} = A_J$ as quasitriangular bialgebras, then we call the pair $(\psi, J)$ a \emph{twist pair}, see \cite{AV22a}. 

\begin{remark} \label{rmk:Drinfeldtwist}
Note that, up to right-multiplication by elements of the centralizer of $\Delta(A)$ in $A \ten A$, $J$ is uniquely determined by $\psi$ via the constraint $(\Del^{\op})^\psi = \Del_J$, in the same way that $R$ is almost uniquely determined by \eqref{eq:R-intw}. \hfill \rmkend
\end{remark}

\subsection{Cylindrical bialgebras \cite{AV22a}}\label{ss:cyl-bialg}
A cylindrical bialgebra $(A,B, \psi, J, K)$ is the datum of 
\begin{itemize}\itemsep2mm
\item 
a quasitriangular bialgebra $A$ (with coproduct $\Del$ and universal R-matrix $R$);

\item 
a right coideal subalgebra $B\subseteq A$, \ie $\Delta(B)\subseteq B\ten A$;

\item 
a twist pair $(\psi, J)$;

\item 
a \emph{basic universal K-matrix} $K\in A$, \ie an invertible element satisfying the QSP intertwining identity
\begin{equation}\label{eq:K-intw} 
	K \cdot b = \psi(b) \cdot K \qq \text{for all } b\in B,
\end{equation}
and the coproduct identity
\begin{equation}\label{eq:K-coprod}
	\Delta(K) = J^{-1} \cdot (1 \ten K) \cdot R^\psi \cdot (K \ten 1).
\end{equation}

\end{itemize}

Equivalently, we shall also say that $(\psi, J, K)$ is a \emph{cylindrical structure} on $(A,B)$.
This generalizes the notion of \emph{cylinder-braided coideal subalgebra}, due to Balagovi\'c and Kolb in \cite[Def.~4.10]{BK19}, which in turn generalizes the notion of \emph{cylinder twist}, due to tom Dieck and H\"aring-Oldenburg in \cite{tD98, tDHO98}.

From \eqref{eq:K-coprod} we readily obtain $\veps(K)=1$.
Furthermore, we obtain the following generalized universal reflection equation.

\begin{prop}{\cite[Prop.~2.4]{AV22a}} \label{prop:basicK:RE}
Let $(\psi,J,K)$ be a cylindrical structure on $(A,B)$.
Then we have
\begin{equation}\label{eq:psi-RE-k}
R^{\psi \psi}_{21} \cdot (1 \ten K) \cdot R^\psi \cdot (K \ten 1) = (K \ten 1) \cdot R^\psi_{21} \cdot (1 \ten K) \cdot R
\end{equation}
\end{prop}

\begin{proof}
It follows by combining the intertwining property of the R-matrix \eqref{eq:R-intw}, the coproduct formula \eqref{eq:K-coprod} and the twist pair property $R_{21}^{\psi \psi} = R_J$. 
\end{proof}

\begin{example}
Universal R-matrices are examples of basic universal K-matrices, cf. \cite[Rmk. 4.10]{BW18a} for a similar remark in the setting of quantum symmetric pairs. 
Namely, $A^{\cop} \ten A$ is a quasitriangular bialgebra with coproduct $(23) \circ (\Delta^{\op} \ten \Delta)$ and universal R-matrix $R_{31} \cdot R_{24}$; moreover, $\Delta(A) \subseteq A^{\cop} \ten A$ is a right coideal subalgebra.
One checks that 
\eq{
(\psi,J,K) = ( (12), 1 \ten 1, R)
}
is a cylindrical structure, called \emph{diagonal}, on $(A^{\sf cop} \ten A, \Delta(A))$. \hfill \rmkend
\end{example}

The following \emph{support condition} on a cylindrical structure $(\psi, J, K)$ on $(A,B)$ is natural:
\eq{ \label{tensor-K-support}
	R^\psi_{21} \cdot (1 \ten K) \cdot R \qu \in \qu B \ten A.
}
Indeed, while \eqref{eq:K-intw} functions as an upper bound on $B$, \eqref{tensor-K-support} plays the role of a lower bound, see also \cite[Rmk.~2.11]{Kol20}.
We also obtain from \eqref{tensor-K-support} a second derivation of \eqref{eq:psi-RE-k} from by combining it with \eqref{eq:K-intw} and the coideal property.
Since $B$ is a coideal subalgebra, \eqref{tensor-K-support} permits the module category $\Mod(B)$ over the monoidal category $\Mod(A)$ to be endowed with a braided structure compatible with the braided structure on $\Mod(A)$. 
We develop the categorical framework in Appendix \ref{ss:bdrybimodcat}. 
The key algebraic ingredient is the universal tensor K-matrix.

\subsection{Tensor K-matrices and reflection bialgebras} \label{ss:reflection-bialgebras}
In a similar spirit, we introduce the notion of a reflection bialgebra.
This is the natural generalization (to arbitrary twist pairs) of the notions of {\em reflection algebra}, given by Enriquez in \cite[Def.~4.1]{Enr07}, and {\em quasitriangular comodule algebra}, given by Kolb in \cite[Def.~2.7]{Kol20}. 

\begin{definition}
A \emph{reflection bialgebra} $(A,B, \psi,J,\TKM{})$ is the datum of 
\begin{itemize}\itemsep2mm
\item 
a quasitriangular bialgebra $A$ (with coproduct $\Del$ and universal R-matrix $R$);

\item 
a right coideal subalgebra $B\subseteq A$, \ie $\Delta(B)\subseteq B\ten A$;

\item 
a twist pair $(\psi, J)$;

\item
a \emph{universal tensor K-matrix} $\TKM{} \in B \ten A$, \ie an invertible element satisfying 
\begin{align}
	\label{eq:2K-intw} 
	\qq \qq \TKM{} \cdot\Delta(b) &= (\id\ten\psi)(\Delta(b)) \cdot \TKM{} \qq \text{for all } b\in B, \\
	\label{eq:2K-coprod-1} 
	(\Delta\ten\id)(\TKM{}) &= R^\psi_{32} \cdot \TKM{13} \cdot R_{23},\\
	\label{eq:2K-coprod-2} 
	(\id\ten\Delta)(\TKM{}) &= J_{23}^{-1} \cdot \TKM{13} \cdot R^\psi_{23} \cdot \TKM{12}.
\end{align}
\end{itemize}
Equivalently, we shall also say that $(\psi, J, \TKM{})$ is a \emph{reflection structure} on $(A,B)$.
\end{definition}

Note that by applying $\id \ten \id \ten \veps$ to \eqref{eq:2K-coprod-2} we obtain the normalization 
\[
(\id\ten\veps)(\TKM{})=1.
\]

One readily deduces a following reflection equation in $A \ten A \ten A$.
\begin{prop} \label{prop:tensorK:RE}
The element $\TKM{}$ satisfies the following equation in $B \ten A \ten A$:
\begin{equation}\label{eq:psi-RE}
R^{\psi \psi}_{32} \cdot \TKM{13} \cdot R^\psi_{23} \cdot \TKM{12} = \TKM{12} \cdot R^\psi_{32} \cdot \TKM{13} \cdot R_{23}.
\end{equation}
\end{prop}

\begin{proof}
As for Proposition \ref{prop:basicK:RE}, it follows from \eqref{eq:R-intw}, \eqref{eq:2K-coprod-2} and the twist pair property. 
Alternatively, one can deduce it from \eqref{eq:2K-intw} and \eqref{eq:2K-coprod-1} and the coideal property.
\end{proof}

Indeed, a reflection structure is equivalent to a cylindrical structure supported on $B$.
First of all, in analogy with (part of) \cite[Lem.~2.9]{Kol20} (cf.~\cite[Rmk.~4.2]{Enr07}), we note that reflection structures straightforwardly induce cylindrical structures.

\begin{lemma} \label{lem:refl-to-cyl}
	Let $(A, B, \psi, J, \TKM{})$ be a reflection bialgebra and set
	\[
	\KM{}=(\veps\ten\id)(\TKM{})\in A.
	\]
	Then $(A, B, \psi, J, K)$ is a cylindrical bialgebra.
	Furthermore, the following factorization holds
	\eq{ \label{tensorK:formula}
		\TKM{} = R^\psi_{21}\cdot (1 \ten K) \cdot R.
	}
\end{lemma}

\begin{proof}
The proof is almost identical to \cite[Lem.~2.9]{Kol20}.
By applying the counit in the appropriate factor, we obtain \eqref{eq:K-intw} from \eqref{eq:2K-intw} and \eqref{eq:K-coprod} from \eqref{eq:2K-coprod-2}. 
Finally, by a straightforward computation we have
\begin{flalign}
&& \TKM{} &\stackrel{\phantom{\eqref{eq:2K-coprod-1}}}{=} (((\veps \ten \id) \circ \Delta) \ten \id)(\TKM{}) \\
&& &\stackrel{\eqref{eq:2K-coprod-1}}{=} (\veps \ten \id \ten \id)(R^\psi_{32} \cdot \TKM{13} \cdot R_{23}) \\
&& &\stackrel{\phantom{\eqref{eq:2K-coprod-1}}}{=} R^\psi_{21} \cdot (1 \ten K) \cdot R. & \tag*{\qedhere}
\end{flalign} 
\end{proof}

\begin{remark} 
	If $\psi$ is a bialgebra automorphism and $J = R_{21}^{-1}$, the setup of Kolb in \cite[Sec.~2]{Kol20} is recovered.
	Note that in Kolb's approach the twist automorphism $\psi$ is removed from the axiomatics by replacing $\Mod_{\sf fd}(\Uqg) \subset \Mod(\Uqg)$ by an equivariantized category in the case $\dim(\fkg)<\infty$. 
	However, we are also interested in the case $\dim(\fkg) = \infty$ and categories of modules $\cC \subset \Mod(\Uqg)$ strictly larger than $\Mod_{\sf fd}(\Uqg)$.
	Since in general $\cC$ is not preserved by the pullback of $\psi$, equivariantization is not applicable in our setting. \hfill \rmkend
\end{remark}

The concept of a reflection structure generalizes to the case where $B$ is a right $A$-comodule algebra with coaction map $\Delta_B: B \to B \ten A$, which is the setting used in \cite{LBG23,Lem23}. 
In order to formulate the generalized axiomatics of a universal tensor K-matrix, one simply replaces the coproduct $\Delta$ by $\Delta_B$ in axioms \eqref{eq:2K-intw} and \eqref{eq:2K-coprod-1}.\footnote{
Recall that a right $A$-coaction on $B$ is a linear map$: B \to B \ten A$ satisfying $(\Delta_B \ten \id_A) \circ \Delta_B = (\id_B \ten \Delta) \circ \Delta_B$.
Quite generally, if there exists an algebra map $\veps_B: B \to \bbF$ (\ie $B$ is an augmented algebra) then $\wt B \coloneqq ((\veps_B \ten \id_A) \circ \Del_B)(B)$ is a right coideal subalgebra of $A$. 
Now a reflection structure $(\psi,J,\TKM{})$ on $(A,B)$ induces a cylindrical structure $(\psi,J,K)$ on $(A,\wt B)$ via $K = (\veps_B \ten \id)(\TKM{}) \in A$.
It has the property $R^\psi_{21}\cdot (1 \ten K) \cdot R \in \wt B \ten A$. 
A natural sufficient condition for this cylindrical structure to be supported on $\wt B$ is isomorphicity of $B$ and $\wt B$ as right $A$-comodule algebras, \ie  injectivity of $\iota_B \coloneqq (\veps_B \ten \id_A) \circ \Del_B$.
In this case, one has the factorization $(\iota_B \ten \id)(\TKM{}) = R^\psi_{21} \cdot (1 \ten K) \cdot R $.
}

\subsection{Gauge transformations} \label{ss:gauge}
For any quasitriangular bialgebra $A$ and coideal subalgebra $B$, it is straightforward to check that the following assignments define an action of $A^\times$ on cylindrical structures on $(A,B)$ called \emph{gauge transformation} (also see \cite[Rmk.~8.11]{AV22a}):
\[
(\psi,J,K) \mapsto (\Ad(g) \circ \psi, (g \ten g) \cdot J \cdot \Del(g)^{-1}, g \cdot K), \qq g \in A^\times.
\]
This immediately lifts to an action on reflection structures, with the action on the tensor K-matrix given by
\eq{ \label{tensorK:gauge}
	\TKM{} \mapsto (1 \ten g) \cdot \TKM{},
}
which is compatible with Lemma \ref{lem:refl-to-cyl}.
Note that the support condition \eqref{tensor-K-support} is invariant under this action so that, in order to prove \eqref{tensor-K-support}, it suffices to do so for a convenient choice of gauge transformation.

\subsection{Weakly cylindrical bialgebras}

Several authors \cite{Tan92,Res95,Gav97} have considered a generalization of the notion of a quasitriangular bialgebra where the role of the universal R-matrix is played by an algebra automorphism.
A boundary analogue of this, generalizing the notion of reflection algebra with a distinguished comodule algebra, was considered by \cite{KY20} in a very general context of pairs of suitable Nichols algebras and comodule algebras. 
To describe both types of structures we will mostly use the conventions from that paper, in particular the terminology \emph{weak}.

In particular, a bialgebra $A$ with coproduct $\Delta$ is called \emph{weakly quasitriangular} if there exists an algebra automorphism $\cS$ of $A \ten A$ fixing $\Delta(A) \subset A \ten A$ pointwise such that
\begin{align*}
(\id \ten \Del) \circ \cS &= (\cS \ten \id) \circ (\id \ten \cS) \circ (\Del \ten \id), \\
(\Del \ten \id) \circ \cS &= (\id \ten \cS) \circ (\cS \ten \id) \circ (\id \ten \Del).
\end{align*}
Quasitriangular bialgebras are weakly quasitriangular: just set $\cS$ equal to $(12) \circ \Ad(R)$.\\

We call a weakly quasitriangular bialgebra $A$ with coproduct $\Del$ and automorphism $\cS$ \emph{weakly cylindrical} if there exists a right coideal subalgebra $B \subseteq A$ and an algebra automorphism $\ka: A \to A$ fixing $B$ pointwise such that 
\eq{
	\label{eq:weakK-coprod} \Delta \circ \ka = (\ka \ten \id) \circ \cS \circ (\ka \ten \id) \circ \Del.
}
Further, we say that the weakly cylindrical structure $\ka$ is \emph{supported on} $B$ if the automorphism $\cS  \circ (\ka \ten \id) \circ \cS$ preserves $B \ten A$.

We readily obtain that cylindrical bialgebras with twist pair $(\psi,J)$ are weakly cylindrical by setting $\ka = \psi^{-1} \circ \Ad(K)$, i.e. $\ka = \xi^{-1}$ in the notation of Section \ref{ss:class}, and this weakly cylindrical structure is supported on $B$ if the cylindrical structure $(\psi,J,K)$ is supported on $B$.
Note that the verification of \eqref{eq:weakK-coprod} relies on the fact that $A^{\cop,\psi} = A_{J}$.\\

A \emph{weak reflection bialgebra} is a weakly quasitriangular bialgebra $A$ with coproduct $\Del$ and automorphism $\cS$ together with a right coideal subalgebra $B \subseteq A$ and an automorphism $\cK$ of $B \ten A$ fixing $\Delta(B)$ pointwise such that
\begin{align}
	\label{eq:weak2K-coprod1} (\Del \ten \id) \circ \cK &= (\id \ten \cS) \circ (\cK \ten \id) \circ (\id \ten \cS) \circ (\Del \ten \id) , \\
	\label{eq:weak2K-coprod2}  (\id \ten \Del) \circ \cK &= (\cK \ten \id) \circ (\id \ten \cS) \circ (\cK \ten \id) \circ (\id \ten \Del).
\end{align}
Generalizing the coideal subalgebra of $A$ to a comodule algebra over $A$ as at the end of Section \ref{ss:reflection-bialgebras}, this coincides with the notion of a weakly quasitriangular comodule algebra, see \cite[Def.~6.11]{KY20}. 

\begin{remarks}
\hfill
\begin{enumerate}\itemsep2mm
\item 
The definition of weak cylindricity is clearly simpler than the definition of cylindricity involving the twist pair $(\psi,J)$.
On the other hand, for the representation theory of quantum loop algebras it is natural to separate the role of the universal K-matrix and the twist automorphism.
\item
In the traditional (non-weak) formalism for quasitriangularity and cylindricity, typically universal solutions are constructed in a completion of the bialgebra with respect to a certain category of representations. 
If the representation theory of the bialgebra is intractable then the weak formalisms provide workarounds, although the use of completions, see \cite[Sec.~6.3]{KY20}, is still necessary.\footnote{Indeed, if $(A,B)$ is a quantum symmetric pair then $(12) \circ \Ad(R)$ is not an automorphism of $A \ten A$, although it fixes $\Del(A) \subset A \ten A$ pointwise, and similarly $\psi^{-1} \circ \Ad(K)$ is not an automorphism of $A$, although it fixes $B \subset A$ pointwise.} 
\item
If $A$ is a Hopf algebra, then $B \ten A \ten A$ is generated as a unital algebra by the subalgebras $\Del(B) \ten A$ and $B \ten \Del(A)$.\footnote{
To see this, note that $a \ten 1 = \Delta(a^{(1)}) \cdot (1 \ten S(a^{(2)}))$ for all $a \in A$.
}
In this case, the above axioms for the maps $\cS$, $\ka$ and $\cK$ imply the weak analogues of Yang-Baxter and (basic and tensor) reflection equations:
\begin{align}
(\id \ten \cS) \cdot (\cS \ten \id) \cdot (\id \ten \cS) &= (\cS \ten \id) \cdot (\id \ten \cS) \cdot (\cS \ten \id), \\
\cS \cdot (\ka \ten \id) \cdot \cS \cdot (\ka \ten \id) &= (\ka \ten \id) \cdot \cS \cdot (\ka \ten \id) \cdot \cS, \\
(\id \ten \cS) \cdot (\cK \ten \id) \cdot (\id \ten \cS) \cdot (\cK \ten \id) &= (\cK \ten \id) \cdot (\id \ten \cS) \cdot (\cK \ten \id) \cdot (\id \ten \cS),
\end{align}
equations for algebra automorphisms of $A \ten A \ten A$, $A^{\ten 2}$ and $B \ten A \ten A$, respectively.
Hence, we get a representation of the Artin-Tits braid group of type B$_n$ on $B \ten A^{\ten n}$ for any $n \in \Z_{>0}$, cf. \cite[Prop.~2.4]{AV22a}.
In fact, for suitable Nichols algebras, by \cite[Thms.~6.9, 6.15]{KY20} one has weak analogues of parameter-independent Yang-Baxter equations and a left reflection equation in the style of \cite[Sec.\ 4]{Che92}, see also \cite[Sec. 2.5]{AV22a}.
\hfill \rmkend
\end{enumerate}
\end{remarks}

\subsection{From K-matrices to tensor K-matrices}\label{ss:abs-non-KM} 

Returning to the approach of \cite{Kol20}, we ask the natural question when a cylindrical bialgebra $(A,B, \psi, J, K)$ can be promoted to a reflection bialgebra $(A,B, \psi, J, \TKM{})$ with $\TKM{} = R^\psi_{21}\cdot (1 \ten K) \cdot R$.
From Lemma \ref{lem:refl-to-cyl} it follows that the support condition \eqref{tensor-K-support} must be satisfied. 
It turns out that this is sufficient.

\begin{prop} \label{prop:cylindricaltoreflection}
Let $(A, B, \psi, J, \KM{})$ be a cylindrical bialgebra and set
\eq{
\TKM{} = R^\psi_{21}\cdot (1 \ten K) \cdot R.
}
Then the identities \eqrefs{eq:2K-intw}{eq:2K-coprod-2} are satisfied and $\KM{}=(\veps\ten\id)(\TKM{})$. 
Furthermore, $(\psi,J,\TKM{})$ is a reflection structure on $(A,B)$ (in other words, $(A, B, \psi, J, \TKM{})$ is a reflection bialgebra) if and only if \eqref{tensor-K-support} is satisfied.
\end{prop}

\begin{proof} 
For any $b \in B$, the coideal property $\Delta(B) \subseteq B \ten A$ allows the following proof of the twisted centralizer property \eqref{eq:2K-intw}:
\begin{align}
R^\psi_{21}\cdot (1 \ten K) \cdot R \cdot \Delta(b) &\stackrel{\eqref{eq:R-intw}}{=} R^\psi_{21}\cdot (1 \ten K) \cdot \Delta^{\op}(b) \cdot  R\\
&\stackrel{\eqref{eq:K-intw}}{=} R^\psi_{21}\cdot(\id\ten\psi)(\Delta^{\op}(b)) \cdot (1 \ten K) \cdot  R\\
&\stackrel{\phantom{\eqref{eq:K-intw}}}{=} (\id\ten\psi)(R_{21}\cdot\Delta^{\op}(b)) \cdot (1 \ten K) \cdot  R\\
&\stackrel{\eqref{eq:R-intw}}{=} (\id\ten\psi)(\Delta(b))\cdot R^\psi_{21} \cdot (1 \ten K) \cdot  R.
\end{align}

The first coproduct identity \eqref{eq:2K-coprod-1} is obtained as follows:
\begin{align}
(\Delta\ten\id) \big( R^\psi_{21}\cdot (1 \ten K) \cdot R \big)
&\stackrel{\eqref{eq:R-coprod}}{=} (\id\ten\id\ten\psi)(R_{32}\cdot R_{31})\cdot(1\ten1\ten K)\cdot R_{13}\cdot R_{23}\\
&\stackrel{\phantom{\eqref{eq:2K-coprod-1}}}{=} R^\psi_{32}\cdot R^\psi_{31}\cdot(1\ten1\ten K)\cdot R_{13}\cdot R_{23}.
\end{align}

For the second coproduct identity \eqref{eq:2K-coprod-2} we first recall that $\Delta_J=\Delta^{\op,\psi}$ and use it to rewrite \eqref{eq:R-coprod} in the following way:
\begin{align}
(\id\ten\Delta)(R^\psi_{21}) 
= J^{-1}_{23}\cdot (\id\ten\Delta^{\op,\psi})(R^\psi_{21}) \cdot J_{23} 
= J^{-1}_{23}\cdot R^\psi_{31}\cdot R^\psi_{21} \cdot J_{23}.
\end{align}
Hence, we obtain
\begin{align}
(\id\ten\Delta)&\big(R^\psi_{21}\cdot (1 \ten K) \cdot R \big)=\\
&\stackrel{\eqref{eq:R-coprod}}{=} J^{-1}_{23}\cdot R^\psi_{31}\cdot R^\psi_{21} \cdot J_{23} \cdot (1 \ten \Delta(K)) \cdot R_{13}\cdot R_{12} \\
&\stackrel{\eqref{eq:K-coprod}}{=} J^{-1}_{23} \cdot R^\psi_{31}\cdot R^\psi_{21} \cdot (1 \ten 1 \ten K) \cdot R^\psi_{23} \cdot (1 \ten K \ten 1)\cdot R_{13}\cdot R_{12}\\
&\stackrel{\phantom{\eqref{eq:R-coprod}}}{=}  J^{-1}_{23} \cdot R^\psi_{31} \cdot (1 \ten 1 \ten K)\cdot R^\psi_{21} \cdot R^\psi_{23} \cdot R_{13}\cdot (1 \ten K \ten 1) \cdot R_{12}\\
&\stackrel{\eqref{eq:YBE}}{=}  J^{-1}_{23} \cdot R^\psi_{31} \cdot (1 \ten 1 \ten K)\cdot R_{13} \cdot R^\psi_{23} \cdot R^\psi_{21}\cdot (1 \ten K \ten 1) \cdot R_{12}.
\end{align}
We also note that, since $(\eps \ten \id)(R)=(\id \ten \eps)(R)=1$,
	\begin{align}
		(\veps\ten\id)(\TKM{})=(\veps\ten\id) \big(R^\psi_{21}\cdot (1 \ten K) \cdot R\big) = K.
	\end{align} 

Finally, if $\TKM{}= R^\psi_{21}\cdot (1 \ten K) \cdot R$ defines a reflection structure $(\psi, J, \TKM{})$ on $(A,B)$ then \eqref{tensor-K-support} is just the defining condition $\TKM{} \in B \ten A$. 
\end{proof}


\section{Boundary bimodule categories} \label{ss:bdrybimodcat}

In this section, we introduce the notion of a \emph{boundary bimodule category}
tailored around the category of modules over a reflection bialgebra, see Section~\ref{ss:reflection-bialgebras}. This framework encodes in particular
the action of the tensor K-matrix of a quantum symmetric pair of arbitrary type. Thus, it is very similar, in spirit, with the approach used in \cite{Enr07, Bro13, Kol20}, and more recently in \cite{LYW23}.

As braided monoidal categories are equipped with a natural action of the braid group on the tensor powers of their objects,
boundary bimodule categories give rise to representations of a \emph{cylindrical braid groupoid} whose definition relies on Cherednik's generalized reflection equation (see Section~\ref{ss:modelBBC}).\\

\noindent
{\bf Conventions.} 
A boundary bimodule category is a category $\cM$ acted upon by a braided monoidal category $\cC$ and equipped with an extra structure. 
In the cases of our interest, the monoidal structure on $\cC$ is always given by the ordinary tensor product of vector spaces. 
Hence the associativity and the unit constraints will be omitted for simplicity. 
In contrast, the monoidal action of $\cC$ on $\cM$, while being canonically unital, has a non--canonical associativity constraint, which will therefore not be omitted.

\subsection{Module categories}\label{ss:module-category}

We briefly review the notion of a module category $\cM$ over a (braided) monoidal category $\cC$ as presented in \cite{HO01, Kol20}. 
Let $\mathcal{C}$ be a monoidal category with tensor product $\ten$ and unit $\mathbf{1}$.

A \emph{right monoidal action} of $\mathcal{C}$ on a category $\mathcal{M}$ is a functor $\monactp\colon \mathcal{M} \times \mathcal{C} \to \mathcal{M}$ together with a natural isomorphism, called \emph{associativity constraint},
\[
\Phi_{MVW}: M\monactp(V\ten W)\to (M\monactp V)\monactp W.
\]
The naturality consists in the following properties:
\vspace{1mm}
\begin{itemize} \itemsep2mm

\item 
for any $M\in\cM$, $M\monactp{\bf 1}=M$;

\item 
for any $M\in\cM$, $V\in\cC$, $\Phi_{M{\bf1}V}=\id_{M\monactp V}=\Phi_{MV{\bf 1}}$;

\item 
for any $M\in\cM$ and $U,V,W\in\cC$,
\begin{equation}
\begin{tikzcd}
M\monactp ((U\ten V)\ten W) \ar[rr, equal] \ar[d, "\Phi_{M,U\ten V,W}"'] && M\monactp (U\ten (V\ten W))\ar[d, "\Phi_{M,U, V\ten W}"] \\[1em]
(M\monactp (U\ten V))\monactp W) \ar[dr, "\Phi_{MUV}\monactp\id_W"'] & \circlearrowleft & (M\monactp U)\monactp (V\ten W)) \ar[dl, "\Phi_{M\monactp U,V,W}"] \\[1em]
&  ((M\monactp U)\monactp V))\monactp W) &
\end{tikzcd}
\end{equation}

\end{itemize}

Let $\cM, \cN$ be module categories over $\cC$ and $F:\cM\to\cN$ a functor.\footnote{By abuse of notation, we use the symbols $\monactp$ and $\Phi$ to denote, respectively, the action and the associativity constraint on both $\cM$ and $\cN$.}  A {\em module structure} on $F$ is a natural isomorphism 
\begin{equation}\label{eq:mod-ten-structure}
	{\bf J}_{MV}\colon F(M)\monactp V\to  F(M\monactp V),
\end{equation}
such that

\begin{itemize}\itemsep2mm

\item 
for any $M\in\cM$, ${\bf J}_{M{\bf1}}=\id_{F(M)}$;

\item 
for any $M\in\cM$ and $V,W\in\cC$, 
\begin{equation}\label{eq:mod-cocycle}
\begin{tikzcd}
(F(M)\monactp V)\monactp W \ar[rr, "{\bf J}_{MV}\monactp\id_W"] & & F(M\monactp V)\monactp W\ar[d, "{\bf J}_{M\monactp V, W}"] \\[1em]
F(M)\monactp(V\ten W) \ar[dr, "{\bf J}_{M,V\ten W}"'] \ar[u, "\Phi_{F(M)VW}"] &\circlearrowleft & F((M\monactp V)\monactp W)  \\[1em]
&  F(M\monactp(V\ten W)) \ar[ur, "F(\Phi_{MVW})"'] &
\end{tikzcd}
\end{equation}

\end{itemize}

Henceforth, we refer to the datum $(F,{\bf J})$ as a {\em module functor} $\cM\to\cN$.

\begin{example}\label{ex:mon-act}
	Set $\mathcal{C}=\Mod{(\Uqg)}$ and $\mathcal{M}=\Mod{(\Uqk)}$.
	Since $\Uqk$ is a coideal subalgebra in $\Uqg$, the usual tensor
	product induces a functor 
	\[
	\monactp\colon\Mod{(\Uqk)}\times\Mod{(\Uqg)}\to\Mod{(\Uqk)}
	\]
	which is readily verified to be a monoidal action of $\Mod{(\Uqg)}$
	on $\Mod{(\Uqk)}$.
	\rmkend
\end{example}

\subsection{Braided module categories}\label{ss:braided-module}

Let $\cC$ be a braided monoidal category with braiding ${\bf c} \colon \ten\to\ten^{\scsop{op}}$ and $\cM$ a right module category over $\cC$. 
For any $M\in\cM$ and $V,W\in\cC$, set 
\[
(\id_M\monactp{\bf c}_{VW})_{\Phi}=\Phi_{MWV}\circ (\id_M\monactp{\bf c}_{VW})\circ\Phi_{MVW}^{-1}
\]
Following \cite{Bro13, BZBJ18, Kol20},
a {\em module braiding} on $\cM$ is a natural automorphism  
\[
{\bf d}_{MV}\colon M\monactp_{} V\to M\monactp_{} V,
\]
such that

\begin{itemize}\itemsep2mm

\item 
for any $M\in\cM$, ${\bf d}_{M{\bf1}}=\id_{M}$;

\item 
for any $M\in\cM$, $V,W\in\cC$,
\begin{equation}\label{eq:braided-coprod-L}
\begin{tikzcd}
(M\monactp V)\monactp W\ar[dr,phantom, "\circlearrowleft"] \ar[r, "{\bf d}_{M\monactp V, W}"] \ar[d, "(\id_M\monactp {\bf c}_{VW})_{\Phi}"']
&[25pt] (M\monactp V)\monactp W\\
(M\monactp W)\monactp V \ar[r, "{\bf d}_{MW}\monactp\id_V"']
&(M\monactp W)\monactp V \ar[u, "(\id_M\monactp {\bf c}_{WV})_{\Phi}"']
\end{tikzcd}
\end{equation}

\item 
for any $M\in\cM$, $V,W\in\cC$,
\begin{equation}\label{eq:braided-coprod-R}
\begin{tikzcd}
M\monactp (V\ten W)\arrow[rr, "{\bf d}_{M, V\ten W}"] \arrow[d,"\Phi_{MVW}"']
&[20pt] &[20pt] M\monactp (V\ten W)\arrow[d, "\Phi_{MVW}"'] \\
(M\monactp V)\monactp W \arrow[d,"{\bf d}_{MV}\monactp\id_W"']&\circlearrowleft& (M\monactp V)\monactp W\\
(M\monactp V)\monactp W \arrow[r, "(\id_M\monactp {\bf c}_{VW})_{\Phi}"'] & (M\monactp W)\monactp V \arrow[r,"{\bf d}_{MW}\monactp\id_V"']&  (M\monactp W)\monactp V  \arrow[u, "(\id_M\monactp {\bf c}_{WV})_{\Phi}"']
\end{tikzcd}
\end{equation}	

\end{itemize}
\vspace{2mm}

A module functor $F:\cM\to \cN$ is {\em braided} if the module structure on $F$ intertwines the braidings, \ie if the following diagram commutes for any $M\in\cM$ and $V\in\cC$:
\begin{equation}
\begin{tikzcd}
F(M)\monactp V\arrow[r, "{\bf d}_{F(M),V}"]\arrow[d, "{\bf J}_{MV}"'] &[25pt] F(M)\monactp V\arrow[d, "{\bf J}_{MV}"]\\
F(M\monactp V)\arrow[r, "F({\bf d}_{MV})"']  & F(M\monactp V)
\end{tikzcd}
\end{equation}

\begin{remarks}\hfill
\begin{enumerate}\itemsep2mm
\item
The identity $\id_{M\monactp V}$ does not define a module braiding on $\cM$ unless the monoidal category $\cC$ is symmetric, \ie ${\bf c}_{WV}={\bf c}_{VW}^{-1}$.

\item 
It is pointed out in \cite[Rmk.~3.6]{BZBJ18} that there is an infinite family of possible axioms for a braided module category. 
The relations \eqref{eq:braided-coprod-L} and \eqref{eq:braided-coprod-R} constitute just one example of such axioms.

\item
In \cite{Kol20}, Kolb proved that the action of the tensor K-matrix of a quantum symmetric pair $\Uqk\subset\Uqg$ with $\dim\g<\infty$ gives rise to a braided module structure on $\Modfd(\Uqk)$ over $\Modfd(\Uqg)$. 
This is no longer true when $\dim\g=\infty$, which motivates the notion of a boundary bimodule category, which we introduce in Sections \ref{ss:bimodule-cat}-\ref{ss:boundary-cat}. 
\hfill\rmkend 
\end{enumerate}

\end{remarks}

\subsection{Bimodule categories}\label{ss:bimodule-cat}

Let $\cC$ be a monoidal category. 
A \emph{left monoidal action} of $\cC$ on a category $\cM$ is a right monoidal action of $\cC^{\op}$ on $\cM$.
A \emph{bimodule category} over $\cC$ is a category $\cM$ equipped with 

\begin{itemize}\itemsep2mm
\item 
a right monoidal action $\monactp\colon\cM\times\cC\to\cM$ with associativity constraint $\Phi^{\monactp}$;

\item 
a left monoidal action $\monactm\colon\cM\times\cC\to\cM$ with associativity constraint $\Phi^{\monactm}$;

\item
a natural isomorphism
\begin{equation}\label{eq:action-commutativity}
{\bf e}_{MVW}\colon (M\monactm V)\monactp W\to(M\monactp W)\monactm V
\end{equation}
for any $M\in\cM$ and $V,W\in\cC$.
\end{itemize}

The naturality of ${\bf e}$ consists in the following statements:

\begin{itemize}\itemsep2mm
\item 
for any $M\in\cM$, $V\in\cC$, ${\bf e}_{MV{\bf 1}}=\id_{M\monactm V}$ and ${\bf e}_{M{\bf 1}W}=\id_{M\monactp W}$;

\item 
for any $M\in\cM$, $U,V,W\in\cC$,
\begin{equation}\label{eq:bimod-hex-1}
\adjustbox{scale=0.95, center}{
\begin{tikzcd}
(M\monactm (U\ten V))\monactp W) \ar[rr, "{\bf e}_{M,U\ten V, W}"] \ar[d, "\Phi^{\monactm}_{MUV}\monactp\id_{W}"']
& & (M\monactp W)\monactm (U\ten V)\ar[d, "\Phi^{\monactm}_{M\monactp W, U\ten V}"]\\[1em]
((M\monactm V)\monactm U)\monactp W \ar[dr, "{\bf e}_{M\monactm V, U,W}"'] &\circlearrowleft & ((M\monactp W)\monactm V)\monactm U \\[1em]
&  ((M\monactm V)\monactp W)\monactm U) \ar[ur, "{\bf e}_{MVW}\monactm\id_U"']&
\end{tikzcd}
}
\end{equation}

\item 
for any $M\in\cM$, $U,V,W\in\cC$,
\begin{equation}\label{eq:bimod-hex-2}
\adjustbox{scale=0.95, center}{
\begin{tikzcd}
(M\monactm U)\monactp (V\ten W)) \ar[rr, "{\bf e}_{M,U\ten V, W}"] \ar[d, "\Phi^{\monactp}_{M\monactm U,V,W}\monactp\id_{W}"']
& & (M\monactp (V\ten W))\monactm U\ar[d, "\Phi^{\monactp}_{MVW}\monactm\id_U"]\\[1em]
((M\monactm U)\monactp V)\monactp W \ar[dr, "{\bf e}_{MUV}\monactp\id_W"'] &\circlearrowleft & ((M\monactp V)\monactp W)\monactm U  \\[1em]
&  ((M\monactp V)\monactm U)\monactp W) \ar[ur, "{\bf e}_{M\monactp V,U,W}"']&
\end{tikzcd}
}
\end{equation}
\end{itemize}

\begin{remark}\label{rmk:bimod-braid}
Let $\cC$ be a braided monoidal category with braiding ${\bf c} \colon \ten\to\ten^{\scsop{op}}$. 
Then, $\cC$ is naturally a bimodule over itself. 
Namely, the right action $\monactp$ is given by the tensor product and has a trivial associativity constraint, \ie $\Phi^{\monactp}_{MVW}=\id_{M\ten V\ten W}$. 
The left action $\monactm$ is also given by the tensor product, but its associativity constraint is given by the braiding, \ie $\Phi^{\monactm}_{MVW}=\id_M\ten{\bf c}_{VW}$. 
Finally, the commutativity constraint \eqref{eq:action-commutativity} is also given by the braiding, \ie  ${\bf e}_{MVW}=\id_M\ten{\bf c}_{VW}$. 
It is easy to check that the naturality of ${\bf e}$ in this case follows from the hexagon axioms and the naturality of ${\bf c}$.
\hfill\rmkend
\end{remark}

Let $\cM, \cN$ be bimodule categories and $F:\cM\to\cN$ a functor.\footnote{By abuse of notation, we use the same symbol to denote the monoidal actions on $\cM$ and $\cN$.} 
A {\em bimodule structure} on $F$ is the datum of a pair of natural isomorphisms 
\begin{equation}\label{eq:bimod-ten-structure}
{\bf J}^{\monactp}_{MV}\colon F(M)\monactp V\to  F(M\monactp V),
\qq \qq 
{\bf J}^{\monactm}_{MV}\colon F(M)\monactm V\to  F(M\monactm V),
\end{equation}
such that

\begin{itemize}\itemsep2mm

\item 
for any $M\in\cM$, $V,W\in\cC$,
\begin{equation}\label{eq:bimod-coprod-1}
\adjustbox{scale=0.8,center}{
\begin{tikzcd}
F(M)\monactp (V\ten W) \ar[drr, phantom,"\circlearrowleft"] \ar[rr,"{\bf J}^{\monactp}_{M,V\ten W}"]\ar[d, "\Phi^{\monactp}_{F(M)VW}"'] 
&[20pt] &[20pt] F(M\monactp(V\ten W))\ar[d, "F(\Phi^{\monactp}_{MVW})"]\\
(F(M)\monactp V)\monactp W \ar[r, "{\bf J}^{\monactp}_{MV}\monactp \id_{W}"'] & F(M\monactp V)\monactp W \ar[r, "{\bf J}^{\monactp}_{M\monactp V,W}"'] &
F((M\monactp V)\monactp W)
\end{tikzcd}
}
\end{equation}

\item 
for any $M\in\cM$, $V,W\in\cC$,
\begin{equation}\label{eq:bimod-coprod-2}
\adjustbox{scale=0.8,center}{
\begin{tikzcd}
F(M)\monactm (V\ten W) \ar[drr, phantom,"\circlearrowleft"] \ar[rr,"{\bf J}^{\monactm}_{M,V\ten W}"]\ar[d, "\Phi^{\monactm}_{F(M)VW}"'] 
&[20pt] &[20pt] F(M\monactm(V\ten W))\ar[d, "F(\Phi^{\monactm}_{MVW})"]\\
(F(M)\monactm W)\monactm V \ar[r, "{\bf J}_{MW}\monactm \id_{V}"'] &
F(M\monactm W)\monactm V \ar[r, "{\bf J}_{M\monactm W,V}"']&
F((M\monactm W)\monactm V)
\end{tikzcd}
}
\end{equation}

\item 
for any $M\in\cM$ and $V,W\in\cC$,
\begin{equation}\label{eq:boundary-cocycle}
\begin{tikzcd}
(F(M)\monactm V)\monactp W \ar[ddr, phantom,"\circlearrowleft"] \arrow[r, "{\bf e}_{F(M),V,W}"]\arrow[d, "{\bf J}^{\monactm}_{MV}\monactp \id_W"'] 
&[25pt] F(M)\monactp W)\monactm V\arrow[d, "{\bf J}^{\monactp}_{MV}\monactm \id_V"]\\[15pt]
F(M\monactm V)\monactp W\arrow[d, "{\bf J}^{\monactp}_{M\monactm V,W}"']  
& F(M\monactp W)\monactm V \arrow[d, "{\bf J}^{\monactm}_{M\monactp W,V}"]\\[15pt]
F((M\monactm V)\monactp W)\arrow[r, "F({\bf e}_{MVW})"'] &F((M\monactp W)\monactm V)
\end{tikzcd}
\end{equation}

\end{itemize}

\subsection{Boundary bimodule categories}\label{ss:boundary-cat}

Let $\cC$ be a braided monoidal category with braiding ${\bf c} \colon \ten\to\ten^{\scsop{op}}$ and $\cM$ a bimodule category.
A  \emph{boundary structure on $\cM$} is a natural isomorphism 
\begin{equation}\label{eq:boundary}
{\bf f}_{MV}\colon M\monactp V\to M\monactm V
\end{equation}
such that

\begin{itemize}\itemsep2mm
\item 
for any $M\in\cM$ and $V,W\in\cC$,
\begin{equation}\label{eq:boundary-left}
\begin{tikzcd}
& (M\monactp V)\monactp W \ar[ddd,phantom,"\circlearrowleft"]\arrow[dr, "{\bf f}_{M\monactp V, W}"] & \\[15pt]
M\monactp (V\ten W) \arrow[d,"\id_M\monactp{\bf c}_{VW}"'] \arrow[ur, "\Phi^{\monactp}_{MVW}"] & & (M\monactp V)\monactm W \\[15pt]
M\monactp (W\ten V) \arrow[dr, "\Phi^{\monactp}_{MWV}"'] & & (M\monactm W)\monactp V\arrow[u,"{\bf e}_{MWV}"'] \\[15pt]
& (M\monactp W)\monactp V \arrow[ur, "{\bf f}_{MV}\monactp\id_V"'] &
\end{tikzcd}
\end{equation}
\item 
for any $M\in\cM$ and $V,W\in\cC$,
\begin{equation}\label{eq:boundary-right}
\begin{tikzcd}
& M\monactm (V\ten W)\ar[ddd,phantom,"\circlearrowleft"] \arrow[dr, "\Phi^{\monactm}_{MVW}"] & \\[15pt]
M\monactp (V\ten W) \arrow[d,"\Phi^{\monactp}_{MVW}"']  \arrow[ur, "{\bf f}_{M, V\ten W}"] & & (M\monactm W)\monactm V \\[15pt]
(M\monactp V)\monactp W \arrow[dr, "{\bf f}_{MV}\monactp\id_W"'] & & (M\monactp W)\monactm V\arrow[u,"{\bf f}_{MV}\monactm\id_V"'] \\[15pt]
& (M\monactm V)\monactp W \arrow[ur, "{\bf e}_{MVW}"'] &
\end{tikzcd}
\end{equation}

\end{itemize}

\begin{remark}
Consider $\cC$ as a module category over itself. 
Then, every braided structure on $\cC$ (see Section \ref{ss:braided-module}) is also a boundary with respect the bimodule structure described in  Remark~\ref{rmk:bimod-braid}.
\hfill\rmkend
\end{remark}

Let $\cM,\cN$ be boundary bimodule categories and $F\colon\cM\to\cN$ a bimodule functor. 
Then, $F$ is {\em boundary} if the bimodule structures intertwines the boundary structures, \ie
\begin{equation}
\begin{tikzcd}
F(M)\monactp V\arrow[r, "{\bf f}_{F(M),V}"]\arrow[d, "{\bf J}^{\monactp}_{MV}"'] &[25pt] F(M)\monactm V\arrow[d, "{\bf J}^{\monactm}_{MV}"]\\
F(M\monactp V)\arrow[r, "F({\bf f}_{MV})"']  & F(M\monactm V)
\end{tikzcd}
\end{equation}
for any $M\in\cM$ and $V\in\cC$.


\subsection{From reflection bialgebras to boundary structures} 
The notion of a boundary structure on a bimodule category is tailored to encode the defining datum of a reflection bialgebra (see Section~\ref{ss:reflection-bialgebras}). 
We make this precise in the following result. 
For the more refined context of quantum symmetric pairs, see Section \ref{ss:boundarybimodulestructure}. 

\begin{prop}\label{prop:reflection-to-boundary}

Let $(A, B, \psi, J, \bbK)$ be a reflection bialgebra.
The following datum defines a boundary bimodule structure on $\Mod(B)$ over $\Mod(A)$.

\vspace{1mm}
\begin{enumerate}\itemsep2mm
\item {\bf Module structure.} 
There is a right monoidal action 
\[ 
\monactp\colon\Mod(B)\times\Mod(A)\to\Mod(B) 
\]
with trivial associativity constraint, given by the standard tensor product, \ie $M\monactp V=M\ten V$ .

\item {\bf Bimodule structure.} 
The left monoidal action 
\[
\monactm:\Mod(B)\times\Mod(A)\to\Mod(B)
\]
is obtained by twisting $\monactp$ with $\psi$, \ie $M\monactp V=M\ten{\Vpsi}$. 
The associativity constraint
\[
\Phi^{\monactm}_{MVW}\colon M\ten\VWpsi\to M\ten {\Wpsi}\ten {\Vpsi}
\]
is provided by the action of $J$. 
The commutativity constraint
\[
{\bf e}_{MVW}\colon M\ten {\Vpsi}\ten W\to M\ten W\ten {\Vpsi}
\]
is provided by the action of the R-matrix $R$ on the second and third factor.

\item {\bf Boundary structure.} 
The boundary structure
\[
{\bf f}_{MV}\colon M\ten V\to M\ten {\Vpsi}
\]
is provided by the action of the tensor K-matrix $\bbK$.

\end{enumerate}

\end{prop}


\subsection{The model of a boundary bimodule category} \label{ss:modelBBC}
In a braided monoidal category, every tensor power $V^{\ten n}$ is naturally acted upon by the ordinary braid group (Artin-Tits group of type $\sfA_{n-1}$). 
Similarly, in a braided module category, every object of the form $W\monactp V^{\ten n}$ is naturally acted upon by the cylindrical braid group (Artin-Tits group of type $\sfB_n$). 
In the case of boundary bimodule categories, the symmetry is encoded instead by a {\em cylindrical braid groupoid} $\sfM$, whose generators satisfy a constant version of Cherednik's generalized reflection equation \cite[Sec.\ 4]{Che92}. 
Its definition can be thought of as a boundary/cylindrical analogue of the category \textbf{braid} from \cite{JS86} or, equivalently, of the free PROB considered in \cite{KS20}.\\

\noindent
Let $\sfC$ be the strict monoidal category defined as follows.
The objects of $\sfC$ are labelled by non--negative integers, \ie for any $n\in\bbZ_{\geqslant0}$, there is an object $[n]\in\sfC$.
The tensor product in $\sfC$ is given by $[n]\ten[m]=[n+m]$ for any $n,m\in\bbZ_{\geqslant 0}$. Hence, $[0]$ is the unit.
The morphisms in $\sfC$ are generated by an invertible element $\Ct\in\End_{\sfC}([2])$ satisfying the braid relation
\begin{equation}\label{eq:prop-ch-1}
\Ct_{12}\Ct_{23}\Ct_{12} = \Ct_{23}\Ct_{12}\Ct_{23}
\end{equation}
 in $\End_{\sfC}([3])$.

\begin{remark}\label{rmk:prop-1}
Let $V$ be an object in a braided monoidal category $\cC$ with braiding $\mathbf{c}$. 
There is an essentially unique\footnote{
The functor $\cG$ depends upon the choice of a bracketing on $n$ elements for every $n$, see, \eg \cite[Sec.~7]{ATL19}.
} 
braided monoidal functor $\cG_V:\sfC\to\cC$ such that $\cG_V([n])=V^{\ten n}$ and $\cG_V(\Ct)=\mathbf{c}_{VV}$. 
\hfill\rmkend
\end{remark}

\noindent
Roughly, the category $\sfM$ is the strict boundary category over $\sfC$ freely generated by one object $\star\in\sfM$. More precisely, the objects of $\sfM$ are given by $\{\star\}\times\coprod_{n\geqslant0}\{\pm\}^n$, while the morphisms are generated by the following elements.
\vspace{1mm} 
\begin{enumerate}\itemsep2mm
\item 
For any ${\sfs}\in\sfM$, three invertible elements
\begin{gather}
\Ct^{++}\in\Hom_\sfM(\sfs++\,,\,\sfs++)\\
\Ct^{--}\in\Hom_\sfM(\sfs--\,,\,\sfs--)\\
\Ct^{-+}\in\Hom_\sfM(\sfs-+\,,\,\sfs+-)
\end{gather}
subject to the mixed braid relations
\begin{equation}\label{eq:mixed-braid-relations}
\Ct^{\tau\zeta}_{12}\Ct^{\sigma\zeta}_{23}\Ct^{\sigma\tau}_{12} = \Ct^{\sigma\tau}_{23}\Ct^{\sigma\zeta}_{12}\Ct^{\tau\zeta}_{23}
\end{equation}
in $\Hom_\sfC(\sfs\,\sigma\,\tau\,\zeta,\sfs\,\zeta\,\tau\,\sigma)$, for all $\sigma,\tau,\zeta\in\sfA$ such that the corresponding operators are all defined.

\item 
An invertible element $\Cb_{\star}\in\Hom_\sfM(\star\,+\,,\,\star\,-)$
subject to {Cherednik's reflection equation} 
\begin{equation}\label{eq:RE-M}
\Ct^{--}\Cb_{\star}\Ct^{-+}\Cb_{\star}=\Cb_{\star}\Ct^{-+}\Cb_{\star}\Ct^{++}
\end{equation}
in $\Hom_\sfM(\star\,++\,,\,\star\,--)$.
\end{enumerate}
\vspace{0.2cm}
Then, $\sfM$ is naturally equipped with right and left actions of $\sfC$ with trivial associativity constraints. Specifically, the right and left actions of $[n]\in\sfC$ are given by concatenation with sequences of $+$ and $-$ of length $n$. 
This yields a bimodule structure on $\sfM$, where the commutativity constraint ${\bf e}$, see \eqref{eq:action-commutativity}, is given by $\Ct^{-+}$.

Finally, $\sfM$ is naturally equipped with a boundary structure $\Cb$, see \eqref{eq:boundary}, which is {induced} by $\Cb_{\star}$. Namely,
for any $\sfs\in\sfM$, we set
\begin{equation}\label{eq:boundary-M}
\Cb_{\sfs}=\Ct''\Cb_{\star}\Ct'
\end{equation}
in $\Hom_{\sfM}(\sfs+,\sfs-)$, where $\Ct'$ and $\Ct''$ are given by compositions of 
$\Ct^{++}$, $\Ct^{--}$, and $\Ct^{-+}$. More precisely, let $\sfs'\in\{\pm\}^n$ such that $\sfs=\star \sfs'$. Then, $\Ct'$ is the unique map in $\Hom_{\sfM}(\star \thinspace\sfs' +,\star\negthinspace + \sfs')$ obtained by iterations of $\Ct^{++}$ and $\Ct^{-+}$.
Similarly, $\Ct''$ is the unique map in $\Hom_{\sfM}(\star \negthinspace - \sfs',\star\thinspace \sfs' -)$ obtained by iterations of $\Ct^{--}$ and $\Ct^{-+}$.
Relying on \eqref{eq:RE-M}, one readily checks that $\Cb$ is a boundary structure on $\sfM$.

\begin{prop}
Let $\cC$ be a braided monoidal category, $V\in\cC$, and  $\cG_V:\sfC\to\cC$ the 
braided monoidal functor defined in Remark~\ref{rmk:prop-1}.
Let $\cM$ be a boundary bimodule category over $\cC$ with boundary structure ${\bf f}$.
Then, for any $M\in\cM$, there is an essentially unique boundary bimodule functor $\cF:\sfM\to\cG_V^*\cM$ such that $\cF(\star)=M$ and $\cF(\Cb_{\star})={\bf f}_{MV}$.
\end{prop}

\begin{remark}
In analogy with the category $\mathbf{braid}$ \cite{JS86}, $\sfM$ is the coproduct of cylindrical braid groupoids (of fixed length) with respect to the standard embeddings. The construction of the boundary structure \eqref{eq:boundary-M}, which encodes the coproduct identity \eqref{eq:2K-coprod-1}  satisfied by the tensor K-matrix, is similarly motivated by the natural embeddings of cylindrical braid groupoids obtained by replacing the initial object $\star$ with a longer sequence $\sfs$ starting with $\star$.\footnote{In the usual braid groups, this corresponds to gluing strands together.}  \hfill\rmkend
\end{remark}

%


\providecommand{\bysame}{\leavevmode\hbox to3em{\hrulefill}\thinspace}
\providecommand{\MR}{\relax\ifhmode\unskip\space\fi MR }
\providecommand{\MRhref}[2]{
	\href{http://www.ams.org/mathscinet-getitem?mr=#1}{#2}
}
\providecommand{\href}[2]{#2}

\end{document}